\documentclass[reqno,a4paper,10pt]{amsart} 

\usepackage[utf8]{inputenc}
\usepackage[english]{babel}
\usepackage{microtype}
\usepackage{fullpage}
\usepackage[dvipsnames]{xcolor}
\usepackage{tikz}
\usepackage{amsmath}
\usepackage{amssymb}
\usepackage{amsthm}
\usepackage{thmtools}
\usepackage[colorlinks = true, linkbordercolor = {white}, urlcolor={purple}, linkcolor={purple}, citecolor = {purple}]{hyperref}
\usepackage{cleveref}
\usepackage{amsfonts}
\usepackage{enumerate}
\usepackage[foot]{amsaddr}
\usepackage{mathtools}
\usepackage[normalem]{ulem}
\usepackage[font=footnotesize]{caption} 
\usepackage{pifont}
\usepackage[percent]{overpic}

\newcommand{\ceil}[1]{\left\lceil #1 \right\rceil}
\newcommand{\me}{\mathrm{e}}

\newcommand{\Li}{\operatorname{Li}}
\newcommand{\Ls}{\operatorname{Ls}}
\newcommand{\Lim}{\operatorname{Lim}}
\newcommand{\dmk}{d_{\operatorname{MK}}}
\newcommand{\Lip}{\operatorname{Lip}}
\newcommand{\tp}{.}
\newcommand{\tc}{,}
\newcommand{\bm}{\mathbf}

\usetikzlibrary{arrows}

\theoremstyle{definition}

\newtheorem{theorem}{Theorem}[section]
\newtheorem{definition}[theorem]{Definition}
\newtheorem{corollary}[theorem]{Corollary}
\newtheorem{prop}[theorem]{Proposition}
\newtheorem{lemma}[theorem]{Lemma}
\newtheorem{example}[theorem]{Example}
\newtheorem{question}[theorem]{Question}

\theoremstyle{remark}
\newtheorem{remark}[theorem]{Remark}

\renewcommand{\emph}{\textsl}
\renewcommand{\textit}{\textsl}

\renewcommand{\inf}{\mathop{\mathrm{inf}\vphantom{\mathrm{sup}}}}

\newcommand{\Z}{{\mathbb Z}}

\newcommand{\R}{{\mathbb R}}
\newcommand{\N}{{\mathbb N}}

\newcommand{\A}{{\mathcal A}}
\newcommand{\mc}{\mathcal}
\newcommand{\im}{{\mathrm{i}}}
\newcommand{\dd}{\,\mathrm{d}}

\newcommand{\Leb}{\operatorname{Leb}}
\newcommand{\Sub}{\Theta}

\newcommand{\Conv}{\mathop{\scalebox{1.5}{\raisebox{-0.2ex}{$\ast$}}}}

\numberwithin{equation}{section}

\makeatletter
\@namedef{subjclassname@2020}{
  \textup{2020} Mathematics Subject Classification}
\makeatother

\begin{document}

\allowdisplaybreaks

\title{Rauzy fractals of random substitutions}

\author{P.\,Gohlke$^{\,1}$, A.~Mitchell$^{\,2}$, D.~Rust$^{\,3}$ and T.\,Samuel$^{\,4}$}
 
% \author[P.\,Gohlke]{P.\,Gohlke}
% \author[A.\,Mitchell]{A.\,Mitchell}
% \author[D.\,Rust]{D.\,Rust}
% \author[T.\,Samuel]{T.\,Samuel}

% \address[P.\,Gohlke]{Faculty of Mathematics and Computer Science, Friedrich Schiller University, 07743 Jena, DE}
% \address[A.\,Mitchell]{Department of Mathematical Sciences, Loughborough University, LE11 3TU, UK}
% \address[D. \,Rust]{School of Mathematics and Statistics, The Open University, Milton Keynes, MK7 6AA, UK}
% \address[T.\,Samuel]{Department of Mathematics and Statistics, University of Exeter, Exeter, EX4 4QF, UK}



\setcounter{tocdepth}{1}

\maketitle

{\centering{\footnotesize
{$^{\,1}$ Faculty of Mathematics and Computer Science, Friedrich Schiller University, 07743 Jena, Germany\\
$^{\,2}$ Department of Mathematical Sciences, Loughborough University, Loughborough, LE11 3TU, UK\\ 
$^{\,3}$ School of Mathematics and Statistics, The Open University, Milton Keynes, MK7 6AA, UK\\
$^{\,4}$ Department of Mathematics and Statistics, University of Exeter, Exeter, EX4 4QF, UK\\}
}}

\begin{abstract}
We develop a theory of Rauzy fractals for random substitutions, which are a generalisation of deterministic substitutions where the substituted image of a letter is determined by a Markov process.
We show that a Rauzy fractal can be associated with a given random substitution in a canonical manner, under natural assumptions on the random substitution.
Further, we show the existence of a natural measure supported on the Rauzy fractal, which we call the Rauzy measure, that captures geometric and dynamical information. 
We provide several different constructions for the Rauzy fractal and Rauzy measure, which we show coincide, and ascertain various analytic, dynamical and geometric properties.
While the Rauzy fractal is independent of the choice of (non-degenerate) probabilities assigned to a given random substitution, the Rauzy measure captures the explicit choice of probabilities.
Moreover, Rauzy measures vary continuously with the choice of probabilities, thus provide a natural means of interpolating between Rauzy fractals of deterministic substitutions. 
Additionally, we highlight connections between Rauzy fractals and Rauzy measures of random substitutions and related S-adic systems.
\end{abstract}

\section{Introduction}

Rauzy fractals are geometric objects that can be associated with sequences arising from substitutions.
They first appeared in 1982 in the work of Rauzy \cite{Rauzy_1982}, who constructed a domain exchange of a compact subset of $\R^2$ that reflects the action of the \emph{tribonacci substitution}.
Questions regarding diffraction spectra and dynamical properties of substitution sequences can often be reframed in terms of the geometry and topology of the associated Rauzy fractal \cite{arnoux-ito}.
Perhaps most notably, the Pisot substitution conjecture can be reformulated in terms of tiling properties of Rauzy fractals---see for instance \cite{Akiyama_2015, Siegel-Thuswaldner}. 
Similar criteria in terms of Rauzy fractals have been established in the context of \emph{S-adic systems}---a generalisation of substitution sequences that involves the action of several substitutions \cite{BMST16,BST19}.
We provide a general framework for Rauzy fractals and establish a theory of Rauzy fractals for \emph{random substitution systems}, which are positive entropy systems arising from a generalisation of substitutions where the substituted image is determined by a Markov process.

\subsection{Substitutions}\label{SS:subst-intro}

Sequences arising from (deterministic) substitutions are the prototypical examples of mathematical quasicrystals.
A substitution is a rule that replaces each symbol from a finite alphabet with a concatenation of symbols from the same alphabet.
For example, the \emph{tribonacci substitution} $\theta_t$ is defined over the three-letter alphabet $\mc A = \{1,2,3\}$ by the rule
\begin{align}\label{EQ:trib-intro}
    \theta_t \colon 
    \begin{cases}
    1 \mapsto 12 ,\\
    2 \mapsto 13 ,\\
    3 \mapsto 1 ,
    \end{cases}
\end{align}
and the \emph{twisted tribonacci substitution} $\tilde{\theta}_t$ is defined over the same alphabet by the rule
\begin{align}\label{EQ:twisted-trib-intro}
    \tilde{\theta}_t \colon 
    \begin{cases}
    1 \mapsto 21 ,\\
    2 \mapsto 13 ,\\
    3 \mapsto 1 .
    \end{cases}
\end{align}
Note that the only difference between the tribonacci substitution and the twisted version is that the order of letters in the image of $a$ has been reversed.
The action of a substitution extends naturally to finite words and infinite sequences, by applying the substitution to each letter in turn and concatenating the result in the order prescribed by the initial word or sequence.

To a given substitution, a subshift can be associated in a natural way.
Questions concerning characterisation of diffraction and dynamical spectra, order versus disorder, and the topological and dynamical features of substitution subshifts are central questions in the field of aperiodic order, which have been extensively investigated;
see \cite{Akiyama_Arnoux_2017,baake-grimm,Siegel-Thuswaldner,Solomyak_1997} and the references therein.

To a given substitution over an alphabet $\mc A = \{1,\ldots,d\}$, we associate an $d \times d$-matrix $M_{\theta}$, called the \emph{substitution matrix of $\theta$}, defined by $(M_{\theta})_{i,j} = \lvert \theta (j) \rvert_{i}$.
The matrix $M_\theta$ encodes how the number of occurrences of a given letter in a word changes under the substitution action.
If $M_{\theta}$ is a primitive matrix, then we say that the substitution $\theta$ is \emph{primitive}. 
In this case, the Perron--Frobenius theorem gives that $M_{\theta}$ has a unique largest (real) eigenvalue $\lambda \geqslant 1$ and corresponding left and right eigenvectors with positive entries, denoted by $\mathbf{L}$ and $\mathbf{R}$ respectively.
We take $\mathbf{L}$ and $\mathbf{R}$ to be normalised such that the entries of $\mathbf{R}$ sum to $1$ and $\mathbf{L}\cdot \mathbf{R} = 1$, and refer to the triple $(\lambda, \mathbf{L}, \mathbf{R})$ as the \emph{Perron-–Frobenius data} of $\theta$.

If $\lambda$ is a Pisot number, that is, $\lambda > 1$ and all of its Galois conjugates have modulus strictly less than $1$, then we call $\theta$ a \emph{Pisot} substitution. 
If, in addition, the characteristic polynomial of $M_{\theta}$ is irreducible over $\Z$, then we say that $\theta$ is \emph{irreducible Pisot}. 
If $|\det (M_\theta)|=1$, then we say that $\theta$ is \emph{unimodular}. 

Observe that both the tribonacci and twisted tribonacci substitutions $\theta_t$ and $\tilde{\theta}_t$ defined in \eqref{EQ:trib-intro} and \eqref{EQ:twisted-trib-intro} have the same substitution matrix, given by
\begin{align*}
    M = M_{\theta_t} = M_{\tilde{\theta}_t} =
    \begin{pmatrix}
        1 & 1 & 1\\
        1 & 0 & 0\\
        0 & 1 & 0 
    \end{pmatrix} ,
\end{align*}
which is primitive since all entries of $M^3$ are positive. 
The Perron--Frobenius eigenvalue of $M$ is the tribonacci constant $t \approx 1.83929$; in particular, $t$ is the unique real number satisfying $t^3 = t^2 + t + 1$.
Since $t$ is a degree-three Pisot number and the matrix $M$ has determinant one, it follows that $\theta_t$ and $\tilde{\theta}_t$ are both unimodular, irreducible Pisot substitutions.

\subsection{Rauzy fractals of substitutions}\label{SS:Rauzy-intro}

To an irreducible Pisot substitution  $\theta$, one can associate a compact set $\mathcal{R}_{\theta}$ called a \emph{Rauzy fractal}.
Many questions relating to Pisot substitutions can be reformulated in terms of the geometry of the associated Rauzy fractal \cite{Siegel-Thuswaldner}.

Let $(\lambda,\mathbf{L},\mathbf{R})$ be the Perron--Frobenius data of an irreducible Pisot substitution $\theta$. By irreducibility, the entries of $\mathbf L$ and $\mathbf R$ are rational functions in $\lambda$.
Let $\mathbb{H}$ denote the contracting hyperplane of $M_{\theta}$, namely, the $(d -1)$-dimensional subspace of $\mathbb{R}^{d}$ orthogonal to $\mathbf{L}$.
Let $h_{\theta}$ denote the action of $M_{\theta}$ restricted to $\mathbb{H}$. Since $\theta$ is irreducible Pisot, every eigenvalue of $M_{\theta}$ other than the Perron--Frobenius eigenvalue, has absolute value strictly less than $1$. 
It will sometimes be necessary to define a metric on the hyperplane $\mathbb{H}$. For our purposes, it will be convenient to choose a metric for which the action of $h_\theta$ is a contraction.
Such a metric always exists as a consequence of the matrix $M_\theta$ being diagonalisable, which follows from the fact that $\theta$ is irreducible Pisot.  Note that diagonalisation is not necessary, but sufficient, for such a metric to exist, see for instance \cite{Lind_1982}.

Every primitive substitution has a power that admits a substitution-fixed point, 
that is, a right-infinite sequence $x$ over the alphabet $\A$ such that $\theta^k (x) = x$ for some $k \geqslant 1$; see \cite{baake-grimm,queffelec} for further details. 
If $\theta$ is Pisot, then this sequence has the additional property of being \emph{$C$-balanced} \cite{Adamcewski_03}, which implies 
the abelianisation vector of any finite subword of $x$ lies within a uniformly bounded distance of the ray spanned by $\mathbf{R}$. 
In particular, this vector is $n$-times the right Perron--Frobenius eigenvector of the substitution matrix.

If $x$ denotes a substitution-fixed point of a Pisot substitution $\theta$, we let $\mc S(x)$ denote the subset of $\Z^d$ consisting of the abelianisation vectors of the finite words obtained by truncating $x$ after $n$ letters, for all positive integers $n$.
Here, by the abelianisation vector of a given word, we mean the vector whose entries are the number of occurrences of each letter in the word. 
For example, the abelianisation vector of the word $1213121$ is $(4,2,1)^{T}$.
We highlight that the set $\mc S (x)$ comprises the vertices in the \emph{broken line} of the sequence $x$---see \cite{Siegel-Thuswaldner} for more details.

Let $\pi$ denote the projection along $\mathbf{R}$ onto $\mathbb{H}$.
Since $x$ is $C$-balanced, the points in $\mc S (x)$ lie a bounded distance from the ray spanned by the Perron--Frobenius eigenvector $\bm R$.
Thus, it follows that the set $\mc R^{*} (x) = \pi (\mc S(x))$ is bounded. 
The \emph{Rauzy fractal} of $\theta$ is then defined by $\mc R(x) = \overline{\mc R^*(x)}$.
Every substitution-fixed point $x$ gives rise to the same set $\mc R (x)$---see \cite{Siegel-Thuswaldner} for more details.
As such, we write $\mc R_{\theta}$ for this common set.
The Rauzy fractals of the tribonacci and twisted tribonacci substitutions are plotted in \Cref{fig:rauzy-fractal-tribonacci-and-projection}, together with a visualisation of the construction.

\begin{figure}[t]
\begin{center}
\includegraphics[trim={3em 9.2em 3em 5.2em},clip,width=0.29\textwidth]{Images/Rauzy_n_22_p_0_16}
\hspace{0.5em}
\includegraphics[trim={3em 7.5em 1.5em 6.2em},clip,width=0.29\textwidth]{Images/Rauzy_n_22_p_16_16}
\hspace{-1.5em}
\hbox{\includegraphics[trim={33.5em 28em 20em 20em},clip,width=0.22\textwidth,angle=-35]{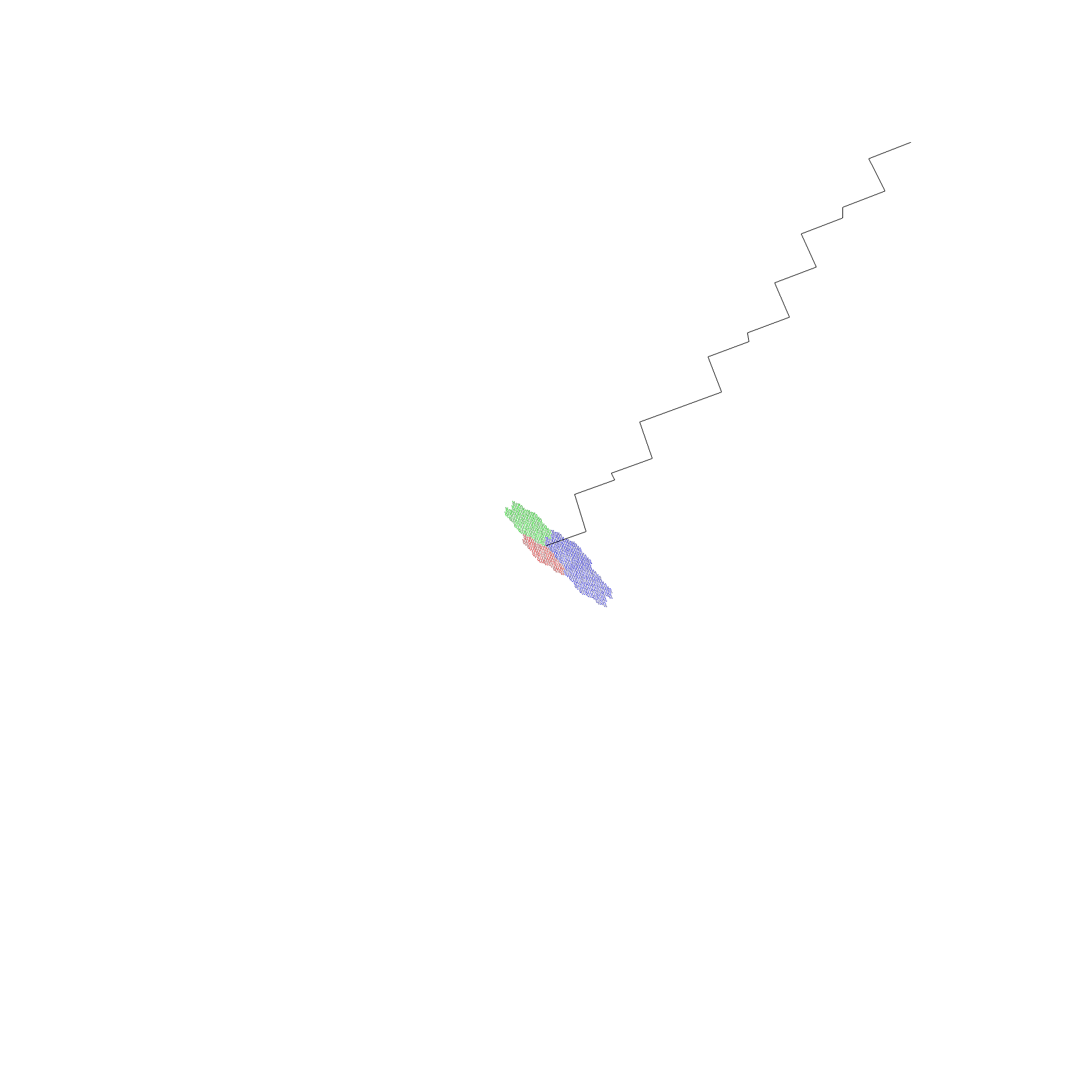}}
\hspace{-8em}
\begin{tikzpicture}[node distance=2.5cm, auto]
  				\node (A) {};
  				\node (B) [node distance=0.4cm, above of=A] {};
  				\node (C) [right of=B] {};
  				\node (D) [node distance=0.5cm, above of=C] {};
  				\draw[->] (D) to node [swap] {} (B);
\end{tikzpicture}
\vspace{-5em}
\caption{Rauzy fractals for the tribonacci (left) and twisted tribonacci (center) substitutions. Broken line and projection (right).}
\label{fig:rauzy-fractal-tribonacci-and-projection}
\end{center}
\end{figure}

Several fundamental topological properties of Rauzy fractals associated with substitutions are stated in the following. 

\begin{prop}[{\cite[Theorem~2.6]{Siegel-Thuswaldner}}]\label{PROP: deterministic-rauzy-properties}
Let $\theta$ be a unimodular, irreducible Pisot substitution over a $d$-letter alphabet, where $d \geq 2$.
The Rauzy fractal $\mc R_{\theta}$ is a compact subset of $\mathbb{H}$ with non-empty interior, and thus has full Hausdorff dimension.
Moreover, $\mc R_{\theta}$ is the closure of its interior. 
\end{prop}

The Rauzy fractal associated with a substitution over a $d$-letter alphabet can be decomposed in a natural way into $d$ regions called \emph{subtiles}, which are measure-disjoint.
These subtiles are related via a graph-directed iterated function system that arises naturally from the substitution action \cite{Siegel-Thuswaldner}.
For the tribonacci substitution, the subtiles $\mc R_1$, $\mc R_2$ and $\mc R_3$ are the unique non-empty compact sets satisfying 
\begin{align*}
    \mc R_1 = h (\mc R_1) \cup h (\mc R_2) \cup h (\mc R_3), \quad
    \mc R_2 = h (\mc R_1) + \bm v, \quad
    \mc R_3 = h (\mc R_2) + \bm v \tc
\end{align*}
where $h$ denotes the action of the substitution matrix on the contracting plane and $\bm v \in \mathbb{H}$. 
In \Cref{fig:rauzy-fractal-tribonacci-and-projection}, the subtiles are highlighted for the Rauzy fractals of the tribonacci and twisted tribonacci substitutions, where the blue subtile corresponds to the letter $1$, the green subtile corresponds to the letter $2$ and the red subtile corresponds to the letter $3$.

\subsection{Random substitutions}

Random substitutions are a generalisation of deterministic substitutions where the substituted image of a letter is chosen from a fixed finite set with respect to a probability distribution.
Given $p \in (0,1)$, a random analogue of the tribonacci substitution can be defined by
\begin{align*}
	\vartheta_{\mathbf P} \colon
		\begin{cases}
		1 \mapsto
			\begin{cases}
			12 & \text{with probability } p,\\
			21 & \text{with probability } 1-p,
		     \end{cases}\\
        2 \mapsto 13 \quad \text{with probability } 1,\\
        3 \mapsto 1\phantom{3} \quad \text{with probability } 1 \tp
		\end{cases}
	\end{align*}
The action of a random substitution extends naturally to finite words, by applying it \emph{independently} to each letter in turn and concatenating the result in the order prescribed by the initial word. Similarly to the deterministic setting, a subshift can be associated with a given random substitution in a natural way.
Moreover, a substitution matrix can be associated in a similar manner and the notion of primitivity extends naturally to random substitutions -- we give the precise definitions in \Cref{S:rand-subst-intro}.
However, in stark contrast to deterministic substitutions and S-adic systems,
subshifts associated with random substitutions typically have positive topological entropy.
In fact, in the primitive setting, a random substitution subshift has zero topological entropy if and only if it is the subshift of a deterministic substitution \cite{mitchell}.
Despite the presence of positive entropy, the corresponding diffraction measure can still admit a non-trivial pure point component \cite{baake-spindeler-strungaru}, indicating long-range correlations.
To a given primitive random substitution, an ergodic measure can be associated in a natural way, which we call the \emph{frequency measure} corresponding to the random substitution.
While the subshift of a random substitution is blind to the choice of (non-degenerate) probabilities, the frequency measure captures the explicit choice of probabilities.

The theory of random substitutions has developed significantly in recent years.
For example, dynamical and diffraction spectra have been studied for several classes of examples \cite{baake-spindeler-strungaru,godreche-luck,Moll_2014}; a systematic approach to topological and measure theoretic entropy has been developed \cite{gohlke,MT-entropy}; and sufficient conditions under which a random substitution subshift is topologically mixing have been ascertained \cite{escolano-manibo-miro,miro-et-al}.

In this paper, we develop a theory of Rauzy fractals for random substitutions.
Since subshifts of deterministic primitive substitutions are minimal, the Rauzy fractal obtained by projecting any sequence from the subshift is a translation of the Rauzy fractal associated with a substitution-fixed point.
However, subshifts of random substitutions are typically not minimal, and there is no direct analogue of a substitution-fixed point for random substitutions.
Consequently, the Rauzy fractals obtained by projecting sequences from random substitution subshifts are no longer uniform (up to translation). 
Nevertheless, there is a maximal Rauzy fractal, that emerges from every transitive point of the subshift and is therefore typical for every fully supported ergodic measure. 

Another difference to the deterministic setting is that projection of the set of vertices in the broken line can no longer be expected to be uniformly distributed on the Rauzy fractal. 
Rather, they typically follow a specific distribution that is determined by the probability parameters. 
We refer to this distribution as the \emph{Rauzy measure} and we show that the Rauzy measure is an almost-sure object with respect to several natural measures on the random substitution subshift. 
It can further be obtained by following a typical succession of randomly created inflation words, which makes it accessible to numerical approximation. 
A natural application of Rauzy measures is the explicit computation of diffraction measures, as has been demonstrated in \cite{baake-spindeler-strungaru} for the example of the random Fibonacci substitution.

\subsection{Outline and overview of main results}

We begin by presenting the preliminaries on random substitutions in \Cref{S:rand-subst-intro}.
In \Cref{S:Bal-seq-Rauzy}, we develop a general theory of Rauzy fractals for $C$-balanced sequences, which we utilise in the construction of the Rauzy fractal associated with a random substitution.
While we have developed this framework with applications to random substitutions in mind, we believe that the results will be of independent interest. Indeed Rauzy fractals for $C$-balanced sequences that do not arise naturally from substituions or S-adic sequences have recently appeared in the literature \cite{Berthe_etal_chairman,choi_etal_self_replicating}.
These include (semi-)continuity properties with respect to the underlying sequence and a dynamical interpretation via generic factors.
We also show that Rauzy measures satisfy the \emph{Lebesgue covering property}.
Namely, that repeating the Rauzy measure along an appropriately chosen lattice combine to a constant multiple of Lebesgue measure.

In \Cref{S: Rauzy random subs}, we return our attention to random substitutions.
Utilising the framework of Rauzy fractals for $C$-balanced sequences developed in \Cref{S:Bal-seq-Rauzy}, we provide two methods for constructing a canonical Rauzy fractal associated with an irreducible Pisot random substitution, which we prove coincide.
Using our construction, we show that Rauzy fractals of random substitutions are the closure of their interior and contain an open ball, thus have full Hausdorff dimension, analogously to the deterministic setting. 
Moreover, the subtiles of the
Rauzy fractal associated with a random substitution can be obtained as the attractors of a graph-directed iterated function system (GIFS).
However, we highlight that in contrast to the deterministic setting, this GIFS may not satisfy the open set condition.

\Cref{S: measures} concerns Rauzy measures associated with random substitutions.
While the Rauzy fractal is blind to the choice of (non-degenerate) probabilities assigned to a random substitution, the Rauzy measure captures the explicit choice of probabilities.
We show that this measure can be constructed both via the invariant distribution of the random substitution and via the Dirac masses associated with the projection of vertices in the broken line of a typical element of the subshift.
Further, we show that similarly to Rauzy fractals themselves, Rauzy measures are self-similar objects with respect to an appropriately chosen GIFS.
Harvesting this interpretation, we obtain that the Rauzy measure depends continuously on the probability parameters. 
As a consequence, the Rauzy measures of random substitutions provide a convenient tool to interpolate between two (or more) deterministic substitutions that share the same substitution matrix. 
Here, it is worth pointing out that every irreducible Pisot substitution (up to taking higher powers) shares its substitution matrix with a substitution that satisfies the Pisot substitution conjecture \cite{Barge_2016,BST23}.
Constructions for interpolating between Rauzy fractals via a parametrised family is not new \cite{berthe_shift_radix,BST19}, though we feel our construction is suitably novel to warrant further study.

It follows from the results in \Cref{S:Bal-seq-Rauzy} that Rauzy measures associated with random substitutions satisfy the Lebesgue covering property.
Consequently, the Rauzy measure is absolutely continuous with respect to Lebesgue measure.
As a byproduct of our proof, we demonstrate the absolute continuity of self-similar measures for a certain class of GIFS, which may be of independent interest.

In \Cref{S:S-adic}, we interpret our results in terms of S-adic sequences, where all generating substitutions share a common abelianisation. There is a natural correspondence between a random substitution and the set of deterministic substitutions that arises from choosing particular realisations for each letter. To each sequence of substitutions, there is a corresponding S-adic Rauzy fractal and a Rauzy measure. We prove that the union of all S-adic Rauzy fractals obtained in this way equals the Rauzy fractal of the original random substitution.
Additionally, if such sequences of substitutions are chosen independently with identical distribution, then the average S-adic Rauzy measure (where the notion of average is made precise in \Cref{SS:Rauzy_measures}) coincides with the Rauzy measure for an associated random substitution.

Finally, in \Cref{S:examples}, we present several examples illustrating the main results and in \Cref{S:open-probs} we put discuss our results in the context of open problems, both old and new, with a view of encouraging further study of Rauzy fractals and Rauzy measures for random substitutions.

\section{Preliminaries}\label{S:rand-subst-intro}

Throughout, we let $\N = \{1,2,3,\ldots\}$ denote the natural numbers and let $\N_0 = \N \cup \{0\}$. 
For a given set $B$, we write $\# B$ for the cardinality of $B$ and let $\mathcal{F}(B)$ be the set of non-empty finite subsets of $B$.

An alphabet $\mathcal{A} = \{ 1, \ldots, d \}$ is a finite set of symbols, which we call \emph{letters}. 
A \emph{word} with letters in $\mc A$ is a finite concatenation of letters in $\mc A$.
We write $\lvert u \rvert = n$ for the \emph{length} of a given word $u$ and, for $m \in \N$, we let $\mathcal{A}^{m}$ denote the set of all words of length $m$ with letters in $\mathcal{A}$. 
We let $\varepsilon$ denote the empty word, which has length zero by convention.
We write $\mathcal{A}^{+} = \bigcup_{m \in \N} \mathcal{A}^{m}$ for the set of all non-empty finite words with letters in $\mc A$ and $\mathcal{A}^{\ast} = \mathcal{A}^{+} \cup \{ \varepsilon \}$. The set $\mathcal{A}^\ast$ forms a monoid under concatenation of words with identity given by the empty word $\varepsilon$.
We let $\mathcal{A}^{\N} = \{x_1 x_n \cdots : x_i \in \mathcal{A} \ \text{ for all} \ i \in \N \}$ denote the set of all infinite sequences with elements in $\mathcal{A}$ and endow $\mathcal{A}^{\N}$ with the product topology, where $\mathcal{A}$ itself is discrete. 
With this topology, the space $\mathcal{A}^{\N}$ is compact and metrisable.
We let $S$ denote the (\emph{left}) \emph{shift map}, defined by $S (x)_{i} = x_{i+1}$ for $x = x_{1}x_{2} \cdots \in \mathcal{A}^{\N}$, and call a closed subset $X$ of $\mathcal{A}^{\N}$ a \emph{subshift} if it is shift-invariant, namely $S(X) = X$.

If $i,j \in \N$, and $x = x_{1} x_{2} \cdots \in \mathcal{A}^{\N}$, then we let $x_{[i,j]} = x_i x_{i+1} \cdots x_{j}$ if $i \leqslant j$, and $x_{[i,j]} = \varepsilon$ if $j < i$.
We use the same notation if $v \in \mathcal{A}^{+}$ and $1 \leqslant i \leqslant j \leqslant |v|$. For $u,v \in \mathcal{A}^{+}$ (or $v \in \mathcal{A}^{\N}$), we write $u \triangleleft v$ if $u$ is a subword of $v$, namely if there exist $i,j \in \N$ with $i \leqslant j$ so that $u = v_{[i,j]}$. 
For $u,v \in \mathcal{A}^{+}$, we write $\lvert v \rvert_u$ for the number of (possibly overlapping) occurrences of $u$ as a subword of $v$.

If $u = u_{1} \cdots u_{n}, v = v_{1} \cdots v_{m} \in \mathcal{A}^{*}$, for some $n,m \in \N_0$, we write $u v$ for the concatenation of $u$ and $v$, that is, $uv = u_{1} \cdots u_{n} v_{1} \cdots v_{m} \in \mathcal{A}^{n+m}$.
The \emph{abelianisation vector} of a word $u \in \mathcal{A}^{*}$ is the vector $\psi (u) \in \N_0^d$, defined by $\psi (u)_{i} = \lvert u \rvert_{i}$ for all $i \in \{ 1, \ldots, d \}$.

Given some probability space $(\Omega, \mc F, \mathbb{P})$, a \emph{random word} is an $\mc F$-measurable function $w \colon \Omega \to \mc A^+$. We will mostly be interested in the \emph{distribution} of $w$, given by the measure $\mathbb{P} \circ w^{-1}$ on $\mc A^+$, and hence the details of the underlying propability space are of less concern. For every function $f$ with domain $\mc A$ and a random word $w$, the term $f(w)$ is defined to be the random function $f \circ w$ on $\Omega$, with distribution $\mathbb{P} \circ w^{-1} \circ f^{-1}$. This applies in particular to $f$ being the length or the abelianisation vector, such that expressions like $|w|$ or $\psi(w)$ are well-defined for every random word $w$. Similar conventions apply to \emph{random sequences} $x \colon \Omega \to \mc A^\N$, assumed to be $(\mc F,\mc B)$-measurable, where $\mc B$ denotes the Borel $\sigma$-algebra on $\mc A^\N$.

\subsection{Random substitutions}

Throughout the rest of this work, we will usually denote a substitution (also called a deterministic substitution) by greek letters in their standard font, such as $\theta$, and random substitutions by greek letters in their variant font, such as $\vartheta$.
We define a random substitution via the data that is required to determine its action on letters. In the second step we extend it to a random map on words.

\begin{definition}
Let $\mathcal{A} = \{ 1, \ldots, d \}$ be a finite alphabet. A random substitution $\vartheta_{\mathbf{P}} = (\vartheta, \mathbf{P})$ is a finite-set-valued function $\vartheta \colon \mathcal{A} \rightarrow \mathcal{F}(\mathcal{A}^{+})$ together with a set of non-degenerate probability assignments 
	\begin{align*}
	\mathbf{P} = \left\{ p_i \colon \# \vartheta(i) \to (0,1], \text{ with } \sum_{v \in \vartheta(i)} p_i(v) = 1 \text{ for all } 1 \leqslant i \leqslant d \right\}.
	\end{align*}
We call each $v \in \vartheta(i)$ a \emph{realisation} of $\vartheta_{\mathbf{P}}(i)$. 
A \emph{marginal} of $\vartheta_{\mathbf{P}}$ is a deterministic substitution $\theta$ such that $\theta(i)$ is a realisation of $\vartheta_{\mathbf{P}}(i)$ for all $1 \leqslant i \leqslant d$.
\end{definition}
For each $i \in \mc A$ and $v \in \vartheta(i)$, we interpret $p_i(v)$ as the probability to map $i$ to $v$ under $\vartheta_{\mathbf{P}}$. Given $i \in \mc A$ and $\vartheta(i) = \{v_1,\ldots,v_r\}$, this is often represented as
	\begin{align*}
	\vartheta_{\mathbf{P}} \colon i \mapsto
		\begin{cases}
		v_1 & \text{with probability } p_i(v_1),\\
		\hfill \vdots \hfill & \hfill \vdots\hfill\\
		v_r & \text{with probability } p_i(v_r).
		\end{cases}
	\end{align*}

\begin{example}[Random tribonacci]\label{Ex R Fib}
Let $p \in (0,1)$.
The \emph{random tribonacci substitution} $\vartheta_{\mathbf{P}} = (\vartheta, \mathbf{P})$ is the random substitution defined over the alphabet $\mathcal{A} = \{ 1, 2, 3\}$ by
	\begin{align*}
	\vartheta_{\mathbf P} \colon
		\begin{cases}
		1 \mapsto
			\begin{cases}
			12 & \text{with probability } p,\\
			21 & \text{with probability } 1-p,
		     \end{cases}\\
       2 \mapsto 13 \quad \text{with probability } 1,\\
        3 \mapsto 1 \quad \text{with probability } 1,
		\end{cases}
	\end{align*}
with defining data $\mathbf{P} = \{p_1 \colon 12\mapsto p, 21 \mapsto p-1, \; p_2 \colon 13 \mapsto 1, \; p_3 \colon 1 \mapsto 1\}$, and corresponding set-valued substitution $\vartheta \colon 1 \mapsto \{12,21\}, 2 \mapsto \{13\}, 3 \mapsto \{1\}$. 
It is a local mixture of the tribonacci and twisted tribonacci substitutions introduced in Section \ref{SS:subst-intro}.
\end{example}

We now describe how a random substitution $\vartheta_{\mathbf{P}}$ determines a (countable state) Markov matrix $Q$, indexed by $\mathcal{A}^{+} \times \mathcal{A}^{+}$. 
We view the entry $Q_{u,v}$ of $Q$ as the probability to map $u$ to $v$ by the random substitution. 
Formally, $Q_{i, v} = p_i(v)$ for $v \in \vartheta(i)$ and $Q_{i,v} =0$ if $v \notin \vartheta(i)$.  
We extend the action of $\vartheta_{\mathbf{P}}$ to finite words by \emph{independently} mapping each letter to one of its realisations. 
Specifically, given $n \in \N$, $u = u_1 \cdots u_n \in \mathcal{A}^{n}$ and $v \in \mathcal{A}^{+}$, we let
	\begin{align*}
	\mathcal{D}_n(v) = \{ (v^{(1)},\ldots, v^{(n)}) \in (\mathcal{A}^{+})^{n} : v^{(1)} \cdots v^{(n)} = v \}
	\end{align*} 
be the set of all decompositions of $v$ into $n$ individual words and let
	\begin{align*}
	Q_{u,v} = \sum_{(v^{(1)},\ldots,v^{(n)}) \in \mathcal{D}_n(v)} \prod_{i = 1}^{n} Q_{u_i,v^{(i)}}.
	\end{align*} 
Namely, $\vartheta_{\mathbf{P}}(u) = v$ with probability $Q_{u,v}$.

For $u \in \mathcal{A}^{+}$, let $(\vartheta_{\mathbf{P}}^{n}(u))_{n \in \N}$ be a stationary Markov chain on a given probability space $(\Omega_u, \mathcal{F}_u, \mathbb{P}_u)$, with transition matrix $Q$, that is
	\begin{align}\label{eq:transition-prob}
	\mathbb{P}_u [\vartheta_{\mathbf{P}}^{n+1}(u) = w \mid \vartheta_{\mathbf{P}}^{n}(u) = v] = \mathbb{P}_v [\vartheta_{\mathbf{P}}(v) = w] = Q_{v,w},
	\end{align} 
for all $v,w \in \mathcal{A}^{+}$ and $n \in \N$. 
In particular, $\mathbb{P}_u [\vartheta_{\mathbf{P}}^{n}(u) = v] = (Q^{n})_{u,v}$, for all $u,v \in \mathcal{A}^{+}$, and $n \in \N$. A complete formal construction of the probability space $(\Omega_u, \mathcal{F}_u, \mathbb{P}_u)$ is detailed in the work of Gohlke--Spindeler \cite[Appendix]{gohlke-spindeler}, however, all that will be necessary for our purposes are the transition probabilities given in \eqref{eq:transition-prob}.

Likewise, in case that $u$ is a random word (that is, a word-valued random variable), we let $(\vartheta_{\mathbf{P}}^{n}(u))_{n \in \N}$ be a stationary Markov chain, induced by the transition matrix $Q$ as outlined above.
We typically write $\mathbb{P}$ for $\mathbb{P}_u$ if the initial (random) word is understood. 
In this case, we also write $\mathbb{E}$ for the expectation with respect to $\mathbb{P}$.
As above, we say that $v$ is a \emph{realisation} of $\vartheta^{n}_{\mathbf{P}}(u)$ if $(Q^{n})_{u,v} > 0$ and set 
	\begin{align*}
	\vartheta^{n}(u) = \{ v \in \mathcal{A}^{+} : (Q^{n})_{u,v} > 0\}
	\end{align*} 
to be the set of all realisations of $\vartheta_{\mathbf{P}}^{n}(u)$. Conversely, $\vartheta^{n}_{\mathbf{P}}(u)$ may be regarded as the set $\vartheta^{n}(u)$, equipped with the additional structure of a probability vector. If $u = a \in \mathcal{A}$ is a letter, then we call a word $v \in \vartheta^{k}(a)$ a  \emph{(level-$k$) inflation word}.
The approach of defining a random substitution in terms of a Markov chain can be traced back to work of Peyri\`{e}re \cite{peyriere} and was pursued further by Denker and Koslicki \cite{koslicki,koslicki-denker}.

Given a random substitution $\vartheta_{\mathbf{P}} = (\vartheta, \mathbf{P})$ over an alphabet $\mathcal{A} = \{ 1, \ldots, d \}$, with cardinality $d \in \N$, we define the \emph{substitution matrix} $M = M_{\vartheta_{\mathbf{P}}} \in \mathbb{R}^{d \times d}$ of $\vartheta_{\mathbf{P}}$ by
	\begin{align*}
	M_{i, j}
	= \mathbb{E}\left[\lvert \vartheta_{\mathbf{P}}  (j) \rvert_{i}\right]
	= \sum_{v \in \vartheta(j)} p_j(v)\, \lvert v \rvert_{i}.
	\end{align*}
Since $M$ has only non-negative entries, its spectral radius is also a real eigenvalue of maximal modulus, denoted by $\lambda$.
By construction, $\lambda \geqslant 1$, where $\lambda = 1$ occurs if and only if $M$ is column-stochastic. 
This corresponds to the trivial case of a non-expanding random substitution, which we discard in the following. 
If $M$ is primitive (that is, if there exists a $k \in \N_0$ so that all the entries of $M^{k}$ are positive), Perron--Frobenius theory implies that $\lambda$ is a simple eigenvalue and that the corresponding left and right eigenvectors $\mathbf{L} = (L_{1}, \ldots, L_{d})^{T}$ and $\mathbf{R} = (R_{1}, \ldots, R_{d})^{T}$ may be chosen to have strictly positive entries. 
We normalise these eigenvectors such that $\lVert \mathbf{R} \rVert_{1} = 1 = \mathbf{L}^{T} \, \mathbf{R}$. 
In this situation, we call $\lambda$ the \emph{Perron--Frobenius eigenvalue} of $\vartheta_{\mathbf{P}}$, and $\mathbf{L}$ and $\mathbf{R}$ the \emph{left} and \emph{right Perron--Frobenius eigenvectors} of $\vartheta_{\mathbf{P}}$, respectively.

\begin{definition}
We say that $\vartheta_{\mathbf{P}}$ is \emph{primitive} if $M = M_{\vartheta_{\mathbf{P}}}$ is primitive and its Perron--Frobenius eigenvalue satisfies $\lambda > 1$.
\end{definition}

We emphasise that for a random substitution $\vartheta_{\mathbf{P}}$, being primitive is in fact independent of the (non-degenerate) data $\mathbf{P}$. In this sense, primitivity is a property of $\vartheta$ rather than $\vartheta_{\mathbf{P}}$.

Another standard assumption in the study of random substitutions is \emph{compatibility}, see \cite{escolano-manibo-miro,fokkink-rust-salo,gohlke,miro-et-al}. In particular, see \cite{rust-periodic-points} for a short history on the use of the term \emph{compatible}. In the following, recall that we denote the abelianisation vector of a finite word $u$ by $\psi (u)$.

\begin{definition}
We say that a random substitution $\vartheta_{\mathbf{P}} = (\vartheta, \mathbf{P})$ is \emph{compatible} if for all $a \in \mathcal{A}$ and $u, v \in \vartheta(a)$, we have $\psi (u) = \psi (v)$. 
\end{definition}

Observe that compatibility is independent of the choice of probabilities, and that a random substitution $\vartheta_{\mathbf{P}} = (\vartheta, \mathbf{P})$ is compatible if and only if for all $u \in \mathcal{A}^{+}$, we have that $\lvert s \rvert_{a} = \lvert t \rvert_{a}$ for all $s$ and $t \in \vartheta (u)$, and $a \in \mathcal{A}$.
We write $\lvert \vartheta (u) \rvert_{a}$ to denote this common value, and let $\lvert \vartheta (u) \rvert$ denote the common length of words in $\vartheta (u)$.
In which case, letting $M = M_{\vartheta_{\mathbf{P}}}$ denote the substitution matrix of $\vartheta_{\mathbf{P}}$, we have that $M_{i, j} = \lvert \vartheta (j) \lvert_{i}$ for all $i, j \in \mathcal{A}$. Note that the random tribonacci substitution defined in \Cref{Ex R Fib} is compatible, since $\psi (12) = \psi (21) = (1,1,0)^{T}$.
It is also primitive, since the cube of its substitution matrix is positive.
Since the matrix of a compatible random substitution is independent of $\bm P$, we often drop the explicit dependence on $\bm P$ in the notion, and write $M_{\vartheta}$ for the matrix of $\vartheta_{\bm P}$.

For random substitutions which are both primitive and compatible, the notions of unimodular and (irreducible) Pisot extend naturally from the deterministic setting.

\begin{definition}
We say that a compatible primitive random substitution $\vartheta_{\bm P} = (\vartheta, \bm P)$ is \emph{Pisot} if the largest eigenvalue of the substitution matrix $M_{\vartheta}$ is a Pisot number.
If, in addition, the characteristic polynomial of $M_{\vartheta}$ is irreducible over $\Z$, then we say that $\vartheta_{\bm P}$ is \emph{irreducible Pisot}.
If $\lvert \operatorname{det} (M_{\vartheta}) \rvert = 1$, then we say that $\vartheta_{\bm P}$ is \emph{unimodular}.
\end{definition}

\begin{definition}
Given a random substitution $\vartheta_{\mathbf{P}} = (\vartheta, \mathbf{P})$, a word $u \in \mathcal{A}^{+}$ is called \emph{($\vartheta$-)\,legal} if there exists an $a \in \mathcal{A}$ and $k \in \N$ such that $u$ appears as a subword of some word in $\vartheta^{k} (a)$. 
We define the \emph{language} of $\vartheta$ by $\mc L_{\vartheta} = \{ u \in \mathcal{A}^{+} : u \text{ is $\vartheta$-legal} \}$.

The \emph{random substitution subshift} associated with $\vartheta_{\mathbf{P}} = (\vartheta, \mathbf{P})$ is the system $(X_{\vartheta}, S)$, where $S$ is the usual shift map, and
\begin{align*}
    X_{\vartheta} = \{ x \in \mathcal{A}^{\N} : \text{every subword of $x$ is $\vartheta$-legal} \}.
\end{align*}
\end{definition}

Under mild condition, the corresponding sequence space $X_{\vartheta}$ is always non-empty \cite{rust-spindeler}. In particular, this holds if $\vartheta_{\mathbf{P}}$ is primitive and compatible.
The notation $X_{\vartheta}$ mirrors the fact that the subshift of a random substitution does not depend on the choice of $\mathbf{P}$. 
We endow $X_{\vartheta}$ with the subspace topology inherited from $\mathcal{A}^{\N}$, and since $X_{\vartheta}$ is defined in terms of a language, it is a compact $S$-invariant subspace of $\mathcal{A}^{\N}$. 
Hence, $X_{\vartheta}$ is a subshift. 
For $n \in \N$, we write $\mc L_{\vartheta}^{n} = \mc L_\vartheta \cap \mathcal{A}^{n}$ to denote the subset of $\mc L_{\vartheta}$ consisting of words of length $n$.  
We also note that, when $\vartheta_{\bm P}$ is primitive, $X_{\vartheta^{k}} = X_{\vartheta}$ for all $k \in \N$.

\begin{prop}[{\cite[Proposition~13]{rust-spindeler}}]\label{PROP:primitive-implies-transitive}
    Let $\vartheta_{\bm P} = (\vartheta, \bm P)$ be a primitive random substitution. 
    Then, the associated subshift $X_{\vartheta}$ is topologically transitive.
\end{prop}

For compatible random substitutions, every element in the associated subshift has well-defined letter frequencies. 
Under the additional assumption of irreducible Pisot, it was shown in \cite{miro-et-al} that every element is in fact $C$-balanced.

\begin{theorem}[{\cite[Theorem~33]{miro-et-al}}]
\label{THM:Pisot-balanced}
    Let $\vartheta_{\bm P}$ be a compatible and irreducible Pisot random substitution. 
    Then, there exists a $C \geq 1$ such that every element of $X_{\vartheta}$ is $C$-balanced.
\end{theorem}

The set-valued function $\vartheta$ naturally extends to $X_{\vartheta}$, where for $x =  x_{1} x_{2} \cdots \in X_{\vartheta}$ we let $\vartheta(x)$ denote the (infinite) set of sequences of the form $v = v_1 v_2 \cdots$, with $v_j \in \vartheta(x_j)$ for all $j \in \N$. By definition, it is easily verified that $\vartheta(X_{\vartheta}) \subseteq X_{\vartheta}$. Some properties of $\vartheta$ are reminiscent of continuous functions, although $\vartheta$ itself is \emph{not} a function.

\begin{lemma}[{\cite[Lemma~2.5]{MT-entropy}}]\label{LEM:subst-image-compact}
If $\vartheta_{\mathbf{P}} = (\vartheta, \mathbf{P})$ is a random substitution and $X \subseteq \mathcal{A}^{\N}$ is compact, then $\vartheta(X)$ is compact.
\end{lemma}

\subsection{Frequency measures}

The choice of the probability parameters $\mathbf{P}$ induces a probabilistic structure on $X_{\vartheta}$ that is reflected by appropriate choices of probability measures. In analogy to the notion of a substitution-fixed point in the deterministic setting, we introduce the concept of an invariant measure.
\begin{definition}
A probability measure $\nu$ on $X_{\vartheta}$ is \emph{invariant} under $\vartheta_{\mathbf{P}}$, if for all $w \in \mc L_{\vartheta}$,
\begin{align*}
\nu([w]) = \sum_{v \in \mc L_{\vartheta}^{|w|}} \nu([v]) \, \mathbb{P} [\vartheta_{\mathbf{P}}(v)_{[1,\lvert w \rvert]} = w].
\end{align*}
\end{definition}

In many cases, the probability $\mathbb{P} [\vartheta_{\mathbf{P}}(v)_{[1,\lvert w \rvert]} = w]$ depends only on a prefix of $v$. 
More precisely, if $m \in \N$ is large enough to guarantee $|\vartheta(v)| \geqslant |w|$ for all $v \in \mc L_{\vartheta}^m$, we obtain
\begin{align}
\label{EQ:nu-invariance}
\nu([w]) = \sum_{v \in \mc L_{\vartheta}^m} \nu([v]) \, \mathbb{P}[\vartheta_{\mathbf{P}}(v)_{[1,|w|]} = w].
\end{align}
It is not difficult to verify that every compatible primitive random substitution $\vartheta_{\mathbf{P}}$ permits an invariant measure (possibly up to replacing $\vartheta_{\mathbf{P}}$ by a higher power); we refer to \cite[Proposition~4.1.3]{gohlke_diss} for details. A $\vartheta_{\mathbf{P}}$-invariant measure $\nu$ is generally not shift-invariant. However, $\nu$ is intimately connected to a natural ergodic measure on $(X_{\vartheta},S)$. Recall that a point $x \in X_{\vartheta}$ is \emph{generic} for a measure $\varrho$ if 
\[
\varrho = \lim_{n \to \infty} \frac{1}{n} \sum_{k=0}^{n-1} \delta_{S^k x},
\]
in the weak topology. 
The following was shown in \cite[Theorem~5.9]{gohlke-spindeler} and \cite[Theorem~4.1.10, Rem.~4.1.12]{gohlke_diss} in the context of bi-infinite sequence spaces; the proof carries over immediately to the one-sided setting.

\begin{theorem}
\label{THM:frequency-measure}
    Let $\vartheta_{\mathbf{P}}$ be a primitive random substitution. There is an $S$-invariant, ergodic measure $\varrho$ on $X_{\vartheta}$ such that for every $\vartheta_{\mathbf{P}}$-invariant measure $\nu$, we have that $\nu$-almost every $x \in X_{\vartheta}$ is generic for $\varrho$.
\end{theorem}

We call $\varrho$ the \emph{frequency measure} corresponding to $\vartheta_{\mathbf{P}}$ because the value $\varrho([w])$ is precisely the frequency of $w$ in $x$ for $\nu$-almost every (and $\varrho$-almost every) $x \in X_{\vartheta}$. 
In fact, this frequency is also observed in large inflation words in the sense that the relative frequency of $w$ in the random word $\vartheta_{\mathbf{P}}^n(a)$ converges $\mathbb{P}$-almost surely to $\varrho([w])$---see \cite[Proposition~4.3]{gohlke-spindeler} for more details. 
Probabilistic aspects of the subshift $X_{\vartheta}$ are naturally studied in terms of its frequency measures, including measure theoretic entropy \cite{MT-entropy} and $L^q$-spectra \cite{mitchell-rutar}.

\section{Rauzy fractals of \texorpdfstring{$C$-balanced}{C-balanced} sequences}\label{S:Bal-seq-Rauzy}

A sequence $x \in \mc{A}^{\mathbb{N}}$ is called \emph{$C$-balanced} if there exists a constant $C \geqslant 1$ such that for all $j, k, n \in \N$ and $a \in \mc A$, we have $\left\lvert \lvert x_{[j,j+n-1]} \rvert_a - \lvert x_{[k,k+n-1]} \rvert_a \right\rvert \leqslant C$.
We also extend the definition of $C$-balanced to subshifts as follows: a subshift $X$ is $C$-balanced if every sequence $x \in X$ is $C$-balanced.
This property is central to the construction of the Rauzy fractal associated with an irreducible Pisot substitution.
More generally, the same procedure allows a Rauzy fractal to be defined for any $C$-balanced sequence and, further, for any topologically transitive $C$-balanced subshift.
In this section, we outline this more general construction.
We prove several topological and analytic properties that we will utilise in the construction of Rauzy fractals of random substitutions in \Cref{S: Rauzy random subs}, which we believe are of interest in their own right.

We first provide an alternative characterisation of the $C$-balanced property, which will be useful for our purposes.
A sequence $w \in \mc A^{\N}$ has \emph{uniformly well-defined letter frequencies} if, for all $i \in \mc A$ and all sequences $(j_n)_n$ of positive integers, the limit
\begin{align*}
    r_i = \lim_{n \rightarrow \infty} \frac{1}{n} \left\lvert w_{[j_n,j_n+n-1]} \right\rvert_{i}
\end{align*}
exists and is independent of the sequence $(j_n)_n$. 
We call the vector $\mathbf{r} = \mathbf{r} (w) = (r_1 \ldots, r_d)$ the \emph{letter frequency vector of $w$}. Observe that $\mathbf{r}$ is a probability vector. 
The following characterisation of the $C$-balanced property was proved in \cite{BD14}, and motivates why a Rauzy fractal can be associated with any $C$-balanced sequence.

\begin{lemma}[{\cite[Proposition~2.4]{BD14}}]\label{LEM:C-bal-ULF}
A sequence $w \in \mc A^{\N}$ is $C$-balanced if and only if $w$ has uniformly well-defined letter frequencies and there exists a constant $B$ so that $\lvert \left\lvert u \right\rvert_{i} - \lvert u \rvert r_i \rvert < B$, for any finite subword $u$ of $w$ and $i \in \mc A$, and where $r_i$ is the entry of the letter frequency vector corresponding to $i$.
\end{lemma}

For topologically transitive $C$-balanced subshifts, there exists a uniform letter frequency vector.

\begin{prop}
    If $X$ is a topologically transitive $C$-balanced subshift, then there is a probability vector $\bm r$ such that for every $x \in X$, $\bm r$ is the letter frequency vector of $x$.
\end{prop}
\begin{proof}
    Since $X$ is $C$-balanced, every element of $X$ has uniformly well-defined letter frequencies by \Cref{LEM:C-bal-ULF}.
    Let $\bm r (x)$ denote the letter frequency vector of a given element $x \in X$.
    By transitivity, there exists an element $w$ with dense shift-orbit in $X$.
    Thus, for every $x \in X$, there is a sequence $(n_k)_k$ of positive integers such that $S^{n_k} (w) \rightarrow x$ as $k \rightarrow \infty$.
    In particular, $\bm r (x) = \bm r (w)$.
\end{proof}

\subsection{Rauzy fractals of \texorpdfstring{$C$-balanced}{C-balanced} sequences}\label{SS:Rauzy-construction}

Let $w$ be a $C$-balanced sequence over a finite alphabet $\mc A = \{1,\ldots,d\}$, with letter frequency vector $\mathbf v$, and let $\mathbb H$ be a \mbox{$(d-1)$-dimensional} hyperplane passing through the origin, such that every non-zero vector in $\mathbb H$ is linearly independent to $\mathbf v$. 
Let $\pi \colon \mathbb{R}^{d} \rightarrow \mathbb{R}^{d}$ denote the linear projection along $\mathbf v$ onto $\mathbb H$. 
For each $a \in \mc A$, let
\begin{align*}
    \mathcal{S}_a (w) = \{ \psi (w_{[1,n]}) : n \in \mathbb{N}_{0} \ \text{and} \ w_{n+1} = a \} .
\end{align*}
Setting $\mc R_{a}^{*}(w) = \pi \mc (\mathcal{S}_a (w))$ and $\mc R_{a} = \overline{\mc R_{a}^{*}}$, we define the Rauzy fractal $\mc R (w)$ of $w$ by
\begin{align*}
    \mathcal{R} (w) = \bigcup_{a \in \mathcal{A}} \mathcal{R}_{a} (w) .
\end{align*}
Since $w$ is $C$-balanced, $\mathcal{R} (w)$ is bounded, and since $\mathcal{R} (w)$ is a finite union of closed sets, it is compact. 
We use the same notation if $w$ a finite word, with the convention that $w_{[1,n]}=w$ if $n>|w|$, and we observe that for finite words $\mathcal{R} (w)$ is a finite set and thus bounded. 

The following describes how the action of the shift map on a $C$-balanced sequence corresponds to a translation in the Rauzy fractal.

\begin{prop}\label{PROP:Rauzy shifts}
Let $w \in \mc A^{\N}$ be a $C$-balanced sequence. Then, for all $a \in \mathcal{A}$ and $k \geqslant 0$, 
\begin{align*}
    \mathcal{R}^{*}_a (S^{k} (w)) \subseteq \mathcal{R}^{*}_a (w) - \pi (\psi(w_{[1,k]})) \quad \text{and} \quad
    \mathcal{R}_a (S^{k} (w)) \subseteq \mathcal{R}_a (w) - \pi (\psi(w_{[1,k]})) .
\end{align*}
\end{prop}
\begin{proof}
If $x \in \mathcal{R}_{a}^{*} (S^{k} (w))$, then there exists an $m \geqslant 0$ with $x = \pi (\psi (S^k (w)_{[1,m]}))$ and $S^k (w)_{m+1} = a$ (that is, $w_{m+k+1} = a$). From this and by linearity of the projection map $\pi$, we have
\begin{align*}
        \pi (\psi(S^k (w)_{[1,m]})) = \pi (\psi(w_{[k+1,k+m]})) = \pi (\psi(w_{[1,k+m]})) - \pi (\psi(w_{[1,k]})).
\end{align*}
Hence, $x \in \mathcal{R}_{a}^{*}(w) - \pi(\psi(w_{[1,k]}))$. Thus, $\mathcal{R}_{a}^{*} (S^{k} (w)) \subseteq \mathcal{R}_{a}^{*} (w) - \pi (\psi(w_{[1,k]}))$, and taking closure gives $\mathcal{R}_{a} (S^{k} (w)) \subseteq \mathcal{R}_{a} (w) - \pi(\psi(w_{[1,k]}))$.
\end{proof}

\subsection{Analytic properties}\label{SS:analytic-props}

In this section we look at limiting behaviour, with respect to the Hausdorff metric, of Rauzy fractals of $C$-balanced sequences. 
The following characterisation of convergence in the Hausdorff metric will be most convenient for our purposes.

\begin{definition}
For a sequence $(A_n)_{n}$ of compact sets in a metric space $(X,\rho)$, the \emph{Kuratowski-Painlev\'{e} limit inferior} is defined by
\begin{align*}
\underset{n\to\infty}\Li A_n = \left\{x : \limsup_{n \to \infty} \rho (x,A_n) =0 \right\} ,
\end{align*}
where $\rho (x,A_n)$ is the distance between the point $x$ and the set $A_n$.
Analogously, the \emph{Kuratowski-Painlev\'{e} limit superior} is defined by
\begin{align*}
\underset{n\to\infty}\Ls A_n = \left\{x : \liminf_{n \to \infty} \rho (x,A_n) =0 \right\}.
\end{align*}
If the Kuratowski-Painlev\'{e} limit inferior and superior agree, then the common set is called the \emph{Kuratowski-Painlev\'{e} limit} of $A_n$ and is denoted by $\underset{n\to\infty}\Lim A_n$. Here, by $\rho (x,A)$ for $A \subseteq X$ we mean $\inf \{ \rho(x, y) : y \in A \}$.
\end{definition}

If $A_n$ converges to $A$ in the Hausdorff metric, then $A$ is the Kuratowski-Painlev\'{e} limit of of $(A_n)_n$. Conversely, if for all but a finite number of $n\in \N$, the set $A_n$ is compact, then Kuratowski-Painlev\'{e} convergence is equivalent to convergence in the Hausdorff metric \cite{beer}. Further, observe that 
\begin{align*}
\underset{n\to\infty}\Li A_n \subseteq \underset{n\to\infty}\Ls A_n.
\end{align*}
For more details on Kuratowski-Painlev\'{e} convergence and its implications, we refer the reader to \cite[Chapter 5]{beer},  specifically Corollary 5.1.11 and Theorem 5.2.10 of \cite{beer}. 

It is possible for a sequence $(w^{(n)})_{n \in \N}$ of $C$-balanced sequences to converge to a $C$-balanced sequence $w$, but for $\mc R (w^{(n)})$ not to converge to $\mc R (w)$ in the Hausdorff metric.
For example, consider $w^{(n)} = a^n b a^\infty$, which converges to $w = a^\infty$ as $n \to \infty$.
For all $n \in \N$, $\mc R(w^{(n)})$ is the same set of two points and does not converge to the singleton set $\mc R(w)$.
However, the following lower semi-continuity always holds.

\begin{lemma}
\label{LEM:lsc}
If $(w^{(n)})_{n \in \N}$ is a sequence of $C$-balanced sequences converging to a $C$-balanced sequence $w$, then
\begin{align*}
\mc R_a (w) \subseteq \underset{n\to\infty} \Li \mc R_a (w^{(n)}).
\end{align*}
\end{lemma}

\begin{proof}
Let $v \in \mc R_a (w)$ and let $\varepsilon > 0$. There exists an $m = m(\varepsilon) \geqslant 0$ with $|\pi(\psi(w_{[1,m]})) - v| < \varepsilon$ and $w_{m+1} = a$.
By assumption, there is an $n_0$ such that $w^{(n)}_{[1,m+1]} = w_{[1,m+1]}$ for all $n \geqslant n_0$, and therefore $\pi(\psi(w_{[1,m]})) \in \mc R_a(w^{(n)})$.
It follows that $\rho (v, \mc R_a(w^{(n)})) < \varepsilon$ for all $n \geqslant n_0$. 
As $\varepsilon > 0$ was arbitrary, we conclude that $\rho(v,\mc R_a(w^{(n)})) \rightarrow 0$ as $n \rightarrow \infty$, so $v$ lies in $\Li_{n\to\infty} \mc R_a (w^{(n)})$. 
Since $v \in \mc R_a (w)$ was chosen arbitrarily, the assertion follows.
\end{proof}

It follows from the above and \Cref{PROP:Rauzy shifts} that for every element in the orbit closure of a $C$-balanced sequence $w$, the associated Rauzy fractal is contained in a translate of the Rauzy fractal of $w$.

\begin{lemma}
\label{LEM:limit-subset}
Let $w$ be a $C$-balanced sequence.
For every $x$ in the shift-orbit closure of $w$, there is a vector $t \in \mc R (w)$ with $\mc R_a (x) \subseteq \mc R_a (w) - t$, for all $a \in \mc A$.
\end{lemma}

\begin{proof}
By assumption, there exists a sequence $(n_k)_{k}$ with
\begin{align*}
\lim_{k \to \infty} S^{n_k} (w) = x.
\end{align*}
Let $\mc T$ be the set of accumulation points of the sequence $\pi(\psi(w_{[1,n_k]}))$.
Note that the set $\mc T$ is non-empty since the sequence is bounded.
For $t \in \mc T$ there exists a subsequence $(m_k)_{k}$ such that $\pi(\psi(w_{[1,m_k]})) \to t$ and $S^{m_k}(w) \to x$ as $k \to \infty$. 
Hence, by \Cref{LEM:lsc} and \Cref{PROP:Rauzy shifts}, we obtain $\mc R_a(x) \subseteq \mc R_a(w) - t$. Since $0 \in \mc R (x)$, it follows that $t \in \mc R (w)$.
\end{proof}

In fact, since in the above proof the statements hold for all $t \in \mc T$, we have
\[
\mc R_a(x) \subseteq \bigcap_{t \in \mc T} \mc R_a(w) -t.
\]

\begin{corollary}\label{CO: dense-shift-Rauzy}
Let $X$ be a subshift and suppose that $w, x \in X$ are $C$-balanced and have dense shift-orbit in $X$.
Then, there exists a vector $t \in \mc R (w)$ such that $\mc R_a(x) = \mc R_a(w) - t$, for all $a \in \mc A$.
\end{corollary}

\begin{proof}
By \Cref{LEM:limit-subset}, there exist vectors $s,t \in \mc R (w)$ with $\mc R_a(x) \subseteq \mc R_a(w) - t \subseteq \mc R_a(x) - s - t$, for all $a \in \mc A$.
Since $\mc R_a(x)$ is bounded, this is only possible if all the subset relations are in fact equalities and $s = - t$.
Thus, $\mc R_a(x) = \mc R_a(w) - t$ for all $a \in \mc A$.
\end{proof}

As a consequence of \Cref{LEM:lsc} and \Cref{LEM:limit-subset}, if a sequence $(w^{(n)})_n$ of $C$-balanced words converges to a sequence $w \in X$ with dense shift-orbit, $(\mc R(w^{(n)}))_n$ converges to $\mc R(w)$ in the Hausdorff metric.

\begin{prop}
\label{PROP:Hausdorff-convergence}
Let $X$ be a $C$-balanced subshift. If $(w^{(n)})_{n}$ is a sequence in $X$ that converges to a $w \in X$ with dense shift-orbit, then $\mc R_a(w^{(n)}) \rightarrow \mc R_a(w)$ in the Hausdorff metric as $n \to \infty$, for all $a \in \mc A$.
\end{prop}

\begin{proof}
By \Cref{LEM:limit-subset}, for every $w^{(n)}$ there exists a vector $t_n$ in $\mc R(w)$ such that $\mc R_a(w^{(n)}) \subseteq \mc R_a(w) - t_n$. Let $(n_k)_{k}$ be a subsequence such that $t_{n_k}$ converges to some $t \in \mc R(w)$ as $k \to \infty$. This implies that
\begin{align*}
\lim_{k \to \infty} \mc R_a(w) - t_{n_k} = \mc R_a(w) - t
\end{align*}
in the Hausdorff distance.
Using \Cref{LEM:lsc}, we obtain that
\begin{align*}
\mc R_a(w) \subseteq \underset{n \to \infty}\Li \mc R_a(w^{(n)}) 
\subseteq \underset{k \to \infty} \Li \mc R_a(w^{(n_k)})
 \subseteq  \underset{k \to \infty}  \Li \mc R_a(w) - t_{n_k} 
= \mc R_a(w) - t.
\end{align*}
It follows that $t = 0$ and that each of the subset relations is in fact an equality. 
Since every subsequence of $t_n$ converges to $0$, we obtain that $\lim_{n \to \infty} t_n = 0$, therefore
\begin{align*}
\mc R_a(w) \subseteq \underset{n \to \infty}\Li \mc R_a(w^{(n)}) 
\subseteq \underset{n \to \infty} \Ls \mc R_a(w^{(n)}) \subseteq \underset{n \to \infty} \Ls \mc R_a(w) - t_n = \mc R_a(w). 
\end{align*}
Hence,
\begin{align*}
\underset{n \to \infty} \Li \mc R_a(w^{(n)}) = \underset{n \to \infty} \Ls \mc R_a(w^{(n)}) = \mc R_a(w),
\end{align*}
implying that the Kuratowski-Painlev\'{e} limit exists and is given by $\mc R_a(w)$. Since all the sets $\mc R_a(w^{(n)})$ are contained in the bounded Minkowski difference $\mc R(w) - \mc R(w)$, the convergence in Hausdorff distance follows.
\end{proof}

\subsection{Rauzy fractals as generic factors}
\label{SUBSEC:generic-factors}

For the remainder of this section, we assume $X$ to be a topologically transitive $C$-balanced subshift.
Let us fix a sequence $w \in X$ with dense shift-orbit as a reference point. 
We can define a mapping $\phi \colon X' \to \mc R(w)$ on the set $X'$ of points with dense shift-orbit by $\phi(x) = t$, where $t$ is the unique vector such that $\mc R(x) = \mc R(w) - t$. 

\begin{prop}
\label{PROP:factor-continuity}
The map $\phi \colon X' \to \mc R(w)$ is continuous.
\end{prop}

\begin{proof}
Let $x \in X'$ and $x^n \in X'$ for all $n \in \N$ such that $x^n \to x$. 
By the definition of $\phi$, we have that $\mc R(x) = \mc R(w) - \phi(x)$ and $\mc R(x^n) = \mc R(w) - \phi(x^n)$.  
Further, by \Cref{PROP:Hausdorff-convergence}, 
\begin{align*}
\mc R(w) - \phi(x) = \mc R(x) = \lim_{n \to \infty} \mc R(x^n) = \lim_{n \to \infty} \mc R(w) - \phi(x^n)
\end{align*}
in the Hausdorff distance.
This implies that $\lim_{n \to \infty} \phi(x^n) = \phi(x)$.
\end{proof}

This result is particularly useful if the set of transitive sequences $X'$ is invariant under the shift map. This is the case if some (equivalently every) sequence $x \in X'$ is \emph{recurrent}. That is, every subword of $x$ appears in $x$ infinitely often. We take this as a standing assumption for the remainder of this subsection.

\Cref{PROP:factor-continuity} has a convenient dynamical interpretation, as it helps to construct an explicit \emph{generic factor} of the subshift $(X,S)$. Following \cite{huang-ye}, we say that a transitive dynamical system $(Y,T)$ is a generic factor of $(X,S)$ if there is a continuous map $\varphi$ from the set of transitive points $X'$ of $X$ to the set of transitive points $Y'$ of $Y$ such that $\varphi \circ S = T \circ \varphi$ on $X'$. Generic factors are a convenient classification tool in cases when the maximal equicontinuous factor is trivial and the system is not uniquely ergodic, such that \textit{a priori} there is not a unique choice for the Kronecker factor. For details on the concept of a \emph{maximal equicontinuous generic factor} (MEGF) and some of its applications in the context of aperiodic order we refer the reader to \cite{keller}. 

In order to interpret $\phi$ as a factor map, we need to equip $\mc R(w)$ with an action that corresponds to the shift action on $X'$. It is readily verified from Corollary~\ref{CO: dense-shift-Rauzy} and Proposition~\ref{PROP:Rauzy shifts} that
\begin{align}
\label{EQ:pre-semi-conjugation}
\phi(Sx) = \phi(x) + \pi(\psi(x_1)),
\end{align}
for all $x \in X'$. For this translation to be independent of $x$, we wish to identify the vectors $\pi(\mathbf{e}_i)$ for all $1\leqslant i \leqslant d$. That is, we consider $({\mathbf v}_i)_{i = 2}^{d}$, with ${\mathbf v}_i = \pi(\mathbf{e}_i - \mathbf{e}_1)$, and the lattice spanned by these vectors
\begin{align}
\label{EQ:lattice}
\mc J = \left\{ \sum_{i = 2}^d z_i {\mathbf{v}_i} : z_i \in \Z \; \text{for} \; i \in \{1, 2, \ldots, d \} \right\},
\end{align}
where $({\mathbf e}_{i} )_{i = 1}^{d}$ denotes the standard basis of $\R^{d}$. With $\pi_{\mc J}$ the natural projection from $\mathbb{H}$ to the $d-1$ dimensional torus $\mathbb{H}/\mc J$, we obtain the following consequence of Proposition~\ref{PROP:factor-continuity}.

\begin{corollary}
\label{COR:generic-factor}
Let $X$ be the orbit closure of a recurrent, $C$-balanced sequence $w$.
    The map $\pi_{\mc J} \circ \phi$ is a generic factor map from $(X,S)$ to the equicontinuous dynamical system $(G, T)$, where $T$ is the torus rotation
    \[
    T \colon \mathbb{H}/\mc J \to \mathbb{H}/\mc J, \quad x \mapsto x + \pi(\mathbf{e}_1) \mod \mc J,
    \]
    and $G = \pi_{\mc J}(\mc R(w))$ is the subgroup of $\mathbb{H}/\mc J$ generated by $\pi(\mathbf{e}_1)$.
\end{corollary}

\begin{proof}
The map $\varphi = \pi_{\mc J} \circ \phi$ is clearly continuous as it is a composition of continuous maps; compare Proposition~\ref{PROP:factor-continuity}. 
    Projecting the relation in \eqref{EQ:pre-semi-conjugation} to the torus yields 
    \[
    \varphi (Sx) = \varphi(x) + \pi({\mathbf e}_1) \mod \mc J
    \; = T(\varphi(x)),
    \]
    for all $x \in X'$. 
    By construction, the set $\{ \psi(x_{[1,n]}) \}_{n \in \mathbb{N}}$ lies dense in $\mc R(x)$ and consequently the points
    \[
    \phi(S^n x) = \phi(x) + \pi(\psi(x_{[1,n]})), 
    \]
    with $n \in \mathbb{N}$, lie dense in $\mc R(w) = \mc R(x) + \phi(x)$.
    This holds in particular for $\mc R^\ast(w) = \{ \phi(S^n w) \}_{n \in \N_0}$, whose projection is given by
    \[
    \pi_{\mc J}(\mc R^\ast(w)) = \{ \varphi(S^n w) \}_{n \in \N_0} = \{ T^n(0) \}_{n \in \N_0},
    \]
    the (forward) orbit of $0$ under $T$. Due to the compactness of $\mc R(w)$, this yields
    \[
    G:= \overline{\{ T^n(0) \}_{n \in \N}} = \pi_{\mc J}(\mc R(w)).
    \]
    Similarly, we obtain that $\{ \varphi(S^n x) = T^n \varphi(x) \}_{n \in \mathbb{N}}$ is dense in $\pi_{\mc J}(\mc R(w))$ for every $x \in X'$. That is, $\varphi(x)$ has a dense orbit in $(G,T)$ for every $x \in X'$.
\end{proof}

\begin{remark}
\label{REM:MEGF}
    If $(X,S)$ is the subshift of a unimodular, irreducible Pisot substitution, we have that $X' = X$ and $(G,T)$ is in fact a factor. Under the assumption that the projection $\pi_{\mc J} \colon \mc R(w) \to \mathbb{H} / \mc J$ is almost surely $1$-$1$, the map $\varphi$ is even a measurable isomorphism between $(X,S)$ and $(G,T)$, where both are endowed with their unique ergodic measure. In this case, $(X,S)$ has pure point dynamical spectrum and $(G,T)$ is the maximal equicontinuous factor \cite{siegel}. For random substitutions, the spectrum is generally richer and eigenfunctions do not necessarily have a continuous representative. In fact, it was shown for the random Fibonacci substitution in a geometric setting that the MEF is trivial and a natural analogue of $\varphi$ provides the MEGF instead \cite{baake-spindeler-strungaru}.
\end{remark}

\subsection{Measures and the Lebesgue covering property}\label{SS:rauzy-measures}

For a $C$-balanced sequence $w$ and $n \in \mathbb{N}_{0}$, let 
\begin{align*}
\mu (w_{[1,n]}) = \sum_{i=1}^n \delta_{\pi \circ \psi(w_{[1,i]})}
\quad \text{and} \quad
\mu (w) = \lim_{n \rightarrow \infty} n^{-1} \mu (w_{[1,n]})
\end{align*}
whenever this latter weak limit exists. Let us emphasise that $\mu(w)$ is a measure associated with a sequence $w$ and not to be confused with the evaluation of a measure at a singleton set.  We call the measure $\mu(w)$, when it exists, a \emph{Rauzy measure}.  Note, in any case, that there always exists a limit point, under the topology of weak convergence, of the sequence $(n^{-1} \mu (w_{[1,n]}))_{n}$, by the Banach--Alaoglu theorem. However, this limit point may not be unique.
Note that the Rauzy fractal $\mc R(w)$ always contains the support $\operatorname{supp} \mu(w)$ of the Rauzy measure.

In this section, we show that covering the stable subspace $\mathbb H$ with Rauzy measures (according to the lattice $\mc J$) creates a multiple of Lebesgue measure. Recall the definition of $\mc J$ from \eqref{EQ:lattice}. For ${\mathbf j} \in \mc J$, we define a corresponding translation $t_{\mathbf j} \colon \mathbb H \to \mathbb H$ by $t_{\mathbf j}({\mathbf x}) = {\mathbf x} + {\mathbf j}$, which is a bijection on $\mathbb H$.

\begin{theorem}
\label{PROP:Lebesgue-covering}
If $w \in \mc A^{\N}$ is an infinite, $C$-balanced word with totally irrational frequency vector $\mathbf r$, and if the weak limit $\mu(w)$ exists, then 
\begin{align*}
\mu_{\mc J}(w) = \sum_{{\mathbf j} \in \mc J} \mu(w) \circ t_{\mathbf j}^{-1} = D \Leb,
\end{align*}
where $D$ is the density of points in $\mc J$ and $\Leb$ denotes the Lebesgue measure on $\mathbb H$.
\end{theorem}

As an intermediate step, instead of $\mathbb H$ we consider the plane orthogonal to ${\mathbf v_0} = (1,\ldots,1)$, that is, $E_0 = \{ {\mathbf v} \in \R^d : {\mathbf v} \cdot {\mathbf v_0} = 0 \}$. Observe that the intersection $V_0 = E_0 \cap \Z^d$ is a lattice. 

\begin{lemma}
$V_0$ is the integer span of the vectors $({\mathbf w_i})_{i=2}^d$, with ${\mathbf w_i} = ({\mathbf e_i} - {\mathbf e_1})$.
\end{lemma} 

\begin{proof}
Notice that each of the vectors ${\mathbf w_i}$ lies in $V_0$ and so does each of their linear combinations with integer coefficients. Conversely, if ${\mathbf u} = (u_1,\ldots, u_d) \in V_0$, then $u_1 = - \sum_{i =2}^d u_i$ and so ${\mathbf u} = \sum_{i = 2}^d u_i {\mathbf w_i}$.
\end{proof}

Analogously to the definition of $E_0$ and $V_0$, for an integer $k \geqslant 0$, we let $E_k = \{ {\mathbf v} \in \R^d : {\mathbf v} \cdot {\mathbf v_0} = k \}$, $E_{[0,k]} = \{ {\mathbf v} \in \R^d : {\mathbf v} \cdot {\mathbf v_0} \in [0,k] \}$, $V_k = E_k \cap \Z^d$, and $V_{[0,k]} = E_{[0,k]} \cap \Z^d$.

Since the scalar product of a vector in $\Z^d$ with $\mathbf v_0$ is always an integer, each element of $\Z^d$ is contained in precisely one of the affine spaces $E_k$. Hence,
\begin{align}
\label{EQ:V+decomposition}
V_{[0,n]} = \bigsqcup_{k = 0}^n V_k,
\end{align}
a disjoint union of the sets $V_k$.
This observation will be useful for the following decomposition result, and for which we recall that $\mc S(w)= \{\psi(w_{[1,k]}) : k = 0, 1, 2, \ldots\}$.

\begin{lemma}
\label{LEM:V+covering}
If $w \in \mc A^n$, then
\begin{align*}
V_{[0,n]} = \bigsqcup_{{\mathbf v} \in V_0} \mc S(w) + {\mathbf v}.
\end{align*}
\end{lemma}

\begin{proof}
Note that $\psi(w_{[1,n]}) \cdot {\mathbf v_0} = n$ for all $n \in \mathbb{N}_{0}$. 
Hence, the intersection of $\mc S(w)$ with each of the affine spaces $E_k$ is a singleton ${\mathbf s_k} \in V_k$. 
Since $V_k = V_0 + {\mathbf s_k}$, we obtain from \eqref{EQ:V+decomposition} that
\[
V_{[0,n]} 
= \bigsqcup_{k = 0}^n V_0 + {\mathbf s_k}
= \bigsqcup_{k = 0}^n \bigsqcup_{{\mathbf v} \in V_0} {\mathbf s_k} + {\mathbf v} = \bigsqcup_{{\mathbf v} \in V_0} \mc S(w) + {\mathbf v}.
\qedhere
\]
\end{proof}

We now give the proof of \Cref{PROP:Lebesgue-covering}.

\begin{proof}[Proof of \Cref{PROP:Lebesgue-covering}]
It suffices to show that $\mu_{\mc J}(w)(B) = D \Leb(B)$ for every open ball $B$ in $\mathbb H$ (with respect to the maximum-norm).
This is because the collection of such balls forms a $\pi$-system that generates the Borel $\sigma$-algebra and equality of the measures then follows by the $\pi$--$\lambda$ theorem, see for example \cite[Theorem~3.3]{billingsley}.
By construction, the restriction of $\pi$ to $V_0$ is a bijection onto its image $\mc J$.
Hence, for each ${\mathbf j} \in \mc J$ there is precisely one ${\mathbf v} = {\mathbf v_j} \in V_0$ with $\pi({\mathbf v}) = {\mathbf j}$. Given an arbitrary bounded and continuous function $f$ on $\mathbb H$ , we obtain
    \begin{align*}
        \mu(w) \circ t_{\mathbf j}^{-1} (f) 
            = \mu(w)(f \circ t_{\mathbf j})
            &= \lim_{n \to \infty} n^{-1} \mu(w_{[1,n]}) (f \circ t_{\mathbf j})\\
            &= \lim_{n \to \infty} n^{-1} \sum_{{\mathbf y} \in \mc S(w_{[1,n]})} f ({\mathbf j} + \pi({\mathbf y}))
            = \lim_{n \to \infty} n^{-1} \sum_{{\mathbf z} \in \mc S(w_{[1,n]}) + {\mathbf v_j}} f(\pi({\mathbf z})),
    \end{align*}
where $S(w_{[1,n]}) = \{ \psi(w_{[1,i]}) : i \in \{ 0, 1, 2, \ldots, n \} \}$, for $n \in \mathbb{N}_{0}$. This chain of equalities implies, by the Portmanteau theorem, in the sense of weak convergence, that
    \begin{align*}
        \mu(w) \circ t_{\mathbf j}^{-1} 
        = \lim_{n \to \infty} n^{-1} \sum_{{\mathbf z} \in \mc S(w_{[1,n]}) + {\mathbf v_j}} \delta_{\pi({\mathbf z})}.
    \end{align*}
Due to the $C$-balancedness of $w$, each $\mu(w) \circ t_{\mathbf j}^{-1}$ has compact support and hence, given an arbitrary open ball $B$, only finitely many of the expressions $\mu(w) \circ t_{\mathbf j}^{-1}(B)$ are non-zero. This together with \Cref{LEM:V+covering} and the characterisation of weak convergence in terms of open sets, implies
    \begin{align*}
        \mu_{\mc J}(w)(B) 
        = \sum_{{\mathbf j} \in \mc J} \mu(w) \circ t_{\mathbf j}^{-1}(B)
        &\leqslant \lim_{n\to \infty} n^{-1} \sum_{{\mathbf j} \in \mc J} \sum_{{\mathbf z} \in \mc S(w_{[1,n]}) + {\mathbf v_j}} \delta_{\pi({\mathbf z})}(B).\\ 
        & = \lim_{n \to \infty} n^{-1} \sum_{{\mathbf x} \in V_{[0,n]}} \delta_{\pi({\mathbf x})}(B)
        = \lim_{n \to \infty} n^{-1} \# ( \pi^{-1}(B) \cap \Z^d \cap E_{[0,n]} ).
    \end{align*}
By uniform distribution results (see for example \cite[Proposition~4.2]{Hof98} or \cite[Proposition~2.1]{Schl98}), the intersection of $\Z^d$ with $\pi^{-1}(B)$ has well-defined frequency $1$ (along appropriate averaging sequences). Hence, the limit does not change if the cardinality in the last expression is replaced by the volume of the corresponding tube $T_n$, given by $T_n = \pi^{-1}(B) \cap E_{[0,n]}$. 
Since $\mathbb H$ is not necessarily perpendicular to $\mathbf r$, it is convenient to project $B$ to ${\mathbf r}^{\perp}$. 
More precisely, we let $\pi_{\mathbf r}$ be the projection along $\mathbf r$ onto ${\mathbf r}^{\perp}$ and observe
\begin{align*}
\pi^{-1}(B) = \pi^{-1}(\pi_{\mathbf r}(B)) = \pi_{\mathbf r}(B) + \R \mathbf r.
\end{align*}
For all ${\mathbf v} \in \R^d$ the intersection $({\mathbf v} + \R \mathbf r) \cap E_{[0,n]}$ has the same length $r_n$. Indeed, let $\mathbf{w}_0$ and $\mathbf{w}_n = c_n \mathbf{r} + \mathbf{w}_0$ be the intersection points of $\mathbf v + \R \mathbf r$ with $E_0$ and $E_n$, respectively. Then, $n = \mathbf{w}_n \cdot \mathbf{v}_0 = c_n \mathbf r \cdot \mathbf{v}_0 = c_n$, and hence $r_n = \|\mathbf{w}_n - \mathbf{w}_0\| = n \| \mathbf r\|$,
which is independent of $\mathbf v$. The volume of $T_n$ is therefore given by
\begin{align*}
r_n \Leb(\pi_{\mathbf r} (B)) = n \| \mathbf r\| \Leb(\pi_{\mathbf r} (B)),
\end{align*}
yielding
\begin{align*}
\mu_{\mc J}(w)(B) \leqslant \| \mathbf r\| \Leb(\pi_{\mathbf r} (B)).
\end{align*}
A parallel argument (using the characterisation of weak convergence in terms of closed sets) shows that
\begin{align*}
\mu_{\mc J}(w)(\overline{B}) \geqslant \| \mathbf r\| \Leb(\pi_{\mathbf r} (\overline{B}))
\end{align*}
for every closed ball $\overline{B}$ in $\mathbb H$. Using the (inner) regularity of Lebesgue measure, this implies that in fact $\mu_{\mc J}(w)(B) = \| \mathbf r\| \Leb(\pi_{\mathbf r} (B))$ for every open ball $B$.
Hence, we obtain
\begin{align}
\label{EQ:mu-J-R-expression}
\mu_{\mc J}(w) = \| \mathbf r\| \Leb \circ \, \pi_{\mathbf r}.
\end{align}
Since $\pi_{\mathbf r}$ restricts to a linear isomorphism from $\mathbb H$ to ${\mathbf r}^{\perp}$, there is a constant $c$ such that $\mu_{\mc J}(w) = c \Leb$.

We can therefore determine the normalisation by considering the fundamental domain $J$, spanned by $(\mathbf{w}_i)_{i=2}^d$. Its image $J' = \pi_{\mathbf r}(J)$ is then spanned by the vectors $(\mathbf{u}_i)_{i = 2}^d$, with $\mathbf{u}_i = \pi_{\mathbf r} (\mathbf{e}_i - \mathbf{e}_1)$. Since $\mathbf r$ is perpendicular to each of the $\mathbf{u}_i$ the Lebesgue measure of $J'$ can be expressed as
\begin{align*}
\frac{1}{\| \mathbf r\|} |\det(\mathbf r, \mathbf{u}_2,\ldots, \mathbf{u}_d)| 
= \frac{1}{\| \mathbf r\|} |\det(\mathbf r, \mathbf{e}_2 - \mathbf{e}_1, \ldots, \mathbf{e}_d - \mathbf{e}_1)| = \frac{1}{\| \mathbf r\|} \sum_{j =1}^d r_j = \frac{1}{\|\mathbf r\|},
\end{align*}
where the first step follows from the fact that $\mathbf{u}_i$ and $\mathbf{e}_i - \mathbf{e}_1$ only differ by a multiple of $\mathbf r$ and the second step follows by a straightforward calculation, for example by expanding the determinant along the first column.
It follows from \eqref{EQ:mu-J-R-expression} that $\mu_{\mc J}(w)(J) = \| \mathbf r\| \Leb(J') = 1$, and therefore $c = (\Leb(J))^{-1}$, which is precisely the density of points in $\mc J$.
\end{proof}

The statement in Theorem~\ref{PROP:Lebesgue-covering} holds for every weak accumulation point of the sequence $(\mu(w_{[1,n]}))_n$, even if the limit does not exist. This can be verified by restricting to the corresponding subsequence in the proof provided above.

\subsection{Rauzy measures as pushforward measures}
\label{SS:balanced-pushforward-measure}

Recall that we assume $X$ to be a topologically transitive, $C$-balanced subshift, with $X'$ denoting the set of transitive points.
Let $\varrho$ be a fully supported ergodic measure on $(X,S)$. Then, $X'$ is a set of full measure for $\varrho$, and $\varrho$-almost every $x \in X'$ is generic. 
For such points, the Rauzy measure is (up to translation) given by the pushforward of $\varrho$ under the map $\phi$, defined in Section~\ref{SUBSEC:generic-factors}.

\begin{prop}
\label{PROP:rauzy-measure-dynamical}
    For every $\varrho$-generic point $x$, we have
    $
    \mu(x) \circ t_{\phi(x)}^{-1} = \varrho \circ \phi^{-1},
    $
    where $t_{\phi(x)} \colon x \mapsto x + \phi(x)$.
\end{prop}

\begin{proof}
    We regard $\varrho$ as an ergodic measure on the (non-compact) dynamical system $(X',S)$. Let $x$ be a $\varrho$-generic point and define 
    \[
    \varrho_n = \frac{1}{n} \sum_{k=0}^{n-1} \delta_{S^k x},
    \]
    for all $n \in \N$. Note that the weak convergence of $\varrho_n$ persists under the restriction to the (full measure) set $X'$. Since $\phi$ is continuous on $X'$, this implies that $ \varrho_n \circ \phi^{-1}$ converges to $ \varrho \circ \phi^{-1}$ in the weak topology. Since $\phi(S^k x) = \phi(x) + \pi(\psi(x_{[1,k]}))$, we obtain by definition of the Rauzy measure
    \[
    \mu(x) \circ t_{\phi(x)}^{-1} = \lim_{n \to \infty} \frac{1}{n} \sum_{k = 0}^{n-1} \delta_{\phi(S^k x)}
    = \lim_{n \to \infty} \varrho_n \circ \phi^{-1}
    = \varrho \circ \phi^{-1},
    \]
    and the claim follows.
\end{proof}

\begin{remark}
    For generic points $x$, this offers an additional interpretation of the Lebesgue covering property. In fact, the restriction of $\mu_{\mc J}(x)$ to a fundamental domain coincides with the natural projection of $\mu(x)$ to $\mathbb{H}/\mc J$ under $\pi_{\mc J}$. Since $\varphi = \pi_{\mc J} \circ \phi$, this is (up to translation) the same as $\mu \circ \varphi^{-1}$ by Proposition~\ref{PROP:rauzy-measure-dynamical}. On the other hand, by Corollary~\ref{COR:generic-factor}, $\mu \circ \varphi^{-1}$ is the unique $T$-invariant probability measure on the group $G$. Because we assumed that the entries of the right eigenvector are rationally independent, the group $G$ is in fact the full $d-1$ dimensional torus, and the Haar measure is an appropriate multiple of Lebesgue measure.
\end{remark}

\section{Rauzy fractals of random substitutions}\label{S: Rauzy random subs}

Here, we present two methods of canonically associating a Rauzy fractal with a given compatible, irreducible Pisot random substitution.
We provide a construction in \Cref{SS:Rauzy-sequence-definition} via the set of \emph{level-$\infty$ inflation words}. The elements of this set can be viewed as analogues of substitution-fixed points for random substitutions. Utilising the results proved in \Cref{S:Bal-seq-Rauzy} on Rauzy fractals of $C$-balanced sequences, we show that every element in this set with dense shift-orbit gives rise to the same Rauzy fractal.
In \Cref{SS:Rauzy-language-definition}, we provide an alternative construction, via the \emph{prefix language} of the random substitution.
We prove that these two constructions coincide in \Cref{SS:Rauzy-construction-equivalence}.

\subsection*{Notation}

The Rauzy fractal associated with a (compatible, irreducible Pisot) random substitution $(\vartheta, \bm P)$ does not depend on the explicit choice of (non-degenerate) probabilities $\bm P$.
As such, we suppress the dependence on $\bm P$ throughout this section, and simply refer to $\vartheta$ as a random substitution.
To further simplify our notation, we let $(\lambda,\mathbf{L},\mathbf{R})$ denote the Perron--Frobenius data of $\vartheta$ (which is independent of $\bm P$ by compatibility) and let $\pi$ denote the projection along $\mathbf{R}$ onto $\mathbb{H}$. 
Recall from \Cref{THM:Pisot-balanced} that the assumptions on $\vartheta$ guarantee that every sequence in $X_{\vartheta}$ is $C$-balanced, such that the results from \Cref{S:Bal-seq-Rauzy} are applicable in the present setting.

\subsection{Construction via sequences with dense shift-orbit}\label{SS:Rauzy-sequence-definition}

Given a random substitution $\vartheta$ with associated subshift $X_{\vartheta}$, we define the \emph{level-$\infty$ inflation set} by
\begin{align*}
    X_{\vartheta}^{\infty} = \bigcap_{n \in \N_0} \vartheta^{n} (X_{\vartheta}) .
\end{align*}
It follows from Cantor's intersection theorem that the set $X_{\vartheta}^{\infty}$ is always non-empty, noting that \mbox{$\vartheta^n (X_{\vartheta}) \subseteq \vartheta^{n-1} (X_{\vartheta})$} for all $n \in \N$ and that the sets $\vartheta^n (X_{\vartheta})$ are compact by \Cref{LEM:subst-image-compact}. 

As $\vartheta$ is a compatible, irreducible Pisot random substitution, any marginal is an irreducible Pisot substitution, and hence the substitution-fixed point of any marginal belongs to $X_{\vartheta}^{\infty}$.  Moreover, one can construct a point in $X_{\vartheta}^{\infty}$ with a dense shift-orbit in $X_\vartheta$, see for instance \cite[Proposition~13]{rust-spindeler}.

We define $\mc R_{\vartheta,a} = \mc R_a(w)$, where $w \in X_\vartheta^\infty$ is any element with dense shift-orbit in $X_\vartheta$. The \emph{Rauzy fractal associated with the random substitution $\vartheta$} is the set $\mc R_\vartheta = \cup_{a \in \mc A} \mc R_{\vartheta,a}$. It will soon be shown in \Cref{SS:Rauzy-construction-equivalence} that $\mc R_{\vartheta,a}$ (and therefore $\mc R_\vartheta$) is independent of the choice of element $w$ with this property. 
In the special case that $\vartheta$ is deterministic (that is, $\# \vartheta(a) = 1$ for all $a \in \mc A$), the set $\mc R_{\vartheta}$ coincides with the standard notion of a Rauzy fractal in the context of substitutions. 

\subsection{Construction via the prefix language}\label{SS:Rauzy-language-definition}

An alternative method of associating a Rauzy fractal with a random substitution is via the \emph{prefix language} of the random substitution---a similar construction has been given previously in the S-adic setting \cite[Section~2.5]{BMST16}.
It will be shown in \Cref{SS:Rauzy-construction-equivalence} that this construction is equivalent to the one given in \Cref{SS:Rauzy-sequence-definition}. 
For now, we differentiate the two constructions by placing a hat $\widehat{\phantom{\mathcal{R}}}$ above any alternate versions.

In what follows, we let $\mathcal{A}_1$ denote the set of \emph{eventually first} letters: namely, $a \in \mathcal{A}_1$ if and only if there is a letter $b \in \mathcal{A}$ such that $a$ is the first letter of some realisation of $\vartheta^{n} (b)$ for infinitely many $n \in \N$.

\begin{definition}
Let $\vartheta$ be a random substitution.
We define the \emph{prefix language} of $\vartheta$ by
\begin{align*}
    \mc L_{\vartheta, p} = \{ v \in \mathcal{A}^{+} : v \text{ is a prefix of some } w \in \vartheta^{k} (a), \, k \in \N_0, \, a \in \mathcal{A}_1 \} .
\end{align*}
The Rauzy fractal associated with the random substitution $\vartheta$ is defined from the projection of the set of all possible abelianisation vectors arising from words in $\mc L_{\vartheta, p}$. For each $a \in \mathcal{A}$, let
\begin{align*}
    \mathcal{S}_a (\mc L_{\vartheta, p}) = \{ \psi (v) : v a \in \mc L_{\vartheta,p} \},
\end{align*}
and define $\widehat{\mc R}_{\vartheta, a}^{*} = \pi (\mathcal{S}_a (\mc L_{\vartheta, p}))$ and $\widehat{\mc R}_{\vartheta, a} = \overline{\widehat{\mc R}_{\vartheta, a}^{*}}$. The \emph{(prefix) Rauzy fractal associated with the random substitution $\vartheta$} is the set $\widehat{\mc R}_{\vartheta} = \cup_{a \in \mathcal{A}} \widehat{\mc R}_{\vartheta, a}$.
\end{definition}

\subsection{Equivalence of the two constructions}\label{SS:Rauzy-construction-equivalence}

If $x \in X_{\vartheta}^{\infty}$, then $x_{[1,n]} \in \mc L_{\vartheta, p}$ for all $n \in \mathbb{N}_{0}$, and hence
\begin{align}\label{EQ: Rauzy inf-inflation characterisation}
    \mathcal{S}_{a} (x)
    \subseteq
    \mathcal{S}_{a} (\mc L_{\vartheta,p}).
\end{align}
Projecting this identity gives that $\mathcal{R}_a (x)
\subseteq \widehat{\mathcal{R}}_{\vartheta,a}$.
Thus, the Rauzy fractal corresponding to any sequence in $X_{\vartheta}^{\infty}$ is contained in the Rauzy fractal of $\vartheta$ defined via the prefix language. If the sequence is additionally assumed to have a dense shift-orbit, then this inclusion is an equality.

\begin{theorem}\label{THM: r-Rauzy equiv char}
Let $\vartheta$ be a compatible, irreducible Pisot random substitution and let $w \in X_{\vartheta}^{\infty}$ be an element with dense shift-orbit in $X_{\vartheta}$.
Then, for all $a \in \mathcal{A}$, we have $\mathcal{R}_{a} (w) = \widehat{\mathcal{R}}_{\vartheta, a}$.
\end{theorem}

\begin{proof}
We first observe that there is a point $\widetilde{w} \in X_{\vartheta}^\infty$ with the following properties.
\begin{itemize}
\item For each $n \in \N$, $\widetilde{w}$ can be decomposed as a concatenation of level-$n$ inflation words.
\item In each of these decompositions, every word $u \in \vartheta^n(a)$ appears, for all $a \in \mc A$.
\item The shift-orbit of $\widetilde{w}$ is dense.
\end{itemize}
The construction of such a point is given in  \cite[Proposition~13]{rust-spindeler}, where we also observe that, although it is not explicitly stated, the point constructed  belongs to $X_{\vartheta}^{\infty}$. 

Our goal now is to show that $\mc R_a(\tilde{w}) = \widehat{\mc R}_{\vartheta,a}$. By \eqref{EQ: Rauzy inf-inflation characterisation}, we have that $\mc R_a(\widetilde{w}) \subset \widehat{\mc R}_{\vartheta,a}$ for all $a \in \A$. So, if we can show that $\mc R_a^\ast(\widetilde{w})$ is dense in $\widehat{\mc R}_{\vartheta,a}$, then we can conclude that $\mc R_a(\widetilde{w}) = \widehat{\mc R}_{\vartheta,a}$. To this end, let $q \in \widehat{\mc R}_{\vartheta,a}^\ast$ and $\varepsilon > 0$. There exists a word $va \in \mc L_{\vartheta,p}$ such that $q = \pi \circ \psi(v)$, that is, $va$ is a prefix of some word $u \in \vartheta^k(b)$, for some $k \in \N$ and $b \in \mc A$. In fact, we can choose $k$ arbitrarily large. 
Let $k \in \N$ be such that $h_{\vartheta}^k(\widehat{\mc R}_{\vartheta,a})$ is contained in a ball of radius $\varepsilon$ centered at the origin. 
This is possible because $h_{\vartheta}$ is a contractive linear transformation. Since $u$  appears in the level-$k$ inflation word decomposition of $\widetilde{w}$, there exists a word $v' \in \mc L_{\vartheta,p}$ such that $\widetilde{w} = \vartheta^k(v') va \cdots$. Hence, $M^k \psi(v') + \psi(v) \in \mc S_a(\widetilde{w})$ and thereby, $h_{\vartheta}^k(\pi\circ\psi(v')) + q \in \mc R_a^{\ast}(\widetilde{w})$. By the assumption on $k$, this point lies in the $\varepsilon$-neighborhood of $q \in \widehat{\mc R}_{\vartheta,a}^\ast$. Since $q$ was arbitrary, this shows that $\mc R_a^\ast(\widetilde{w})$ lies dense in $\widehat{\mc R}_{\vartheta,a}$.

Now, let $w \in X_{\vartheta}^\infty$. Since $w$ and $\widetilde{w}$ both have dense shift-orbit in $X_{\vartheta}$, it follows by \Cref{CO: dense-shift-Rauzy} that there exists a vector $t$ such that $\mathcal{R}_a (w) = \mathcal{R}_a (\widetilde{w}) - t$ for all $a \in \mathcal{A}$. But since $w \in X_{\vartheta}^{\infty}$, we have $\mathcal{R}_a (w) \subseteq \widehat{\mathcal{R}}_{\vartheta,a} = \mc R_a(\widetilde{w})$, so we must have $t = 0$. This completes the proof.
\end{proof}

Therefore, we may now refer to both constructions of the Rauzy fractal by $\mathcal{R}_{\vartheta}$ without ambiguity.

\subsection{Topological properties}\label{SS:Rauzy-top-properties}

Using the characterisation given by \Cref{THM: r-Rauzy equiv char}, together with the analytic properties of Rauzy fractals corresponding to $C$-balanced sequences proved in \Cref{S:Bal-seq-Rauzy}, we can obtain several key topological properties of Rauzy fractals of random substitutions. 

\begin{prop}\label{PROP: marginal Rauzy}
Let $\vartheta$ be a compatible, irreducible Pisot random substitution. If $\theta$ is a marginal of $\vartheta^{k}$ for some $k \in \N$, then $\mathcal{R}_{\theta,a} \subseteq \mathcal{R}_{\vartheta,a}$ for all $a \in \mathcal{A}$.
\end{prop}
\begin{proof}
Let $w \in X_{\vartheta}^{\infty}$ be an element with a dense shift-orbit in $X_{\theta}$ (for example, take a $\theta$-fixed point). By \Cref{THM: r-Rauzy equiv char}, we have that $\mathcal{R}_a (w) = \mathcal{R}_{\theta,a}$ for all $a \in \mathcal{A}$, and applying the projection $\pi$ to both sides in \eqref{EQ: Rauzy inf-inflation characterisation} gives $\mathcal{R}_a (w) \subseteq \mathcal{R}_{\vartheta,a}$, so the result follows.
\end{proof}

\begin{prop}\label{PROP: Rauzy marginals convergence}
Let $\vartheta$ be a compatible, irreducible Pisot random substitution.
There exists a sequence of marginals $\theta_n$ of $\vartheta^{n}$ such that $\mathcal{R}_{\theta_n} \rightarrow \mathcal{R}_{\vartheta}$ in the Hausdorff metric.
\end{prop}

\begin{proof}
Let $w \in X_{\vartheta}^{\infty}$ be an element with dense shift-orbit. For each $n \in \N$, let $a_n$ be a letter such that $w_{[1,\lvert \vartheta^n(a_n) \rvert]} \in \vartheta^n (a_n)$. Such a letter always exists since $w \in X_{\vartheta}^{\infty}$. Thus, there exists a marginal $\theta_n$ of $\vartheta^n$ such that $\theta_n (a_n) = w_{[1,\lvert \vartheta^n (a_n) \rvert]}$. 
In particular, there exists a sequence of words $(w^{(n)})_n$ such that $w^{(n)} \rightarrow w$, for which $w^{(n)} \in X_{\theta_n}$ for all $n \in \N$. Hence, it follows by \Cref{LEM:limit-subset} and \Cref{THM: r-Rauzy equiv char} that, for each $n \in \N$, there is a vector $t_n \in \mathbb{R}^d$ such that
\begin{align*}
    \mathcal{R}_a (w^{(n)}) \subseteq \mathcal{R}_{\theta_n,a} + t_n \subseteq \mathcal{R}_{\vartheta,a} + t_n ,
\end{align*}
where the second inclusion follows by \eqref{EQ: Rauzy inf-inflation characterisation}. Since $w^{(n)} \rightarrow w$ and $w$ has dense shift-orbit in $X_{\vartheta}$, \Cref{PROP:Hausdorff-convergence} gives that $\mathcal{R}_a (w^{(n)}) \rightarrow \mathcal{R}_a (w) = \mathcal{R}_{\vartheta,a}$ as $n \rightarrow \infty$. Therefore, we have that $t_n \rightarrow 0$, hence, we conclude that $\mathcal{R}_{\theta_n,a} \rightarrow \mathcal{R}_{\vartheta,a}$ as $n \rightarrow \infty$.
\end{proof}

These results can be used to transfer some topological properties of Rauzy fractals for deterministic substitutions to the random setting.
In the following two results, we apply \Cref{PROP: deterministic-rauzy-properties} and therefore also assume that our random substitutions are unimodular.

\begin{corollary}
Let $\vartheta$ be a compatible, unimodular, irreducible Pisot random substitution.
The Rauzy fractal $\mathcal{R}_{\vartheta}$ is the closure of its interior.
\end{corollary}

\begin{proof}
Let $\varepsilon > 0$ and let $x \in \mc R_{\vartheta}$.
By \Cref{PROP: Rauzy marginals convergence}, there exists a sequence of marginals $\theta_n$ of $\vartheta^n$ such that $\mc R_{\theta_n} \rightarrow \mc R_{\vartheta}$ in the Hausdorff metric.
Hence, there exists an $n \in \N$ and $y \in \mc R_{\theta_n}$ such that $\lvert x - y \rvert < \varepsilon / 2$.
Since $\mc R_{\theta_n}$ is the closure of its interior, there is a $z \in \operatorname{int} (\mc R_{\theta_n})$ such that $\lvert y - z \rvert < \varepsilon / 2$; hence, $\lvert x - z \rvert < \varepsilon$.
By \Cref{PROP: marginal Rauzy}, we have $\mc R_{\theta_n} \subseteq \mc R_{\vartheta}$, so it follows that $z \in \operatorname{int} (\mc R_{\vartheta})$ and we conclude that $\mc R_{\vartheta} = \overline{\operatorname{int} (\mc R_{\vartheta})}$.
\end{proof}

\subsection{GIFS}\label{SS:Rauzy-GIFS}

An important feature of Rauzy fractals associated with substitutions is that their subtiles are the attractors of a graph-directed iterated function system (GIFS) that arises naturally from the substitution action.
Graph-directed iterated function systems were first studied in the seminal paper of Mauldin and Williams~\cite{mauldin-williams} and the GIFS structure of Rauzy fractals was identified by Sirvent and Wang~\cite{sirvent-wang}.
In particular, the underlying graph is the \emph{prefix-suffix graph}, which was first introduced by Canterini and Siegel~\cite{canterini-siegel} and encodes the different positions in which letters can appear in the inflated image of a letter.
Here, we prove an analogous result for Rauzy fractals of random substitutions. 
We first generalise the definition of the prefix-suffix graph to random substitutions.

\begin{definition}
Let $\vartheta$ be a random substitution and let $\mathcal{P}_{\vartheta}$ be the finite set defined by
\begin{align*}
    \mathcal{P}_{\vartheta} = \left\{ (p,a,s) \in \mathcal{A}^{*} \times \mathcal{A} \times \mathcal{A}^{*} : \text{ there exists $b \in \mathcal{A}$ such that } p a s \in \vartheta (b) \right\} .
\end{align*}
The \emph{prefix-suffix graph} of $\vartheta$ is the finite directed graph $\Gamma_{\vartheta}$ with vertex set $\mathcal{A}$ such that there is an edge labelled by $(p,a,s) \in \mathcal{P}_{\vartheta}$ from $a$ to $b$ if $p a s \in \vartheta (b)$.
\end{definition}

\begin{theorem}\label{THM: r-Rauzy GIFS}
Let $\vartheta$ be a compatible, irreducible Pisot random substitution.
For each $a \in \mc A$, the subtile $\mc R_{\vartheta, a}$ is the unique (non-empty) compact solution of the self-consistency relation
\begin{align}\label{EQ:GIFS-sets}
    \mathcal{R}_{\vartheta,a} = \bigcup_{b \in \mathcal{A}} \bigcup_{\substack{(p,a,s) \in \mathcal{P}_{\vartheta} \\ p a s \in \vartheta (b)}} h(\mathcal{R}_{\vartheta,b}) + \pi (\psi(p)).
\end{align}
\end{theorem}

\begin{proof}
For each $a \in \mathcal{A}$, the vectors $q \in \mathcal{S}_a (\mc L_{\vartheta,p}) = \{ \psi (v) : v a \in \mc L_{\vartheta,p} \}$ are precisely the vectors of the form $q = M r + \psi(p)$, where $r \in \mathcal{S}_b (\mc L_{\vartheta,p})$ for some $b \in \mathcal{A}$ and $p a s \in \vartheta (b)$. Thus,
\begin{align*}
    \mathcal{S}_a (\mc L_{\vartheta,p}) = \bigcup_{b \in \mathcal{A}} \bigcup_{\substack{(p,b,s) \in \mathcal{P}_{\vartheta} \\ p a s \in \vartheta (b)}} M(\mathcal{S}_b (\mc L_{\vartheta,p})) + \psi (p),
\end{align*}
and projecting to $\mathbb{H}$ gives
\begin{align}\label{EQ: gifs proof eq}
    \mathcal{R}^{*}_{\vartheta,a} = \bigcup_{b \in \mathcal{A}} \bigcup_{\substack{(p,b,s) \in \mathcal{P}_{\vartheta} \\ p a s \in \vartheta (b)}} \pi \circ M (\mathcal{S}_b (\mc L_{\vartheta,p})) + \pi (\psi (p)).
\end{align}
Since $\pi \circ M = h \circ \pi$, it follows from \eqref{EQ: gifs proof eq} that
\begin{align*}
    \mathcal{R}^{*}_{\vartheta,a} = \bigcup_{b \in \mathcal{A}} \bigcup_{\substack{(p,b,s) \in \mathcal{P}_{\vartheta} \\ p a s \in \vartheta (b)}} h(\mathcal{R}^{*}_{\vartheta,b}) + \pi (\psi (p)).
\end{align*}
Taking closure gives \eqref{EQ:GIFS-sets}.

The relation in \eqref{EQ:GIFS-sets} is in fact the self-consistency relation for a self-similar set $(\mc R_{\vartheta,a})_{a \in \mc A}$ of a GIFS $\mc G = \mc G(\vartheta) = (G,(X_a)_{a \in \mc A},(f_e)_{e \in E})$.
Here, $G = \Gamma_{\vartheta}$, the prefix-suffix graph, and $X_a = \mathbb H$ for all $a \in \mc A$.
For each edge $e = (p,b,s)$ in $\Gamma_{\vartheta}$, the corresponding contraction is given by $f_e = f_p$, where 
    \begin{align*}
    f_p \colon x \mapsto h(x) + \pi \circ \psi(p).
    \end{align*}
Thus, it follows that the sets $(\mc R_a)_{a \in \mc A}$ are in fact the unique, non-empty compact sets satisfying \eqref{EQ:GIFS-sets}---see, for instance, \cite{mauldin_urbanski_2003}.
\end{proof}

It should be emphasised that, in contrast to the case for deterministic substitutions, the sets $\mc R_{\vartheta,a}$ are in general no longer measure-disjoint for different choices of $a \in \mc A$.

\section{Measures on Rauzy fractals of random substitutions}\label{S: measures}

In this section, we regard $\vartheta_{\bm P}$ properly as a random substitution (equipped with probabilities). 
We equip the subtile $\mc R_{\vartheta,a}$ with a natural measure, called the \emph{Rauzy measure associated with $\vartheta$ and $a$}. 
This object can be defined in several different ways, which we show all coincide.

Given a finite word $w$, we associate with the set $\mathcal R_a^{\ast}(w)$ a corresponding counting measure
\begin{align*}
\mu_a(w) =  \sum_{x \in \mathcal R_a^{\ast}(w)} \delta_x
 = \sum_{ y \in \mathcal S_a(w)} \delta_{\pi(y)}.
\end{align*}
Here the last step holds because $\pi$ restricted to the integer lattice is injective, and here $\mc R_a^{\ast}(w)$ is as defined in \Cref{SS:Rauzy-construction}.
This is consistent with the notation for $\mu(w)$ employed in \Cref{SS:rauzy-measures} in the sense that $\mu(w) = \sum_{a \in \mc A} \mu_a(w)$. 
This relation also carries over to an infinite sequence $w$ if we set $\mu_a(w) =  \lim_{n\to \infty} n^{-1} \mu_a(w_{[1,n]})$, provided that the limit exists. Note, if $w$ is a \emph{random} word (or random sequence), the expression $\mu_a(w)$ is to be interpreted as a random measure. 
In general, it will depend on the choice of $w \in X_{\vartheta}$, but once a set of probability parameters $\mathbf{P}$ is fixed, almost all sequences will essentially give rise to the same measure. 

In the following, let $\vartheta_{\mathbf{P}}$ be a compatible, irreducible Pisot random substitution. 
Recall that $\nu$ denotes the $\vartheta_{\mathbf{P}}$-invariant measure on $X_{\vartheta}$ and $\varrho$ is the corresponding (ergodic) frequency measure. 
Also, we let $\phi$ be the map from the set of transitive points to the Rauzy fractal, defined in \Cref{SUBSEC:generic-factors}. 
This is well-defined up to a reference point $w'$ with $\phi(w') = 0$.
Recall that $R_a$ is the entry of the right Perron--Frobenius eigenvector $\mathbf{R}$ associated with the letter $a$.
We summarise the main results of this section in the following.

There is a vector of measures $\overline{\mu} = (\overline{\mu}_a)_{a \in \mc A}$, depending on $\vartheta_{\mathbf{P}}$, with the following properties for all $a \in \mc A$.
\begin{enumerate}
    \item $\overline{\mu}_a$ is supported on $\mc R_{\vartheta,a}$ and $\overline{\mu}_a(\mc R_{\vartheta,a}) = R_a$.
    \item $\overline{\mu}_a = \mu_a(w) $ for $\nu$-almost all $w \in X_{\vartheta}$ (in particular, the limit $\mu_a(w)$ exists).
    \item $\overline{\mu}_a = \lim_{n \to \infty} |\vartheta^n(v)|^{-1}\mu_a(\vartheta_{\mathbf{P}}^n(v))$ holds $\mathbb{P}$-almost surely for all $v \in \mc A^+$.
    \item $\overline{\mu}_a$ is absolutely continuous relative to Lebesgue measure, with bounded density.
    \item $(R_a^{-1} \overline{\mu}_a)_{a \in \mc A}$ is self-similar for the GIFS in \Cref{SS:Rauzy-GIFS}, equipped with appropriate probabilities.
\end{enumerate}
Given an appropriate reference point $w' \in X_{\vartheta}$ with $\phi(w') = 0$, we further have
\begin{enumerate}
    \setcounter{enumi}{5}
    \item $\mu = \sum_{a \in \mc A} \overline{\mu}_a$ coincides with $\mu(w)$  for $\varrho$-almost every $w \in X_{\vartheta}$, up to a translation by $\phi(w)$.
    \item $\mu$ is the pushforward of $\varrho$ under $\phi$.
\end{enumerate}
Finally, for a fixed (set-valued) compatible, irreducible Pisot random substitution $\vartheta$, 
\begin{enumerate}
    \setcounter{enumi}{7}
    \item $\overline{\mu}_a$ depends (weakly-)continuously on the probability parameters $\mathbf{P}$.
\end{enumerate}
The first item is an immediate consequence of the following two, which, in turn, will be proved in \Cref{PROP:nu-as-convergence} and \Cref{PROP:as-convergence-markov}, respectively. 
We obtain the absolute continuity of the measures from the Lebesgue covering property discussed in \Cref{SS:Lebesgue-covering-RS}. 
The interpretation in terms of a GIFS is provided in \Cref{PROP:measure-GIFS} and we will use this to obtain the continuous dependence on the probability parameters.
The total mass of $\bar{\mu}_a$ is made precise in \Cref{COR:Lebesgue-sum}.
Items (7) and (8) follow from the general framework on balanced sequences as we will show in \Cref{SS:frequency-measure-image}.

Our first aim is to determine how $\mu_a(w)$ changes as we apply $\vartheta_{\bm P}$. 
We develop in parallel the prerequisites to show that there is a Rauzy measure that emerges both by following the expected measure as well as by following a generic path of the random variables $(\vartheta_{\bm P}^n(v))_{n \in \N}$.
This framework follows a similar inflation word formalism established in \cite{gohlke} for topological entropy of random substitutions and expanded in \cite{MT-entropy} for measure theoretic entropy, whereby we will establish an ``inflation word version" of the Rauzy measure and subsequently show this to be the same as the one obtained as the Rauzy measure (as described in Section \ref{SS:rauzy-measures}) of a generic infinite sequence.

\begin{lemma}
\label{LEM:measure-substitution}
For every $w \in \mc A^{\ast}$, we have
\begin{align*}
\mu_a (\vartheta_{\bm P} (w)) = \sum_{b \in \mathcal A}  \sum_{x \in \mathcal R^{\ast}_b(w)} \mu_a(\vartheta_{\bm P} (b)) \ast \delta_{h x},
\end{align*}
where the occurrence of $\mu_a (\vartheta_{\bm P} (b))$ is to be understood as an independent random measure in each summand.
\end{lemma}

\begin{proof}
We recall that for $a \in \mathcal{A}$ and $v = v_1 \cdots v_n \in \mathcal{A}^n$,
\begin{align*}
    \mathcal{S}_a (v) = \{ \psi (v_{[1,m]}) : m \in \{0, 2, \dots, n-1 \} \ \text{and} \ v_{m+1} = a \} .
\end{align*}
For $w = w_1 \cdots w_n \in \mathcal{A}^n$ we obtain $\vartheta_{\bm P} (w) = \vartheta_{\bm P} (w_1) \cdots \vartheta_{\bm P} (w_n)$, and hence
\begin{align*}
\mathcal S_a (\vartheta_{\bm P} (w)) = \bigcup_{k = 1}^{n} \left( \mathcal S_a(\vartheta_{\bm P} (w_k) ) + \psi(\vartheta_{\bm P} (w_{[1,k-1]})) \right),
\end{align*}
to be understood as a union of (independent) random sets. 
Recall that $\psi(\vartheta_{\bm P} (w_{[1,k-1]})) = M \psi(w_{[1,k-1]})$, where $M$ is the matrix of $\vartheta_{\bm P}$, and therefore
\begin{align*}
\mathcal R^{\ast}_a (\vartheta_{\bm P} (w)) = 
\bigcup_{k=1}^{n} \left( \mathcal R^{\ast}_a (\vartheta_{\bm P} (w_k)) + h ( \pi \circ \psi(w_{[1,k-1]})) \right).
\end{align*}
This gives
\begin{align*}
\mu_a (\vartheta_{\bm P} (w)) =  \sum_{k = 1}^{n} \mu_a(\vartheta_{\bm P} (w_k)) \ast \delta_{h ( \pi \circ \psi(w_{[1,k-1]}))}.
\end{align*}
Recall that $w_k = b$ precisely if $\pi \circ \psi (w_{[1,k-1]}) \in \mathcal R^{\ast}_b(w)$.
We can therefore reformulate the last equation as
\[
\mu_a (\vartheta_{\bm P} (w)) = \sum_{b \in \mathcal A}  \sum_{x \in \mathcal R^{\ast}_b(w)} \mu_a (\vartheta_{\bm P} (b)) \ast \delta_{h x}.\qedhere
\]
\end{proof}

In order to formulate the random GIFS later on, we need one more ingredient. 

\begin{lemma}
\label{LEM:f_p-measure}
For every $w \in \mathcal A^+$,
\begin{align*}
\mu_a (\vartheta_{\bm P} (w)) = \sum_{b \in \mathcal A} \sum_{x \in \mathcal R_b^{\ast}(w)} \sum_{p,\, pas \in \vartheta (b)} \delta_x \circ f_p^{-1},
\end{align*}
where $f_p \colon x \mapsto h(x) + \pi \circ \psi(p)$.
\end{lemma}

\begin{proof}
We can be more explicit about $\mu_a (\vartheta_{\bm P} (b))$, which can be rewritten as
\begin{align*}
\mu_a (\vartheta_{\bm P} (b)) = \sum_{p, \, pas = \vartheta(b)} \delta_{\pi  \circ \psi(p)},
\end{align*}
yielding
\begin{align*}
\mu_a (\vartheta_{\bm P} (w)) = \sum_{b \in \mathcal A} \sum_{x \in \mathcal R_b^{\ast}(w)} \sum_{p,\, pas = \vartheta(b)} \delta_{hx + \pi \circ \psi(p)}.
\end{align*}
With the definition of $f_p$ as above, we obtain 
\[
\mu_a (\vartheta_{\bm P} (w)) = \sum_{b \in \mathcal A} \sum_{x \in \mathcal R_b^{\ast}(w)} \sum_{p,\, pas = \vartheta(b)} \delta_x \circ f_p^{-1}.\qedhere
\]
\end{proof}

\subsection{Expected values}
\label{SUBSEC:expected_values}

Recall that by random word, we mean a random variable, whose co-domain is $\mathcal{A}^{\ast}$.  In the following, for a random word $w$, we set $\overline{\mu}_a(w) \coloneqq \mathbb{E} [\mu_a (w)]$, where the expectation is taken with resect to the distribution of the random word $w$. 

\begin{lemma}
\label{COR:average-iteration}
For every random word $w$ of fixed length, we have
\begin{align*}
\overline{\mu}_a (\vartheta_{\bm P} (w)) = \sum_{b \in \mathcal A} \overline{\mu}_a (\vartheta_{\bm P} (b)) \ast \left( \overline{\mu}_b(w) \circ h^{-1} \right). 
\end{align*}
\end{lemma}

\begin{proof}
Taking expectation values in \Cref{LEM:measure-substitution} and using the Markov property of the substitution action,
\[
\mathbb{E}[\mu_a (\vartheta_{\bm P} (w))] =  \sum_{b \in \mathcal A} \overline{\mu}_a(\vartheta_{\bm P} (b)) \ast \mathbb{E} \left[ \sum_{x \in \mathcal R^{\ast}_b(w)} \delta_{h_{\vartheta} x} \right]
= \sum_{b \in \mathcal A} \overline{\mu}_a(\vartheta_{\bm P} (b)) \ast \left( \overline{\mu}_b(w) \circ h^{-1} \right). \qedhere
\]
\end{proof}

The following should be compared with the set valued recursion in \Cref{THM: r-Rauzy GIFS}. 
Up to a normalisation factor, this will be our measure analogue of the GIFS.

\begin{lemma}
\label{COR:measure-gifs-iteration}
For every random word $w$ of fixed length, we have
\begin{align*}
\overline{\mu}_a (\vartheta_{\bm P} (w)) = \sum_{(p,b,s) \in \Gamma_{\theta}} \mathbb{P}[\vartheta_{\bm P} (b) = pas] \, \overline{\mu}_b(w) \circ f_p^{-1}.
\end{align*}
\end{lemma}

\begin{proof}
This result follows from \Cref{LEM:f_p-measure} and taking expectation.
\end{proof}

Note that the recursion in \Cref{COR:average-iteration} has the structure of a matrix convolution with a vector. 
Following this interpretation and starting from a fixed word $v \in \mc A^{\ast}$, we define the measure-valued vector 
\begin{align*}
\overline{\mu}^n(v) \coloneqq (\overline{\mu}_a(\vartheta_{\bm P}^n(v)))_{a \in \mathcal A},
\end{align*}
for all $n \in \N_0$.
By construction, for each $a \in \mathcal A$, the measures 
\begin{align*}
\frac{1}{|\vartheta_{\bm P}^n(v)|_a} \overline{\mu}^n_a(v)
\end{align*}
are probability measures on a compact space for each $n \in \N$ and hence have a weak accumulation point. 
In the following, we show that we in fact have convergence.
Let us outline the general idea for the proof. 
Taking Fourier transforms in \Cref{COR:average-iteration}, we obtain a self-consistency relation that involves a matrix multiplication. 
Iterating this relation leads to a matrix cocycle that converges compactly. 

As a first step, we introduce some notation. 
Fix some $v \in \mc A^{\ast}$ and set $F^n_v = (\operatorname{FT}[\overline{\mu}^n_a(v)])_{a \in \mc A}$, where $\operatorname{FT}$ denotes the Fourier transform operator.
Since $\overline{\mu}_a(v)$ is a finite linear combination of Dirac measures on $\mathbb H$, its Fourier transform is represented by a continuous function, taking values in the dual vector space $\mathbb H^{\ast}$. 
The latter can be identified with $\R^{d-1}$. A straightforward calculation yields
\begin{align*}
\operatorname{FT}[\overline{\mu}_a(v) \circ h^{-1}](k) = \operatorname{FT}[\overline{\mu}_a(v)](g k),
\end{align*}
where $g = h^{\ast}$ is the dual of $h$, given by
\begin{align*}
\langle g k, x \rangle
= \langle k, h x \rangle,
\end{align*}
for all $x \in \mathbb H$ and $k \in\mathbb H^{\ast}$.
We also let $B = (B_{ab})_{a,b \in \mc A}$ be the (function-valued) matrix given by
\begin{align*}
B_{ab} = \operatorname{FT} [\overline{\mu}_a (\vartheta_{\bm P} (b))],
\end{align*}
for all $a,b \in \mc A$. 
Then, applying the Fourier transform to the relation in \Cref{COR:average-iteration}, with $w = \vartheta_{\bm P}^n(v)$, we obtain
\begin{align*}
\operatorname{FT}[\overline{\mu}^{n+1}_a(v)](k)
 = \sum_{b \in \mc A} B(k)_{ab}
 \operatorname{FT}[\overline{\mu}^n_b(v)](gk).
\end{align*}
In a more compact manner, this can be written as a cocycle relation over the dynamical system $k \mapsto g k$, given by
\begin{align*}
F^{n+1}_v(k) = B(k) \, F^n_v (g k).
\end{align*}
Iterating this relation gives
\begin{align}
\label{EQ:mu-n-cocycle}
F^n_v(k) = B^{(n)}(k) \, F^0_v(g^n k), 
\quad B^{(n)}(k) = B(k) B(g k) \cdots B(g^{n-1}k).
\end{align}
By construction, the assignments $k \mapsto B(k)$ and $k \mapsto F^0_v(k)$ are both Lipshitz-continuous. 
Since $g$ is a contraction, the iteration $k \mapsto gk$ approaches the fixed point $0$ exponentially fast. 
Since all entries of $B$ and $F^{0}_v$ are defined via the Fourier transform of a measure, evaluation at $0$ gives the total mass of the corresponding measure. Hence,
\begin{align*}
B(0) = M, \quad F^{0}_v(0) = \psi(v).
\end{align*}
Recall that $\mathbf{L}$ and $\mathbf{R}$ are the left and right Perron--Frobenius eigenvectors resp.\ of the matrix $M$.
Following the line of argument in \cite{BG20}, we quickly arrive at the following; compare \cite[Theorem~4.6]{BG20} and the discussion thereafter.

\begin{lemma}
\label{LEM:cocycle-convergence}
The sequence $(\lambda^{-n} B^{(n)}(k))_{n \to \infty}$ converges compactly on $\mathbb H^{\ast}$ to a matrix valued function $C(k)$. 
The limit is of the form
\begin{align*}
C(k) = c(k) \mathbf{L}^{T},
\end{align*}
where $k \mapsto c(k)$ is a continuous vector valued function with $c(0) = \mathbf{R}$.
\end{lemma}

From this, convergence of the measure-valued vectors $\lambda^{-n} \overline{\mu}^{n}(v)$ is obtained as a corollary. 
More precisely, we obtain the following.

\begin{prop}
\label{PROP:limit-measure-vector}
There exists a vector of finite measures $\overline{\mu} = (\overline{\mu}_a)_{a \in \mc A}$ such that
\begin{align*}
\lim_{n \to \infty} \frac{1}{|\vartheta^n(v)|} \overline{\mu}^n(v) 
\coloneqq \lim_{n \to \infty} \frac{1}{|\vartheta^n(v)|} (\overline{\mu}_a (\vartheta_{\bm P}^n (v)))_{a \in \mathcal A}
= \overline{\mu}
\end{align*}
in the sense of weak convergence, for all $v \in \mathcal A^+$. 
Further $\|\overline{\mu}\| = \mathbf{R}$, where $\|\overline{\mu}\| = (\| \overline{\mu}_a \| )_{a \in \mc A}$ and $\| \overline{\mu}_a \| = \overline{\mu}_a(\mathbb{H})$ denotes the total mass of the measure $\overline{\mu}_a$. 
\end{prop}

\begin{proof}
Let us note that 
\begin{align*}
\lim_{n \to \infty} \lambda^{-n} |\vartheta^n(v)|
= \lim_{n \to \infty} \lambda^{-n} \sum_{a \in \mc A} (M^n \psi(v))_a = \mathbf{L}^T \psi(v) =:  L_v.
\end{align*}
Hence, when taking the limit of a product, we may replace $|\vartheta^n(v)|$ by $\lambda^n L_v$. Recall that we denote the Fourier transform of $\overline{\mu}^n(v)$ by $F^n_v$. By \eqref{EQ:mu-n-cocycle}, we obtain via \Cref{LEM:cocycle-convergence}
\begin{align*}
\lim_{n \to \infty} (L_v \lambda^n)^{-1} F^{n}_v(k) 
= L_v^{-1} c(k) \mathbf{L}^T \psi(v) = c(k),
\end{align*}
for all $k \in \mathbb H^{\ast}$. Since $k \mapsto c(k)$ is continuous, Levy's continuity theorem implies that $c(k)$ is the Fourier Transform of some vector valued measure $\overline{\mu}$ such that 
\begin{align*}
\lim_{n\to \infty} \frac{1}{|\vartheta^n(v)|} \overline{\mu}^n(v) 
= \lim_{n \to \infty} (L_v \lambda^n)^{-1} \overline{\mu}^n(v) = \overline{\mu}.
\end{align*}
In particular, $\|\overline{\mu}\| = c(0) = \mathbf{R}$.
\end{proof}

There is an explicit representation of $\overline{\mu}$ as an infinite convolution product, involving measure-valued matrices. Indeed, we may rewrite the relation in \Cref{COR:average-iteration} as 
\begin{align*}
\overline{\mu}^{n+1}(v) = \mc M \ast (\overline{\mu}^n(v) \circ h^{-1}),
\end{align*}
where $\mc M = (\overline{\mu}_a (\vartheta_{\bm P} (b)))_{a,b \in \mc A}$. From this, it is straightforward to verify that $\overline{\mu}$ satisfies the self-similarity relation
\begin{align*}
\overline{\mu} = \frac{1}{\lambda} \mc M \ast (\overline{\mu} \circ h^{-1}).
\end{align*}
Iterating this relation leads to 
\begin{align*}
\overline{\mu} = \mc M^{\infty}\mathbf{R}, 
\quad \mc M^{\infty} = \Conv_{n=0}^{\infty} \frac{1}{\lambda} \mc M \circ h^{-n}.
\end{align*}
Note, the Fourier transform of $\mc M^{\infty}$ is given by the matrix valued function $C(k)$ in \Cref{LEM:cocycle-convergence}. This result also guarantees the existence of the infinite convolution as a limit (via Levy's continuity theorem).
From \Cref{LEM:cocycle-convergence}, we may also infer that $\mc M^{\infty} = m^{\infty} \mathbf{L}^T$, for some measure-valued vector $m^{\infty}$. Combining this with the fact that $\overline{\mu} = \mc M^{\infty}\mathbf{R}$, we have $m^{\infty} = \overline{\mu}$ and hence 
\begin{align*}
\mc M^{\infty} = \overline{\mu}\mathbf{L}^T.
\end{align*}
It is worth noting that sometimes the analysis of the matrix valued convolution $\mc M^{\infty}$ can be reduced to a scalar Bernoulli-like convolution. This is the case for the random Fibonacci substitution, see \cite{baake-spindeler-strungaru}.

The following is finally the measure GIFS relation. 

\begin{prop}\label{PROP:measure-GIFS}
The vector of measures $\overline{\mu}$ satisfies the consistency relation
\begin{align}
\label{EQ:measure-GIFS}
\overline{\mu}_a = \frac{1}{\lambda} \sum_{(p,b,s)} \mathbb{P}[\vartheta_{\bm P} (b) = pas] \, \overline{\mu}_b \circ f_p^{-1}.
\end{align}
\end{prop}

\begin{proof}
Due to \Cref{COR:measure-gifs-iteration}, we obtain
\begin{align*}
\frac{1}{|\vartheta^{n+1}(v)|} \overline{\mu}^{n+1}_a = \frac{|\vartheta^{n}(v)|}{|\vartheta^{n+1}(v)|} \sum_{(p,b,s)} \mathbb{P}[\vartheta_{\bm P} (b) = pas] \, \frac{1}{|\vartheta^n(v)|} \overline{\mu}^n_b \circ f_p^{-1}.
\end{align*}
Sending $n \to \infty$ gives the desired relation due to \Cref{PROP:limit-measure-vector}.
\end{proof}

As observed above, we have $\|\overline{\mu}_a\| = R_{a}$ for all $a \in \mc A$. 
Renormalising via $\widetilde{\mu}_{a} = R_a^{-1} \overline{\mu}_a$, we therefore obtain a vector $(\widetilde{\mu}_a)_{a \in \mc A}$ of probability measures. 
By \eqref{EQ:measure-GIFS}, these measures satisfy the relation
\begin{align*}
\widetilde{\mu}_a = \sum_{(p,b,s)} p^a_{(p,b,s)} \widetilde{\mu}_b \circ f_p^{-1}, \quad p^a_{(p,b,s)} = \frac{1}{\lambda} \mathbb{P}[\vartheta_{\bm P} (b) = pas] \frac{R_b}{R_a}.
\end{align*}
Recall from \Cref{SS:Rauzy-GIFS} that the parameters $(p,b,s)$ label the edges $E$ of an appropriate GIFS $\mc G$. 
In fact, one can verify that $p^a = (p^a_e)_{e \in E(a)}$ defines a probability vector on $E(a)$ for each $a \in \mc A$. Indeed,
\begin{align*}
\sum_{(p,b,s)} p^a_{(p,b,s)} = \frac{1}{\lambda R_a} \sum_{b \in \mc A} R_b \sum_{(p,s)} \mathbb{P}[\vartheta_{\bm P}(b) = pas],
\end{align*}
where
\begin{align*}
\sum_{(p,s)} \mathbb{P}[\vartheta_{\bm P}(b) = pas] = \sum_{v \in \vartheta(b)} \mathbb{P}[\vartheta_{\bm P}(b) = v] |v|_a = \mathbb{E}[|\vartheta_{\bm P}(b)|_a] = M_{ab},
\end{align*}
and hence,
\begin{align*}
\sum_{(p,b,s)} p^a_{(p,b,s)} = \frac{1}{\lambda R_a} \sum_{b \in \mc A}  M_{ab} R_b = 1.
\end{align*}
Using these probabilities, we may hence extend $\mc G$ to a GIFS $\mc G' = (G,(X_a)_{a \in \mc A}, (f_e)_{e \in E}, (p^a)_{a \in \mc A})$ and obtain that $(\widetilde{\mu}_a)_{a \in \mc A}$ is the unique self-similar measure vector for $\mc G'$. It directly follows that $(\overline{\mu}_a)_{a \in \mc A}$ is the unique solution to \eqref{EQ:measure-GIFS}. 
Furthermore, this solution is (globally) attractive in the weak topology and, moreover, the support of the Rauzy measure $\overline{\mu}_a$ is the Rauzy fractal $\mc R_{a}$.

\subsection{Continuity of the invariant measure}

Let $\dmk$ denote the \emph{Monge--Kantorovich metric} on the space of probability measures on a given metric space $X$, defined by
\begin{align*}
    \dmk (\mu,\nu) = \sup_{g \in \operatorname{Lip}_1 (X, \R)} \left\lvert \int g \, d \mu - \int g \, d \nu \right\rvert  ,
\end{align*}
where $\operatorname{Lip}_1 (X,\R)$ denotes the set of Lipschitz functions from $X$ to $\R$ with Lipschitz constant at most $1$.

In the context of iterated function systems, it is known that the corresponding self-similar measure depends continuously on the probability parameters that are assigned to the edges, see for instance \cite[Theorem~3.4]{centore1994continuity}. The proof is  analogous in the graph-directed case and is considered folklore. Since we are not aware of a direct reference, we present a proof below.

\begin{prop}\label{PROP:meas-continuity}
The self-similar measures $(\mu_a)_{a \in \mc A}$ of a GIFS $\mc G = (G,(X_a)_{a \in \mc A}, (f_e)_{e \in E}, (p^a)_{a \in \mc A})$ depend Lipshitz continuously on the probability parameters (with respect to the Monge--Kantorovich metric).
\end{prop}

\begin{proof}
We denote the terminal vertex of an edge $e$ by $t(e)$.
Let $(\mu_a)_{a \in \mc A}$ and $(\mu'_a)_{a \in \mc A}$ be the self-similar measures associated with the probability parameters $(p^a)_{a \in \mc A}$ and $(p'^a)_{a \in \mc A}$ respectively, and assume that $|p^a_e - p'^a_e| \leqslant \varepsilon$ for all $a \in \mc A$ and $e \in E(a)$. 
Note that the definition of the Monge--Kantorovich metric does not change if we take the supremum over the Lipshitz continuous functions $g$ with $0$ in the image of $g$, denoted by $\operatorname{Lip}_{1,0}$. We therefore obtain
    \begin{align*}
        \dmk (\mu_a, \mu'_a)
        &= \sup_{g \in \Lip_{1,0}(X_a, \R)} \Bigg\lvert \int g \, d \bigg( \sum_{e \in E(a)} p^a_e \mu_{t(e)} \circ f_{e}^{-1} \bigg) - \int g \, d \bigg(\sum_{e \in E(a)} p'^a_e \mu'_{t(e)} \circ f_{e}^{-1} \bigg) \Bigg\rvert\\
        &\leqslant  \sum_{e \in E(a)} \sup_{g \in \Lip_{1,0} (X_a, \R)} \Bigg\lvert \int g \circ f_{e} \, d (p^a_e \mu_{t(a)}) - \int g \circ f_{e} \, d (p'^a_e \mu'_{t(e)} ) \Bigg\rvert\\
        &\leqslant  \sum_{e \in E(a)} \sup_{g \in \Lip_{1,0} (X_a, \R)} \Bigg\lvert \int g \circ f_{e} \, d (p^a_e \mu_{t(e)} ) - \int g \circ f_{e} \, d (p^a_e \mu'_{t(e)} ) \Bigg\rvert\\
        &+ \sum_{e \in E(a)} \lvert p^a_e - p'^a_e \rvert \sup_{g \in \Lip_{1,0} (X_a, \R)} \Bigg\lvert \int g \circ f_{e} \, d \mu'_{t(e)} \Bigg\rvert\\
        &\leqslant r \sum_{e \in E(a)} p^a_e \sup_{g \in \Lip_{1,0} (X_a, \R)} \left\lvert \int r^{-1} g \circ f_{e} \, d \mu_{t(e)} - \int r^{-1} g \circ f_{e} \, d \mu'_{t(e)} \right\rvert + c \varepsilon \tc
    \end{align*}
where $c = \operatorname{diam}(X_a) \sum_{e \in E(a)} \mu_{t(e)}(X_{t(e)})$ and $r$ is the common contraction ratio of the maps $(f_e)_{e \in E}$.
    Since every $f_{e}$ has contraction ratio $r$, it follows that $r^{-1} g \circ f_{e} \in \Lip_{1} (X_{t(e)}, \R)$.
    Hence, we obtain from the above that
\begin{align*}
\dmk (\mu_a,\mu'_a) 
\leqslant r \sum_{e \in E(a)} p^a_e \dmk (\mu_{t(e)},\mu'_{t(e)}) + c \varepsilon
\leqslant r \max_{b \in \mc A} \dmk(\mu_b, \mu'_b) + c \varepsilon.
\end{align*}
Taking the maximum and reorganising yields
\[
\max_{a \in \mc A} \dmk (\mu_a,\mu'_a) \leqslant \frac{c \varepsilon}{1-r}.\qedhere
\]
\end{proof}

\begin{corollary}
\label{COR:continuity}
For every compatible, irreducible Pisot random substitution $\vartheta_{\mathbf{P}}$, let $\overline{\mu} = \overline{\mu}(\vartheta,\mathbf{P})$ be the corresponding vector of Rauzy measures. For fixed $\vartheta$, with respect to the Monge--Kantorovich metric, the vector of probability measures $\overline{\mu}(\vartheta,\mathbf{P})$ depends continuously on the probability parameters in $\mathbf{P}$. Since convergence in the Monge--Kantorovich metric implies weak convergence, the same holds with respect to the topology of weak convergence.
\end{corollary}

We note that, although Rauzy measures for random substitutions have been shown above to vary continuously with respect to the choice of probabilities, this is not true for the Rauzy fractals themselves (say with respect to the Hausdorff metric). 
However, the distinction only becomes apparent at the extremal cases; that is, when the probabilities become degenerate.
Indeed, for all non-degenerate probabilities (those for which all $p_{i,j}\neq 0$), the Rauzy fractal is independent of the probabilities.
This follows from the fact that, for non-degenerate random substitutions, the subtiles $\mc R_a$ satisfy the GIFS relations \eqref{EQ:GIFS-sets} of Theorem \ref{THM: r-Rauzy GIFS}, and these relations are independent of the generating probabilities.

On the other hand, for a degenerate choice of probabilities, the subtiles $\mc R_a$ will in general only satisfy a GIFS relation for a random substitution with fewer choices of images (that is, the one given by removing the choice for any images of letters with probability $p_{i,j}=0$). This corresponds to removing edges from the defining graph of the GIFS. Therefore, the support of the measure may become smaller at extremal points. This is observed in the case of the random tribonacci substitution (see Figure \ref{fig:rauzy-measure-trib}). 

\subsection{Almost-sure results}

So far we have obtained a self-similar vector of measures $\overline{\mu} = (\overline{\mu}_a)_{a \in \mc A}$ in the limit of large inflation words under taking expectations. However, since the possible number of inflation words typically grows exponentially with the level, the computation of such expectation values quickly becomes impractical. In this section, we show that $\overline{\mu}$ can alternatively be calculated from a single sequence of inflation words, where convergence holds almost surely. This is useful for concrete numerical approximation of Rauzy measures.

\begin{prop}
\label{PROP:as-convergence-markov}
Let $v \in \mc A^+$. For all $a \in \mc A$,
\begin{align*}
\lim_{n \to \infty} \frac{1}{|\vartheta^n(v)|} \mu_a(\vartheta_{\bm P}^n(v)) = \overline{\mu}_a
\end{align*}
almost surely.
\end{prop}

\begin{proof}
Applying the random substitution $\vartheta_{\bm P}^m$ to the random word $\vartheta_{\bm P}^n(v)$, we obtain from \Cref{LEM:measure-substitution}
\begin{align}
\label{EQ:n+m-relation}
\mu_a(\vartheta_{\bm P}^{n+m}(v)) = \sum_{b \in \mathcal A} \sum_{x \in \mathcal R_b^{\ast}(\vartheta_{\bm P}^n(v))} \mu_a(\vartheta_{\bm P}^m(b)) \ast \delta_{h_{\vartheta^m} x}.
\end{align}
For the Fourier transforms
\begin{align*}
g^n_a \coloneqq \lambda^{-n} \operatorname{FT}[\mu_a(\vartheta_{\bm P}^{n}(v))]
\end{align*}
we obtain from \eqref{EQ:n+m-relation} that
\begin{align}
\label{EQ:g-n+m}
g^{n+m}_a (k) 
= \lambda^{-n} \sum_{b \in \mc A} \sum_{x \in \mc  R_b^{\ast}(\vartheta_{\bm P}^n (v))} f^{a,m,b}_{x,n}(k)  \me^{-2 \pi \im k \cdot h^m_{\vartheta} x},
\end{align}
where for each $a,b \in \mc A$, $m\in \N$ and $k \in \mathbb H^{\ast}$, the family of random variables $\{ f^{a,m,b}_{x,n}\}_{x,n}$ is independent and identically distributed as
\begin{align*}
f^{a,m,b}_{x,n}(k) \sim \lambda^{-m} \operatorname{FT}[\mu_a(\vartheta_{\bm P}^m(b))] (k),
\end{align*}
where $\sim$ denotes equality of distributions. 
Since the language $\mc L_{\vartheta}$ is $C$-balanced, there is a radius $r > 0$ such that $\mc R_b^{\ast}(\vartheta_{\bm P}^n(v))$ is contained in the ball $B_r(0)$ for all $n \in \N$. 
Hence, for all $k$ in a fixed compact set $K \subset \mathbb H^{\ast}$ there is a uniform constant $c = c(K)$ with $0<c<1$ and $d > 0$ such that 
\begin{align*}
\me^{-2 \pi \im k \cdot h^m_{\vartheta} x}
= 1 + r_m(x,k), 
\quad |r_m(x,k)| < d c^m.
\end{align*}
In the next step, we estimate the error that we obtain by replacing the exponential by $1$ in \eqref{EQ:g-n+m}. To this end, observe that
\begin{align*}
|f^{a,m,b}_{x,n}(k)| < f^{a,m,b}_{x,n}(0) = \lambda^{-m}|\vartheta^m(b)|_a = \lambda^{-m} (M^m)_{ab} , 
\end{align*}
and $\# \mc R_b^{\ast}(\vartheta_{\bm P}^n(v)) = |\vartheta_{\bm P}^n(v)|_b = (M^n \psi(v))_b$, yielding
\begin{align*}
& \biggl| \lambda^{-n} \sum_{b \in \mc A} \sum_{x \in \mc  R_b^{\ast}(\vartheta_{\bm P}^n (v))} f^{a,m,b}_{x,n}(k)  r_m(x,k) \biggr| 
 \leqslant \lambda^{-n-m} \sum_{b \in \mc A} \sum_{x \in \mc  R_b^{\ast}(\vartheta_{\bm P}^n (v))} d c^m |\vartheta^m(b)|_a 
\\ & = d c^m \lambda^{-n-m} \sum_{b \in \mc A} (M^m)_{ab} (M^n \psi(v))_b = \mc O(c^m),
\end{align*}
since the entries of $\lambda^{-n}M^n$ are uniformly bounded for all $n \in \N$. 
Combining this with \eqref{EQ:g-n+m} yields
\begin{align*}
g_a^{n+m}(k) =  \mc O(c^m) + \lambda^{-n} \sum_{b \in \mc A} \sum_{x \in \mc  R_b^{\ast}(\vartheta_{\bm P}^n (v))} f^{a,m,b}_{x,n}(k) .
\end{align*}
An appropriate version of the strong law of large numbers (compare for example \cite[Lemma~3]{kurtz1997conceptual}) yields that for each $b \in \mc A$ and every $k \in K$,
\begin{align}
\label{EQ:SLLN-convergence}
\lim_{n \to \infty} \frac{1}{|\vartheta^n(v)|_b} \sum_{x \in \mc  R_b^{\ast}(\vartheta_{\bm P}^n (v))} f^{a,m,b}_{x,n}(k) = \mathbb{E}[f^{a,m,b}_{x,n}(k)] = \lambda^{-m} \operatorname{FT}[\overline{\mu}_a(\vartheta_{\bm P}^m(b))](k)
 = \lambda^{-m} B^{(m)}_{ab}(k)
\end{align}
holds almost surely.

Our next goal is to obtain the same statement with a changed order of the quantifiers: we require that for a set of full probability the convergence holds for all $k \in K$. 
Taking a countable intersection of full measure sets, we easily obtain that almost-surely convergence holds on a dense countable subset of $K$. Due to the Arzel\`{a}-Ascoli theorem, it suffices to show (for each possible realisation) the equicontinuity of the (uniformly bounded) sequence of  functions in \eqref{EQ:SLLN-convergence} in order to obtain uniform convergence on $K$ almost-surely. 
In fact, since any uniform modulus of convergence is preserved under taking averages, it is enough to show the equicontinuity of the family $\{ f^{a,m,b}_{x,n} \}_{x,n}$ for all possible realisations. 
Since the random variables are i.i.d., with only finitely many possible realisations, equicontinuity follows from the fact that each of these realisations is a continuous function on $K$. 
Hence, there exists a set $\Omega(m,K)$ of full $\mathbb{P}$-measure such that the convergence in \eqref{EQ:SLLN-convergence} is uniform on $K$ for all realisations in $\Omega(m,K)$. 
Since $\mathbb H^{\ast}$ is $\sigma$-compact, there is a countable subset $\mc T$ such that
$\mathbb H^{\ast} = \cup_{t \in \mc T} K + t$.
Hence, for all realisations in the full-measure set
\begin{align*}
\Omega = \bigcap_{t \in \mc T} \bigcap_{m \in \N} \Omega(m,K+t),
\end{align*}
we obtain compact convergence in \eqref{EQ:SLLN-convergence} for all $m \in \N$.
Note that 
\begin{align*}
\lim_{n \to \infty} \lambda^{-n} |\vartheta^{n}(v)|_b 
= \lim_{n \to \infty} \lambda^{-n} (M \psi(v))_b = R_b L_v. 
\end{align*}
Thus, for all realisations in $\Omega$, we have
\begin{align*}
\lim_{n \to \infty} \lambda^{-n} \sum_{b \in \mc A} \sum_{x \in \mc  R_b^{\ast}(\vartheta^n (v))} f^{a,m,b}_{x,n}(k) 
= L_v \sum_{b \in \mc A} \lambda^{-m} B^{(m)}(k)_{ab} R_b.
\end{align*}
Due to \Cref{LEM:cocycle-convergence}, we obtain
\begin{align*}
\lim_{m \to \infty} L_v \sum_{b \in \mc A} \lambda^{-m} B^{(m)}(k)_{ab} R_b = L_v \sum_{b \in \mc A} c(k)_a L_b R_b = L_v c(k)_a,
\end{align*}
uniformly on compact sets. 
Hence, for all realisations in $\Omega$, we obtain that
\begin{align*}
\limsup_{n \to \infty} |g^n_a(k) - L_v c(k)_a| =
\limsup_{n \to \infty} |g^{n+m}_a(k) - L_v c(k)_a| = o(1) \xrightarrow{m \to \infty} 0,
\end{align*}
implying 
\begin{align*}
\lim_{n \to \infty} g^n_a(k) = L_v c(k)_a,
\end{align*}
uniformly on compact sets. Since $c(k)_a$ is the Fourier transform of $\overline{\mu}_a$, this implies by Levy's continuity theorem that
\begin{align*}
\lim_{n \to \infty} \lambda^{-n} \mu_a(\vartheta_{\bm P}^n(v)) = L_v \overline{\mu}_a,
\end{align*}
in the weak topology for all realisations in $\Omega$. Up to a renormalisation, this is precisely what we intended to show.
\end{proof}

\subsection{The substitution-invariant distribution}

Let $\nu$ be a $\vartheta_{\bm P}$-invariant probability measure as introduced in \Cref{S:rand-subst-intro}. 
By construction, $1=\nu(X_{\vartheta}) = \nu(\vartheta^n(X_{\vartheta}))$ for all $n \in \N$ and hence $X_{\vartheta}^{\infty}$ has full measure for $\nu$.
Using similar ideas as in the previous section, we obtain the analogue of \Cref{PROP:as-convergence-markov} for the invariant distribution $\nu$.

\begin{prop}
\label{PROP:nu-as-convergence}
For $\nu$-almost every $w$, we have
\begin{align*}
\mu_a(w) := \lim_{n \to \infty} \frac{1}{n} \mu_a(w_{[1,n]}) = \overline{\mu}_a,
\end{align*}
in the sense of weak convergence, for all $a \in \mc A$.
\end{prop}

First, we need some preparation. As a first step, we control the size change of a word under a random substitution.

\begin{lemma}
There exists a positive integer $r \in \N$ such that for all $\ell \in \N$, $m,n \in \N$ and all $v \in \mc L_{\vartheta}^n$ with 
\begin{align*}
n \geqslant n^+(\ell,m) = \ceil{\lambda^{-m} \ell} + r
\end{align*}
we have $|\vartheta^m(v)| \geqslant \ell$, and for all $v \in \mc L^n$ with
\begin{align*}
n \leqslant n^-(\ell,m) = \lfloor \lambda^{-m}\ell \rfloor - r
\end{align*}
we have $|\vartheta^m(v)| \leqslant \ell$.
\end{lemma}

\begin{proof}
This is a consequence of the $C$-balancedness of the language. 
Indeed, it guarantees the existence of some $c > 0$ such that for all $n \in \N$ and $v \in \mc L_{\vartheta}^n$, we obtain
\begin{align*}
| |v|_a - n R_a| < c
\end{align*}
for all $a \in \mc A$ and hence for all $m \in \N$,
\begin{align*}
|\vartheta^m(v)| = \sum_{a \in \mc A} (M^m \psi(v))_a
\leqslant n \lambda^m + \sum_{a,b \in \mc A} c (M^m)_{ab} \leqslant \lambda^m(n + r),
\end{align*}
for some $r \in \N$ because the entries of $\lambda^{-m} M^m$ are bounded in $m$. 
Similarly, we obtain\
\begin{align*}
|\vartheta^m(v)| \geqslant \lambda^m(n - r).
\end{align*}
The statement of the Lemma is just a reformulation of these two observations.
\end{proof}

In a similar manner as before we wish to obtain weak convergence via pointwise convergence of the Fourier transforms. We therefore define for every $a \in \mc A$ a function on $\mc L_{\theta}$ via
\begin{align*}
\widehat{\mu}_a \colon w \mapsto \operatorname{FT}[\mu_a(w)].
\end{align*}
If $w$ is a $\nu$-distributed random sequence and $\ell \in \N$, this induces a distribution on the random word $w_{[1,\ell]}$ and similarly on the random function $\widehat{\mu}_a(w_{[1,\ell]})$.
Due to the invariance of $\nu$ under all powers of $\vartheta_{\mathbf{P}}$, we can relate statements about $w_{[1,\ell]}$ to statements about random words of the form $\vartheta_{\mathbf{P}}^m(v)$, provided they are guaranteed to be long enough to cover the prefix of length $\ell$. The same holds for the distribution of $\widehat{\mu}_a(w_{[1,\ell]})$ which is derived from $w_{[1,\ell]}$. We make this more precise in the following.

\begin{lemma}
\label{LEM:mu-FT-renormalise}
Let $\ell, m \in \N$ and $n \geqslant n^+(\ell,m)$. For every $k \in \mathbb H^{\ast}$, $a \in \mc A$ and $c \in \mathbb{C}$, we have
\begin{align*}
\nu \{ w : \widehat{\mu}_a(w_{[1,\ell]})(k) = c \}
= \sum_{v \in \mc L_{\vartheta}^n} \nu[v] \mathbb{P} [ \widehat{\mu}_a(\vartheta_{\bm P}^m(v)_{[1,\ell]})(k) = c].
\end{align*}
\end{lemma}

\begin{proof}
Let $f \colon \mc L_{\vartheta} \to \mathbb{C}$ be defined by $f(w) = \widehat{\mu}_a(w)(k)$, and let $f_{\ell}(w) = f(w_{[1,\ell]})$. By \eqref{EQ:nu-invariance},
\begin{align*}
\nu[f_{\ell} = c] 
= \sum_{u \in \mc L_{\vartheta}^{\ell}, f(u) = c} \nu[u]  
= \sum_{u \in \mc L_{\vartheta}^{\ell}, f(u) = c} \sum_{v \in \mc L_{\vartheta}^n} \nu[v] \, \mathbb{P}[\vartheta_{\bm P}^m(v)_{[1,\ell]} = u]
= \sum_{v \in \mc L_{\vartheta}^n } \nu[v] \, \mathbb{P}[f(\vartheta_{\bm P}^m(v)_{[1,\ell]}) = c],
\end{align*}
which yields the desired assertion.
\end{proof}

\begin{lemma}
\label{LEM:FT-convergence-pointwise}
For every $k \in \mathbb H^{\ast}$, there exists a set $\Omega_k \subset X_{\vartheta}$ with $\nu(\Omega_k) = 1$ such that for all $w \in \Omega_k$, we have
\begin{align*}
\lim_{\ell \to \infty} \ell^{-1} \, \widehat{\mu}_a (w_{[1,\ell]})(k) = c(k)_a.
\end{align*}
\end{lemma}

\begin{proof}
We show this via an application of the Borel--Cantelli lemma. 
Let $\varepsilon > 0$ be arbitrary. With the same notation as in the proof of \Cref{LEM:mu-FT-renormalise} we set
\begin{align*}
A_{\ell} = \{ w: |\ell^{-1} f_{\ell}(w) - c(k)_a| > \varepsilon \}, 
\quad \widetilde{A}_{\ell} = \{ w \in \mc L_{\vartheta}^{\ell} : |\ell^{-1} f(w) - c(k)_a| > \varepsilon \}.
\end{align*}
Then, for arbitrary $m \in \N$, by \Cref{LEM:mu-FT-renormalise} we obtain for $n = n^+(\ell,m)$ that
\begin{align}
\label{EQ:A-ell-renormalisation}
\nu(A_{\ell}) = \sum_{v \in \mc L_{\vartheta}^n} \nu[v] \, \mathbb{P} [\vartheta_{\bm P}^m(v)_{[1,\ell]} \in \widetilde{A}_{\ell}].
\end{align}
By the Borel--Cantelli lemma, the desired almost-sure convergence follows as soon as we have shown that 
\begin{align*}
\sum_{\ell \in \N} \nu(A_{\ell}) < \infty.
\end{align*}
We therefore aim to obtain an appropriate upper bound for $\mathbb{P}[\vartheta_{\bm P}^m(v)_{[1,\ell]} \in \widetilde{A}_{\ell}]$.
By the definition of $n=n^+(\ell,m)$, it follows that $\ell \leqslant |\vartheta^m(v)| \leqslant \ell + (2r + 1) \lambda^m$. 
Since every letter in a word $w$ contributes a summand on the complex unit circle to $f(w)$, this implies, for all realisation of $\vartheta_{\bm P}(v)$,  that $|f(\vartheta_{\bm P} (v)) - f(\vartheta_{\bm P} (v)_{[1,\ell]})| \leqslant (2 r + 1) \lambda^m$. 
Choosing $\ell_0 = \ell_0(m,\varepsilon)$ large enough so that
\begin{align*}
\frac{(2r + 1) \lambda^m}{\ell_0} < \frac{\varepsilon}{2}
\end{align*}
we obtain that for all $\ell \geqslant \ell_0$, the condition $\vartheta_{\bm P} (v)_{[1,\ell]} \in \widetilde{A}_{\ell}$ implies that
\begin{align}
\label{EQ:FT-deviation-1}
|\ell^{-1} f(\vartheta_{\bm P} (v)) - c(k)_a| > \frac{\varepsilon}{2}.
\end{align}
Similar to the proof of \Cref{PROP:as-convergence-markov}, we have
\begin{align*}
\ell^{-1} f(\vartheta_{\bm P} (v)) 
& = \mc O(c^m) + \frac{\lambda^m}{\ell} \sum_{b \in \mc A} \sum_{x \in \mc  R_b^{\ast}(v)} f^{a,m,b}_{x}(k)
\\ & = \mc O(c^m) + \sum_{b \in \mc A} R_b \frac{\lambda^m |v|}{\ell} \frac{|v|_b}{R_b |v|} \frac{1}{|v|_b} \sum_{x \in \mc  R_b^{\ast}(v)} f^{a,m,b}_{x}(k),
\end{align*}
where $|c| < 1$ and $\{ f^{a,m,b}_{x}(k) \}_{x}$ are i.i.d., with distribution given by $\lambda^{-m}\operatorname{FT}[\mu_a(\vartheta_{\bm P}^m(b))](k)$.
In particular
\begin{align}
\label{EQ:FT-expectation}
\mathbb{E}[f^{a,m,b}_{x}(k)] = \lambda^{-m} B^{(m)}_{ab}(k).
\end{align}
Similarly to above, we have that $\ell^{-1} \lambda^m |v|$ differs from $1$ by at most $\ell^{-1} \lambda^m(r+1)$. 
Furthermore, the $C$-balanced property implies that 
\begin{align*}
1 - \frac{|v|_b}{R_b |v|} = \mc O(|v|^{-1}) = \mc O(\ell^{-1}).
\end{align*}
Since $|f_x^{a,m,b}(k)|$ is uniformly bounded by a constant $C$, we obtain that
\begin{align}
\label{EQ:multiple-error-estimate}
\Bigl|\ell^{-1} f(\vartheta_{\bm P} (v)) - \sum_{b \in \mc A} R_b \frac{1}{|v|_b} \sum_{x \in \mc  R_b^{\ast}(v)} f^{a,m,b}_{x}(k) \Bigr| = \mc O(c^m) + \mc O(\ell^{-1}\lambda^m) + \mc O(\ell^{-1})\tp
\end{align}
We may choose $m$ large enough and then $\ell_1 = \ell_1(m,\varepsilon) \geqslant \ell_0(m,\varepsilon)$ large enough such that for $\ell \geqslant \ell_1$ the sum of the error terms on the right hand side is smaller than $\varepsilon/4$.
Due to \Cref{LEM:cocycle-convergence}, we further know that
\begin{align*}
\lim_{m \to \infty} \sum_{b \in \mc A} R_b \lambda^{-m} B^{(m)}_{ab}(k) = \sum_{b \in \mc A} c(k)_a R_b L_b = c(k)_a
\end{align*}
and hence we may assume that $m$ has been chosen large enough to ensure that
\begin{align*}
\left| \sum_{b \in \mc A} R_b \lambda^{-m} B^{(m)}_{ab}(k) - c(k)_a \right| \leqslant \frac{\varepsilon}{8}.
\end{align*}
Together with \eqref{EQ:FT-deviation-1} and \eqref{EQ:multiple-error-estimate}, this yields that for all $\ell \geqslant \ell_1$, we have
\begin{align*}
\frac{\varepsilon}{8} < \sum_{b \in \mc A} R_b \, \left| \frac{1}{|v|_b} \sum_{x \in \mc R_b^{\ast}(v)} f_x^{a,m,b}(k) - \lambda^{-m} B^{(m)}_{ab}(k) \right|.
\end{align*}
Due to \eqref{EQ:FT-expectation} and since $(R_b)_{b \in \mc A}$ is a probability vector, this requires that for some $b \in \mc A$, we have that
\begin{align}
\label{EQ:condition-exponential}
\frac{\varepsilon}{8} < \left| \frac{1}{|v|_b} \sum_{x \in \mc R_b^{\ast}(v)} \bigl( f_x^{a,m,b}(k) - \mathbb{E}[f_x^{a,m,b}(k) ] \bigr) \right|.
\end{align}
By Cram\'{e}r's theorem on large deviations, there exists a constant $r = r(a,m,b,k) > 0$ such that (given $|v|_b$ is large enough) the probability for this event is bounded by $\me^{- r|v|_b} \leqslant \me^{- r' \ell} $, for some $r' > 0$. In summary, there exists a number $m \in \N$ and some $\ell_2 \in \N$ such that for all $\ell \geqslant \ell_2$ the condition $\vartheta_{\bm P}^m(v)_{[1,\ell]} \in \widetilde{A}_{\ell}$ implies that \eqref{EQ:condition-exponential} holds for some $b \in \mc A$ and
\begin{align*}
\mathbb{P}[\vartheta_{\bm P}^m(v)_{[1,\ell]} \in \widetilde{A}_{\ell}] \leqslant d \me^{- r' \ell}.
\end{align*}
Since the choice of $r'$ is independent of both $\ell$ and $v \in \mc L_{\vartheta}^n$ we obtain, via \eqref{EQ:A-ell-renormalisation}, that $\nu(A_{\ell}) \leqslant d \me^{-r'\ell}$, which is summable over $\ell$. 
The desired result follows by an application of the Borel--Cantelli lemma.
\end{proof}

We are now equipped to prove the main result of this subsection.

\begin{proof}[Proof of \Cref{PROP:nu-as-convergence}]
Given $k \in \mathbb H^{\ast}$, \Cref{LEM:FT-convergence-pointwise} gives that for $w$ in a set $\Omega_k$ of full measure, we have
\begin{align}
\label{EQ:FT-convergence-omega}
\lim_{ \ell \to \infty} \ell^{-1} \widehat{\mu}_a(w_{[1,\ell]})(k) = c(k)_a.
\end{align}
The rest follows as in the proof of \Cref{PROP:as-convergence-markov}.
Given a compact set $K$, there is a dense countable subset $E \subset K$ and a set $\Omega(E)$ of full measure such that for all $w \in \Omega(E)$ the convergence holds on $E$.
Since each $\ell^{-1} \widehat{\mu}_a(w_{[1,\ell]})$ is the average of functions with a uniformly bounded Lipshitz constant, the family $\{ \ell^{-1} \widehat{\mu}_a(w_{[1,\ell]})\}_{\ell \in \N}$ is equicontinuous. 
By the Arzel\`{a}-Ascoli theorem, this implies uniform convergence on $K$ for all $w \in \Omega(E)$. We can hence find a set $\Omega$ with $\nu(\Omega) = 1$ such that for all $w \in \Omega$ the convergence in \eqref{EQ:FT-convergence-omega} holds uniformly on all compact subsets of $\mathbb H^{\ast}$, and in particular pointwise for all $k \in \mathbb H^{\ast}$.
Since $c(k)_a$ is the Fourier transform of $\overline{\mu}_a$, the claim follows by an application of Levy's continuity theorem.
\end{proof}

\begin{corollary}
Let $\overline{\mu}^{\nu}_a (w_{[1,n]}) = \mathbb{E}_{\nu}[\mu_a(w_{[1,n]})]$. Then,
\begin{align*}
\lim_{n \to \infty} \frac{1}{n} \overline{\mu}_a^{\nu}(w_{[1,n]}) = \overline{\mu}_a,
\end{align*}
for all $a \in \mc A$, in the sense of weak convergence.
\end{corollary}

\begin{proof}
Let $f$ be an arbitrary bounded and uniformly continuous function on $\mathbb H$. 
Then, since each of $n^{-1} \overline{\mu}^{\nu}(w_{[1,n]})$ is a probability measure, we obtain $| n^{-1} \overline{\mu}_a^{\nu}(w_{[1,n]})(f)| \leqslant \|f\|_{\infty}$. By Lebesgue's dominated convergence theorem, we may therefore conclude from \Cref{PROP:nu-as-convergence},
\begin{align*}
\lim_{n \to \infty} \frac{1}{n} \mu_a^{\nu}(w_{[1,n]})(f) 
= \lim_{n \to \infty} \int_{X_{\vartheta}} \frac{1}{n} \mu_a(w_{[1,n]})(f) \dd \nu(w)
= \overline{\mu}_a(f). 
\end{align*}
Since $f$ was arbitrary, this shows the required weak convergence of probability measures.
\end{proof}

\subsection{The frequency measure}
\label{SS:frequency-measure-image}
We have seen in the previous section that the vector of Rauzy measures $\overline{\mu}(w) = (\overline{\mu}_a(w))_{a \in \mc A}$ is almost surely constant with respect to any $\vartheta_{\mathbf{P}}$-invariant measure $\nu$. Since $\nu$ is neither shift-invariant nor of full support, this might seem unsatisfactory as a statement about the underlying shift space $X_{\vartheta}$. In this section, we show that a similar statement holds for the ergodic frequency measure $\varrho$: the Rauzy measure is almost surely constant, up to a rigid translation that aligns its support with the Rauzy fractal of the sequence at hand. In addition, we make the remarkable observation that the Rauzy measure emerges as the image of $\varrho$ under the map $\phi$ that projects transitive points in $X_\vartheta$ to the Rauzy fractal.

In the following, let $\varrho$ be the frequency measure associated with the random substitution $\vartheta_{\mathbf{P}}$. 
Since $\varrho$ is ergodic, we obtain from \Cref{PROP:rauzy-measure-dynamical} the interpretation of the Rauzy measure as a pushforward of $\varrho$ under the map $\phi \colon X_{\vartheta}' \to \mc R_{\vartheta}$ in \Cref{SUBSEC:generic-factors}, where $X_{\vartheta}'$ is the set of transitive points of $X_{\vartheta}$. 
For $\phi$ to be well-defined we need to pick a reference point $w \in X_{\vartheta}'$ and declare $\phi(w) = 0$. 
To be specific, let us choose $w \in X_{\vartheta}^{\infty} \cap X_{\vartheta}'$ to be one of the $\nu$-typical points with the property that $\mu_a(w) = \overline{\mu}_a$ for all $a \in \mc A$ (\Cref{PROP:nu-as-convergence}) and that $w$ is $\varrho$-generic (\Cref{THM:frequency-measure}).

\begin{corollary}
    The Rauzy measure $\mu = \sum_{a \in \mc A} \overline{\mu}_a$ coincides with $\varrho \circ \phi^{-1}$. Further, $\mu(x) \circ t_{\phi(x)}^{-1} = \mu$ holds for $\varrho$-almost every $x \in X_{\vartheta}$.
\end{corollary}

\begin{proof}
    Since we have assumed that $w$ is $\varrho$-generic, its Rauzy measure $\mu = \sum_{a \in \mc A} \overline{\mu}_a$ coincides with $\varrho\circ \phi^{-1}$ (up to a shift by $\phi(w) = 0$), due to \Cref{PROP:rauzy-measure-dynamical}. Likewise, since $\varrho$ is ergodic by \Cref{THM:frequency-measure}, $\varrho$-almost every point is generic, and therefore satisfies $\mu(x) \circ t_{\phi(x)}^{-1} = \varrho \circ \phi^{-1}= \mu$, again by \Cref{PROP:rauzy-measure-dynamical}.
\end{proof}

\subsection{Lebesgue covering property}
\label{SS:Lebesgue-covering-RS}

Finally, we transfer the Lebesgue covering property from the context of $C$-balanced subshifts to the Rauzy measures associated with random substitutions. Recall that the lattice $\mc J$ is the integer span of the vectors $\pi(\mathbf{e}_i - \mathbf{e}_1)$ with $2\leqslant i \leqslant d$ and $D$ denotes the density of points in $\mc J$.

\begin{corollary}
\label{COR:Lebesgue-sum}
Let $\overline{\mu}$ be the self-similar measure vector corresponding to a compatible, irreducible Pisot random substitution $\vartheta_{\bm P}$. Then,
\begin{align*}
\sum_{j \in \mc J} \sum_{a \in \mc A} \overline{\mu}_a \circ t_j^{-1} = D \Leb.
\end{align*}
\end{corollary}

\begin{proof}
By \Cref{PROP:nu-as-convergence}, we have for $\nu$-almost every $w$ that
\begin{align*}
\sum_{a \in \mc A} \overline{\mu}_a = \lim_{n \to \infty} \frac{1}{n} \sum_{a \in \mc A} \mu_a(w_{[1,n]}) = \lim_{n \to \infty} \frac{1}{n} \mu(w_{[1,n]}) = \mu(w),
\end{align*}
and hence the claim follows immediately from \Cref{PROP:Lebesgue-covering}.
\end{proof}

\begin{corollary}
Under the requirements of \Cref{COR:Lebesgue-sum}, each of the measures $\overline{\mu}_a$ with $a \in \mc A$ is absolutely continuous with respect to Lebesgue measure, with density bounded above by $D$.
\end{corollary}

\begin{proof}
This follows from the fact that $\overline{\mu}_a$ is dominated by $D$ times $\Leb$ due to \Cref{COR:Lebesgue-sum}.
\end{proof}

\begin{remark}
    Recall that we obtain the Rauzy measures of the marginals of a random substitution $\vartheta_{\mathbf{P}}$ by setting appropriate probability parameters to $1$. In this case the Rauzy measures are uniformly distributed on their support and the statement of \Cref{COR:Lebesgue-sum} returns the well-known property of having a $k$-fold multi-tiling for irreducible Pisot substitutions---see for instance \cite[Section 5.11]{BST10}. The Pisot substitution conjecture is equivalent to $k=1$, that is, having no overlaps of positive Lebesgue measure with respect to the sum in \Cref{COR:Lebesgue-sum}. By \Cref{PROP:meas-continuity}, this is equivalent to all overlaps of positive measure needing to disappear in the limit of degenerate $\mathbf{P}$ as we approach a deterministic substitution.
\end{remark}

\section{Applications to S-adic systems}
\label{S:S-adic}

Rauzy fractals were instrumental in proving variants of the Pisot substitution conjecture in the S-adic setting under appropriate conditions. 
The case of two letters was resolved in \cite{BMST16} and progress in more general settings has been made with the help of multidimensional continued fraction algorithms \cite{BST23,FN20}.
Here, we demonstrate several connections between the theory of Rauzy fractals for S-adic systems and that of random substitutions.
We introduce S-adic systems in a somewhat narrow setting, adapted to our needs. 
For a more general overview, we refer the reader to \cite{BD14} and the references therein. 

Let $\Sub$ be a set of (deterministic) substitutions on an alphabet $\mc A$. 
We assume that $\Sub$ is \emph{compatible} in the sense that every substitution $\theta \in \Sub$ has the same substitution matrix $M$. 
This latter condition implies that $\Sub$ is finite. 
We further assume that $M$ is primitive. An S-adic system is determined by a sequence $\pmb \theta  = (\theta_n)_{n \in \N} \in \Sub^{\N}$. 
We call a sequence $\bm a = (a_n)_{n \in \N} \in \mc A^{\N}$ \emph{adapted to $\pmb \theta $} if $\theta_n(a_n)$ starts with the letter $a_{n-1}$ for all integers $n \geqslant 2$, and write $\theta_{[1,n]} = \theta_1 \circ \cdots \circ \theta_n$ for all $n \in \N$. 
The assumption that $\bm a$ is adapted guarantees that $\theta_{[1,n]}(a_n)$ is a prefix of $\theta_{[1,m]}(a_m)$ for all integers $n \leqslant m$. 
Hence, there is a well-defined limiting sequence $u = u(\pmb \theta , \bm a) \in \mc A^{\N}$ with the property that $\theta_{[1,n]}(a_n)$ is a prefix of $u$ for all $n \in \N$. 
With some abuse of notation, in this situation, we write
\[
u = \lim_{n \to \infty} \theta_{[1,n]}(a_n).
\]
It can be verified that there is only a finite number of adapted letter sequences for a given $\pmb \theta $, and it should be noted that $u$ can be written as the image of $\theta_1$ acting on a different limiting sequence that corresponds to the shifted S-adic sequence $S \pmb \theta  = (\theta_{n+1})_{n \in \N}$ and $S \bm a = (a_{n+1})_{n \in \N}$. 
More generally,
\[
u(\pmb \theta , \bm a) = \theta_{[1,n]} ( u(S^n\pmb \theta , S^n \bm a) ),
\]
and for a limiting sequence $u = u(\pmb \theta ,\bm a)$, let $u^{(n)} = u(S^n \pmb \theta , S^n \bm a)$ for the corresponding preimage sequences. 

\subsection{Rauzy fractals}

Let us further assume that every $\theta \in \Sub$ is irreducible Pisot.
Then, every limiting sequence is $C$-balanced \cite{Arnoux}. We can therefore assign a corresponding subtile
\[
\mc R_{\pmb \theta , b} = \mc R_b(u(\pmb \theta , \bm a)),
\quad \mc R_{\pmb \theta } = \mc R(u(\pmb \theta , \bm a)),
\]
for all $b \in \mc A$. 
It was shown in \cite{Arnoux} that $\mc R_{\pmb \theta , b}$ is indeed independent of the choice of $\bm a$ and that it depends continuously on $\pmb \theta  \in \Sub^{\N}$, with respect to the Hausdorff metric. It can also be inferred from the proof of \cite[Lemma~3.1]{Arnoux} that there is a uniform bound for the diameter of $\mc R_{\pmb \theta ,a}$ with $\pmb \theta  \in \Sub^\N$.
Further, the Rauzy fractals satisfy the self-consistency relation
\begin{align}
\label{EQ:S-adic-GIFS}
\mc R_{\pmb \theta ,a} = \bigcup_{{\substack{(p,b,s) \\ pas=\theta_1(b)}}}  h(\mc R_{S\pmb \theta ,b}) + \pi (\psi(p)).
\end{align}
Let us emphasise that $\mc R_{\pmb \theta }$ generally \emph{does} depend on the choice of the directive sequence $\pmb \theta $ and not just on the available set $\Sub$ of substitutions. 
In the following, we show that the smallest enveloping set for all $\mc R_{\pmb \theta }$ with $\pmb \theta  \in \Sub^{\N}$ is given by the Rauzy fractal of a corresponding random substitution $\vartheta$. For this, we need to choose $\vartheta$ in an appropriate manner. 
Since the choice of the probability vectors is immaterial for this discussion, we consider a random substitution as a set-valued function for now.

\begin{definition}
Let $\Sub$ be a compatible collection of primitive substitutions. The \emph{enveloping random substitution} of $\Sub$ is $\vartheta \colon \mc A \to \mc F(\mc A^+)$, with
$
\vartheta(a) = \{ \theta(a) \colon \theta \in \Sub\}
$ 
for all $a \in \mc A$.
\end{definition}

Note that every substitution in $\Sub$ is a marginal of its enveloping random substitution $\vartheta$, but not every marginal of $\vartheta$ needs to be contained in $\Sub$. 

\begin{example}
Let $\Sub = \{\theta, \theta' \}$, with $\theta \colon 1 \mapsto 112, 2 \mapsto 21$ and $\theta' \colon 1 \mapsto 211, 2 \mapsto 12 $. Then the enveloping random substitution of $\Sub$ is given by 
\[
\vartheta \colon 1 \mapsto \{112,211 \}, \quad
2 \mapsto \{12,21 \}.
\]
There are four marginals of $\vartheta$, one for each possible pair of inflation words.
\end{example}

\begin{theorem}
\label{THM:S-adic-interpretation}
Let $\Sub$ be a compatible collection of irreducible Pisot substitutions, with enveloping random substitution $\vartheta$. For all $a \in \mc A$,
\begin{align*}
\bigcup_{\pmb \theta  \in \Sub^{\N}} \mc R_{\pmb \theta ,a} = \mc R_{\vartheta,a}.
\end{align*}
\end{theorem}

As a preparatory step, we show the following under the assumptions of Theorem~\ref{THM:S-adic-interpretation}.

\begin{lemma}
The union $\bigcup_{\pmb \theta  \in \Sub^{\N}} \mc R_{\pmb \theta , a}$ is a compact set.
\end{lemma}

\begin{proof}
Since there is a \emph{uniform} bound for the diameter of each of the sets $\mc R_{\pmb \theta ,a}$, the same holds for the union. It remains to show that the union is closed.
This is a consequence of the fact that $\mc R_{\pmb \theta ,a}$ depends continuously on $\pmb \theta $ and the compactness of $\Sub^\N$. Indeed, let $x$ be in the closure of $\cup_{\pmb \theta  \in \Sub^{\N}} \mc R_{\pmb \theta , a}$. Then, there exists a sequence $(\pmb \theta _n)_{n \in \N}$ such that $d(x,\mc R_{\pmb \theta _n,a})$ converges to $0$. Up to restricting to a subsequence, we may assume that $(\pmb \theta _n)_{n \in \N}$ converges to some $\pmb \theta ' \in \Sub^{\N}$. By continuity, it follows that 
$
\mc R_{\pmb \theta ', a} = \lim_{n \to \infty} \mc R_{\pmb \theta _n,a},
$
in the Hausdorff metric. Since all the sets are compact, this coincides with the Kuratowski-Painlev\'{e} limit of $(\mc R_{\pmb \theta _n, a})_{n \in \N}$, which contains $x$ by assumption. Hence, $x \in \mc R_{\pmb \theta', a} \subset \cup_{\pmb \theta  \in \Sub^{\N}} \mc R_{\pmb \theta , a}$ and the claim follows.
\end{proof}

\begin{proof}[Proof of Theorem~\ref{THM:S-adic-interpretation}]
Let $K_a = \cup_{\pmb \theta  \in \Sub^{\N}} \mc R_{\pmb \theta , a}$ for all $a \in \mc A$. Taking the union over all $\pmb \theta  \in \Sub^\N$ in \eqref{EQ:S-adic-GIFS} yields
\begin{align}
\label{EQ:U_a-consistency}
K_a = \bigcup_{\theta_1 \in \Theta} \bigcup_{{\substack{(p,b,s) \\ pas=\theta_1(b)}}}  h(K_b) + \pi (\psi(p)) = \bigcup_{{\substack{(p,b,s) \\ pas \in \vartheta(b)}}}  h(K_b) + \pi (\psi(p)),
\end{align}
where the last step follows since $\vartheta(b) = \{\theta(b) : \theta \in \Theta \}$, by definition of the enveloping random substitution. 
The relation in \eqref{EQ:U_a-consistency} coincides with the self-consistency relation for the collection $\{\mc R_{\vartheta,a} \}_{a \in \mc A}$. 
Since this relation has a unique solution of non-empty compact sets, the claim follows.
\end{proof}

Let us interpret the main result of this section in terms of the S-adic Pisot conjecture.
Recall that we call a sequence of substitutions $\pmb \theta$ \emph{recurrent} if every word in $\pmb \theta$ appears in $\pmb \theta$ infinitely often. 
If we restrict to the case that the substitution matrix $M$ is unimodular, it follows from \cite[Theorem~3.1]{BST19} that the $\mc J$-periodic repetition of $\mc R_{\pmb \theta}$ is a multi-tiling of $\mathbb{H}$, provided $\pmb \theta$ is recurrent. 
Further, the corresponding dynamical system $(\mathbb{X}_{\pmb \theta}, S)$ has purely discrete spectrum if this multi-tiling is in fact a tiling. 
This is implicitly a condition on the Lebesgue measure of the Rauzy fractal $\mc R_{\pmb \theta}$ and we obtain the following sufficient condition from Theorem~\ref{THM:S-adic-interpretation}.

\begin{corollary}\label{cor:new}
    Let $\Sub$ be a compatible collection of unimodular, irreducible Pisot substitutions, with enveloping random substitution $\vartheta$. 
    Assume that the Lebesgue measure of $\mc R_{\vartheta}$ is smaller than two times the measure of the fundamental domain in $\mc J$. 
    Then $(\mathbb{X}_{\pmb \theta}, S)$ has purely discrete spectrum for every recurrent sequence $\pmb \theta \in \Sub^\N$.
\end{corollary}

We provide an application of this result in \Cref{ex:additional_note_to_cor:new}.

\subsection{Rauzy measures}\label{SS:Rauzy_measures}

Let $w \in \mc A^\N$ be $C$-balanced, and assume that
\[
\mu_a(w) := \lim_{n \to \infty} \frac{1}{n}\mu_a(w_{[1,n]})
\]
exists as a limit for all $a \in \mc A$. 
Further assume that $\theta$ is a primitive substitution on $\mc A$ with Perron--Frobenius eigenvalue $\lambda$.
Then, we have the following special case of \Cref{COR:measure-gifs-iteration}.

\begin{lemma}
\label{LEM:sub-Rauzy-measure-recursion}
The Rauzy measure $\mu_a(\theta(w))$ exists as a weak limit. 
For all $a \in \mc A$,
\[
\mu_a(\theta(w)) = \frac{1}{\lambda} \sum_{(p,b,s)} \mu_b(w) \circ f_p^{-1}. 
\]
\end{lemma}

\begin{proof}
In the special case of a deterministic substitution, \Cref{COR:measure-gifs-iteration} yields
\[
\mu_a(\theta(w_{[1,n]})) = \sum_{(p,b,s)} \mu_b(w_{[1,n]}) \circ f_p^{-1}.
\]
Dividing both sides by $|\theta(w_{[1,n]})| \sim \lambda n$, this gives the desired relation as $n \to \infty$
\end{proof}

\begin{remark}
Even if $\mu_a(w)$ does not exist as a limit, an analogous result holds for accumulation points.
More precisely, there is an increasing sequence $(n_j)_{j \in \N}$ of natural numbers such that $\mu_a(w)$ can be defined as a limit along this subsequence.
Following the same arguments as before, we find that $\mu_a(\theta(w))$ exists along the subsequence $(n_j')_{j \in \N}$ with $n_j' = |\theta(w_{[1,n_j})|$ and the relation in Lemma~\ref{LEM:sub-Rauzy-measure-recursion} persists.
\end{remark}

With this preparation at hand, let us return our attention to S-adic systems, with $\Theta$ a compatible collection of irreducible Pisot substitutions. 
As usual, let $\lambda$ be the Perron--Frobenius eigenvalue of the (uniform) substitution matrix $M$ and $\bm R$ the right Perron--Frobenius eigenvector.

\begin{prop}
\label{PROP:S-adic-recursion}
Let $\pmb \theta \in \Sub^\N$, with limiting sequence $u = u(\pmb \theta , \bm a)$ for some $\bm a \in \mc A^\N$. 
Then, for every $a \in \mc A$, the measure
\[
\mu_a(u) = \lim_{n \to \infty} n^{-1} \mu_a(u_{[1,n]})
\]
exists as a limit and is independent of $\bm a$.
Let $\mu_{\pmb \theta,a} := \mu_a(u)$ for all $\pmb \theta \in \Sub^\N$ and $a \in \mc A$. 
Then, for all $a \in \mc A$,
\begin{align}
\label{EQ:S-adic-measure-iteration}
\mu_{\pmb \theta,a} = \frac{1}{\lambda} \sum_{\substack{(p,b,s)\\ \theta_1(b) = pas}} \mu_{S \pmb \theta , b} \circ f_p^{-1}.
\end{align}
\end{prop}

\begin{proof}
Let $u^{(1)}$ be the preimage sequence of $u$ with $u = \theta_1(u^{(1)})$. 
We obtain from Lemma~\ref{LEM:sub-Rauzy-measure-recursion}, 
\begin{align}
\label{EQ:S-adic-measure-iteration-I}
\mu_a(u) = \frac{1}{\lambda} \sum_{\substack{(p,b,s)\\ \theta_1(b) = pas}}  \mu_b(u^{(1)}) \circ f_p^{-1},
\end{align}
for all $a \in \mc A$,
possibly in the sense of appropriate accumulation points. 
Renormalising to probability measures via $\widetilde{\mu}_a(u) = R_a^{-1} \mu_a(u)$, this means that $\{\widetilde{\mu}_a(u) \}_{a \in \mc A}$ is the image of $\{ \widetilde{\mu}_a(u^{(1)}) \}_{ a \in \mc A}$ under the operator
\[
\Phi_{\theta_1} \colon \{\nu_a \}_{a \in \mc A} \mapsto \biggl\{ \sum_{e \in E_a(\theta_1)} p^a_e \nu_{t(e)} \circ f_e^{-1} \biggr\}_{a \in \mc A},
\]
where $E_a(\theta_1) = \{ (p,b,s): \theta_1=pas \}$ and the data $(f_e)$ and $(p^a_e)$ correspond to an appropriately chosen GIFS with contraction ratio $\lambda^{-1}$; compare the discussion in \Cref{SUBSEC:expected_values}. 
By general arguments, $\Phi_{\theta}$ is a contraction for every $\theta \in \Theta$, with contraction ratio $\lambda^{-1}$ with respect to the Monge--Kantorovich metric; compare the proof of \Cref{PROP:meas-continuity}.
Since 
\[
\{ \widetilde{\mu}_a(u) \}_{a \in \mc A}
= \Phi_{\theta_1} \circ \cdots \circ \Phi_{\theta_n} \left( \{ \widetilde{\mu}_a(u^{(n)}) \}_{a \in \mc A}\right),
\]
and the space of probability tuples on $\mc R_{\vartheta}$ is compact, it follows that there is a unique solution for $\{ \widetilde{\mu}_a(u) \}_{a \in \mc A}$ that is completely determined by the sequence of substitutions $\pmb \theta  \in \Sub^\N$. 
This also proves that all accumulation points that could define $\mu_a(u)$ coincide and it actually exists as a limit. 
By the same argument, $\mu_a(u)$ does not depend on the choice of $\bm a \in \mc A^\N$. 
Hence, $\mu_{\pmb \theta ,a} := \mu_a(u)$ is well-defined.
Recalling that $u^{(1)}$ is a limiting sequence for $S\pmb \theta $, we observe that $\mu_{S \pmb \theta ,a} = \mu_{a}(u^{(1)})$ and the final claim follows from \eqref{EQ:S-adic-measure-iteration-I}.
\end{proof}

\begin{remark}
If the union in \eqref{EQ:S-adic-GIFS} is (Lebesgue-) measure-disjoint for all $\pmb \theta \in \Theta^\N$, it is straightforward to verify that the self-consistency relation in \eqref{EQ:S-adic-measure-iteration} is solved by taking $\mu_{\pmb \theta,a}$ to be Lebesgue measure restricted to $\mc R_{\pmb \theta,a}$ for all $a \in \mc A$.
Since the solution to \eqref{EQ:S-adic-measure-iteration} is unique up to normalisation, this shows that $\mu_{\pmb \theta, a}$ in fact \emph{is} the (normalised) Lebesgue measure restricted to its support in this case. 
The condition that the union in \eqref{EQ:S-adic-GIFS} is measure-disjoint certainly holds if $\Theta$ is a singleton, but we are not aware of a result in this direction for the general S-adic case. 
Similarly, if the periodic repetition of $\mc R_{\pmb \theta}$ along the lattice $\mc J$ has only overlaps of measure $0$, we obtain from the Lebesgue covering property in \Cref{COR:Lebesgue-sum} that $\mu_{\pmb \theta}$ coincides with the normalised Lebesgue measure on $\mc R_{\pmb \theta}$.
\end{remark}

Just as for the Rauzy fractal $\mc R_{\pmb \theta ,a}$, the associated Rauzy measure $\mu_{\pmb \theta ,a}$ will in general depend on $\pmb \theta  \in \Sub^\N$. 
Along the same lines as in \cite{Arnoux}, we can show that the map $\pmb \theta  \mapsto \mc R_{\pmb \theta ,a}$ is continuous.
To be more precise, let $\mc M$ be the space of probability measures on $\mc R_{\vartheta}$ and, with $d = \# \mc A$, we extend the definition of $\dmk$ to $\mc M^d$ via
\[
\dmk(\{\mu_a \}_{a \in \mc A}, \{\nu_a \}_{a \in \mc A}) = \max_{a \in \mc A} \dmk(\mu_a, \nu_a).
\]

\begin{prop}
For every $a \in \mc A$, the map $\pmb \theta  \mapsto \mu_{\pmb \theta ,a}$ is continuous with respect to the weak topology.
\end{prop}

\begin{proof}
Let $\widetilde{\mu}_{\pmb \theta } = \{ \widetilde{\mu}_{\pmb \theta ,a} \}_{a \in \mc A}$, where $\widetilde{\mu}_{\pmb \theta ,a} = R_a^{-1} \mu_{\pmb \theta ,a}$ is the normalisation to a probability measure.
Since $\dmk$ induces the weak topology, it suffices to show that $\pmb \theta  \mapsto \widetilde{\mu}_{\pmb \theta }$ is continuous as a map from $\Sub^\N$ to the metric space $(\mc M^d, \dmk)$. 
In the notation of the proof of Proposition~\ref{PROP:S-adic-recursion}, we have
\begin{align}
\label{EQ:S-adic-vector-mu}
\widetilde{\mu}_{\pmb \theta } = \Phi_{\theta_1}(\widetilde{\mu}_{S\pmb \theta }),
\end{align}
where $\Phi_{\theta_1}$ is a contraction.
By the same arguments as in the proof of Proposition~\ref{PROP:S-adic-recursion}, the function $\pmb \theta  \to \widetilde{\mu}_{\pmb \theta }$ is uniquely determined by the relation in \eqref{EQ:S-adic-vector-mu}.  
Let $\mc F$ be the space of continuous functions from $\Sub^\N$ to $\mc M^d$ equipped with the uniform metric
\[
d(F,G) = \max_{\pmb \theta  \in \Sub^\N} \dmk(F(\pmb \theta ), G\pmb \theta )),
\]
for $F,G \in \mc F$. 
Note that $(\mc M^d,\dmk)$ is a compact metric space and hence complete. 
It is straightforward to verify that this also implies that $(\mc F,d)$ is complete; compare the proof of \cite[Lemma~4.3]{Arnoux}. 
We define $\Phi \colon \mc F \to \mc F$ via
$
(\Phi F)(\pmb \theta ) = \Phi_{\theta_1} F(S \pmb \theta ),
$
for all $\pmb \theta  \in \Sub^\N$. 
We verify that $\Phi$ is a contraction on $\mc F$.
Let $0<c<1$ be a uniform upper bound for the contraction rates of $\Phi_{\theta}$ with $\theta \in \Sub$. 
Then,
\[
d(\Phi(F),\Phi(G)) = \max_{\pmb \theta  \in \Sub^\N} \dmk(\Phi_{\theta_1} F(S \pmb \theta ), \Phi_{\theta_1} G(S \pmb \theta )) \leqslant c d(F,G),
\]
for all $F,G \in \mc F$ and hence $\Phi$ is indeed a contraction.
By Banach's fixed point theorem, it follows that there is a unique function $F^{\ast} \in \mc F$ that is invariant under $\Phi$. 
This means that
\[
F^\ast(\pmb \theta ) = \Phi_{\theta_1} F(S\pmb \theta ),
\]
and hence $F^\ast$ satisfies the consistency relation in \eqref{EQ:S-adic-vector-mu}. 
Since the solution to this relation is unique, it follows that $F^\ast (\pmb \theta ) = \widetilde{\mu}_{\pmb \theta }$ for all $\pmb \theta  \in \Sub^\N$. 
It follows that $\pmb \theta  \to \widetilde{\mu}_{\pmb \theta }$ is continuous.
\end{proof}

At this point, we have shown that there is a canonical way to continuously assign a Rauzy measure $\mu_{\pmb \theta ,a}$ to each $\pmb \theta  \in \Sub^\N$ and $a \in \mc A$.
Taking an average over all $\pmb \theta  \in \Sub^\N$ we obtain an \emph{expected Rauzy measure} with respect to an appropriate measure $\nu$ on $\Sub^\N$.
Arguably, the simplest possible choice is to assume that $\nu$ is a Bernoulli measure on $\Sub^\N$.
More precisely, we pick a probability vector $\bm p = (p_{\theta})_{\theta \in \Sub}$ with $\sum_{\theta \in \Sub} p_{\theta} =1$ and take $\nu$ to be the direct product $\nu = \bigotimes_{i=1}^\infty \bm p$ on $\Sub^\N$.
Since $\mu_{\pmb \theta ,a}$ is continuous in $\pmb \theta $, it is measurable with respect to the product topology on $\Sub^\N$ and we can define
\[
\mathbb{E}_{\nu}[\mu_{\pmb \theta ,a}] = \int_{\Sub^\N} \mu_{\pmb \theta ,a} \dd \nu(\pmb \theta ).
\]
We will show that this averaged Rauzy measure coincides with the Rauzy measure of the enveloping random substitution $(\vartheta,\bm P)$, for an appropriate choice of $\bm P$.
Recall that the parameters in $\bm P$ give precisely the probabilities $\mathbb{P}[\vartheta_{\bm P}(a) = u]$ for all $a \in \mc A$ and all possible realisations $u$.

\begin{definition}
Let $\vartheta$ be the enveloping random substitution of $\Sub$.
The data $\bm P$ \emph{derived} from a probability vector $\bm p = (p_\theta)_{\theta \in \Sub}$ is given by
\[
\mathbb{P}[\vartheta_{\bm P}(a) = u] = \sum_{\theta: \theta(a) = u} p_{\theta},
\]
for all $a \in \mc A$ and $u \in \vartheta(a)$. 
\end{definition}

\begin{theorem}
Let $\Sub$ be a compatible collection of irreducible Pisot substitutions, with enveloping random substitution $\vartheta$.
Assume $\nu = \bigotimes_{i=1}^\infty \bm p$ and let $\bm P$ be the data derived from $\bm p$. 
Further, let $\overline{\mu}= \{\overline{\mu}_a \}_{a \in \mc A}$ be the vector of Rauzy measures associated with $(\vartheta,\bm P)$. 
Then, for all $a \in \mc A$,
\[
\mathbb{E}_{\nu}[\mu_{\pmb \theta ,a}] = \overline{\mu}_a.
\]
\end{theorem}

\begin{proof}
For each $a \in \mc A$ let $\kappa_a = \mathbb{E}_{\nu}[\mu_{\pmb \theta ,a}]$.
Since each of the measures $\mu_{\pmb \theta ,a}$ has total mass $R_a$, the same holds for $\kappa_a$. 
Integrating the relation in \eqref{EQ:S-adic-measure-iteration} with respect to $\nu$ gives
\begin{align*}
\kappa_a & = \frac{1}{\lambda} \sum_{\theta_1 \in \Sub} p_{\theta_1} \sum_{\substack{(p,b,s) \\ \theta_1(b) = pas}} \kappa_b \circ f_p^{-1}
= \frac{1}{\lambda} \sum_{(p,b,s)} \sum_{\theta_1:\theta_1(b) = pas} p_{\theta_1} \kappa_b \circ f_p^{-1}
\\& = \frac{1}{\lambda} \sum_{(p,b,s)} \mathbb{P}[\vartheta_{\bm P}(b) = pas] \kappa_b \circ f_p^{-1},
\end{align*}
using that $\bm P$ is derived from $\bm p$ in the last step.
This is the same self-consistency relation as the one for the vector $\overline{\mu}$ in \Cref{PROP:measure-GIFS}. 
Up to a renormalisation to probability measures, this defines the unique self-similar solution for a particular GIFS. 
Since $\kappa_a$ and $\overline{\mu}_a$ have the same total mass $R_a$, it follows that $\kappa_a = \overline{\mu}_a$ for all $a \in \mc A$.
\end{proof}

\section{Examples}\label{S:examples}

The Fibonacci substitution $\theta \colon 1 \mapsto 12, \, 2 \mapsto 1$ is the prototypical example of a unimodular, irreducible Pisot substitution on two letters.
Its associated Rauzy fractal is an interval, with subtiles $\mc R_{\theta, 1} = [-\tau^{-2},\tau^{-3}]$ and $\mc R_{\theta,2} = [\tau^{-3},\tau^{-1}]$.
The \emph{random Fibonacci substitution} is the compatible random substitution obtained by locally mixing $\theta$ with the substitution $\tilde{\theta} \colon 1 \mapsto 21, \, 2 \mapsto 1$, which gives rise to the Rauzy fractal with tiles $\mc R_{\tilde{\theta},1} = [-\tau^{-1},0]$, $\mc R_{\tilde{\theta},2} = [0,\tau^{-2}]$.

\begin{example}\label{EX:rand-fib}
Given $p \in (0,1)$, let $\vartheta_{\bm P}$ be the random Fibonacci substitution
\begin{align*}
    \vartheta_{\bm P} \colon 
    \begin{cases}
    1 \mapsto
    \begin{cases}
        12 &\text{with probability $p$,}\\
        21 &\text{with probability $1-p$,}
    \end{cases}\\
    2 \mapsto 1 \quad \text{with probability $1$.}
    \end{cases}
\end{align*}
By \Cref{THM: r-Rauzy GIFS}, the tiles of the Rauzy fractal associated with $\vartheta_{\bm P}$ are the unique, non-empty, compact sets $\mc R_1$ and $\mc R_2$ satisfying
\begin{align*}
    \mc R_1 &= \left( -\frac{1}{\tau} \mc R_1 \right) \cup \left( -\frac{1}{\tau} \mc R_1 - \frac{1}{\tau} \right) \cup \left(-\frac{1}{\tau} \mc R_2 \right)\\
    \mc R_2 &= \left( - \frac{1}{\tau} \mc R_1 \right) \cup \left( -\frac{1}{\tau} \mc R_1 + \frac{1}{\tau^2} \right) \text{.}
\end{align*}
In particular, $\mc R_1 = [-1,\tau^{-1}]$ and $\mc R_2 = [-\tau^{-2},1]$.
We highlight that there exist points in these intervals which do not appear in the corresponding intervals of the Rauzy fractal of either marginal of $\vartheta_{\bm P}$.
Further, note that $\mc R_1$ and $\mc R_2$ intersect on a set of positive Lebesgue measure.
The above graph-directed iterated function system can be rewritten as an ordinary iterated function system (IFS) for $\mc R_1$:
\begin{align*}
    \mc R_1 = \left( -\frac{1}{\tau} \mc R_1 \right) \cup \left( -\frac{1}{\tau} \mc R_1 - \frac{1}{\tau} \right) \cup \left(\frac{1}{\tau^2} \mc R_1 \right) \cup \left(\frac{1}{\tau^2} \mc R_1 - \frac{1}{\tau^3}\right)\tp
\end{align*}
This system does not satisfy the open set condition, since the similarity dimension does not equal $1$ and so Hutchinson's formula is not satisfied.
It follows that the graph-directed iterated function system also does not satisfy the open set condition, again in contrast to the deterministic setting.
However, since the iterated function system for $\mc R_1$ is an IFS in $\R$ consisting of contraction ratios that are all positive integer powers of $\tau^{-1}$, it follows by \cite[Theorem~2.5]{ngai-wang} that the IFS satisfies the \emph{finite type condition}, a weak separation condition that is desirable in the absence of the open set condition.
For a detailed introduction to separation conditions weaker than the open set condition, we refer the reader to \cite{hare-hare-rutar}.
\begin{figure}
    \centering
    \includegraphics[scale=0.14]{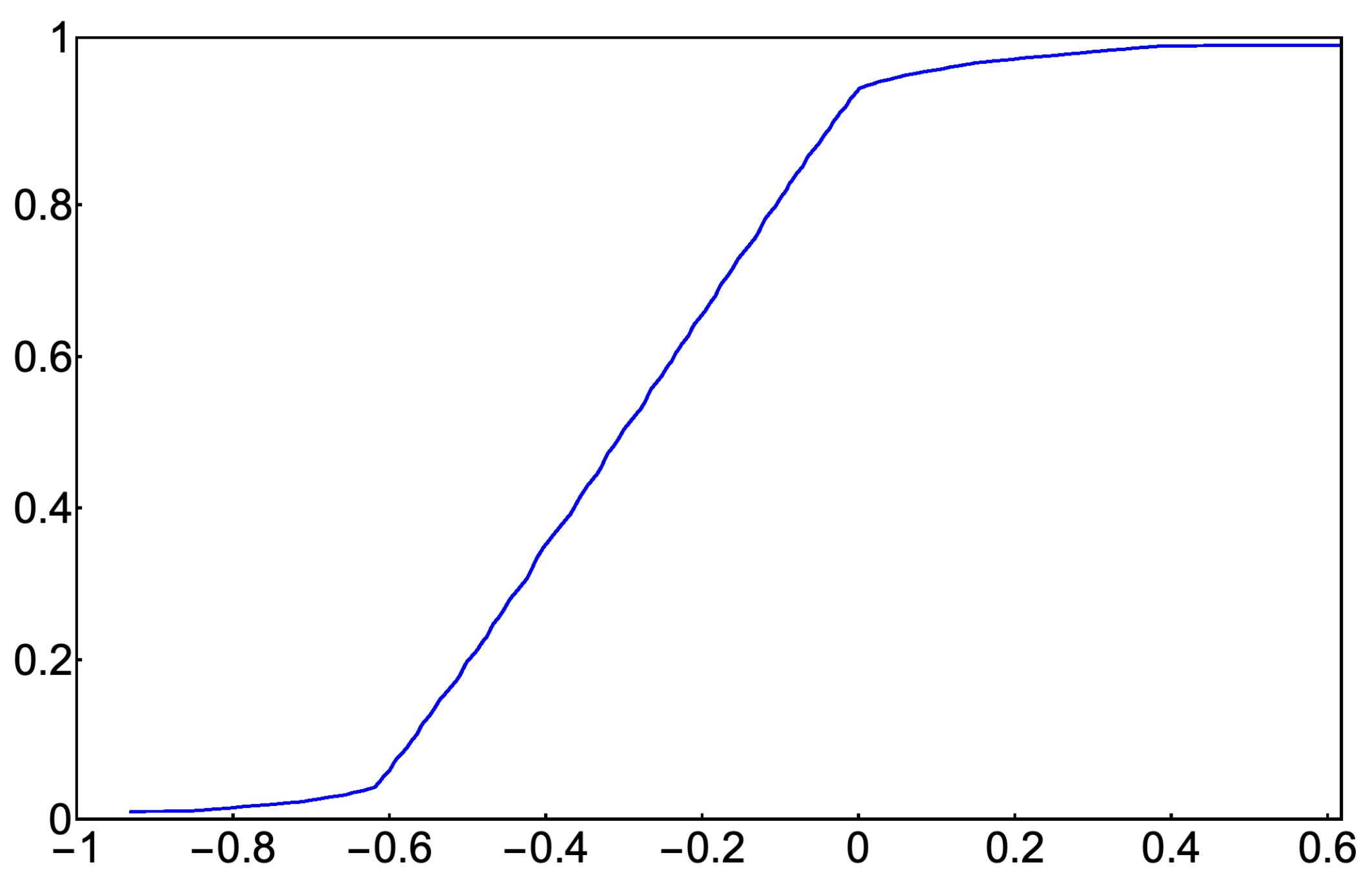}
    \hfill
    \includegraphics[scale=0.14]{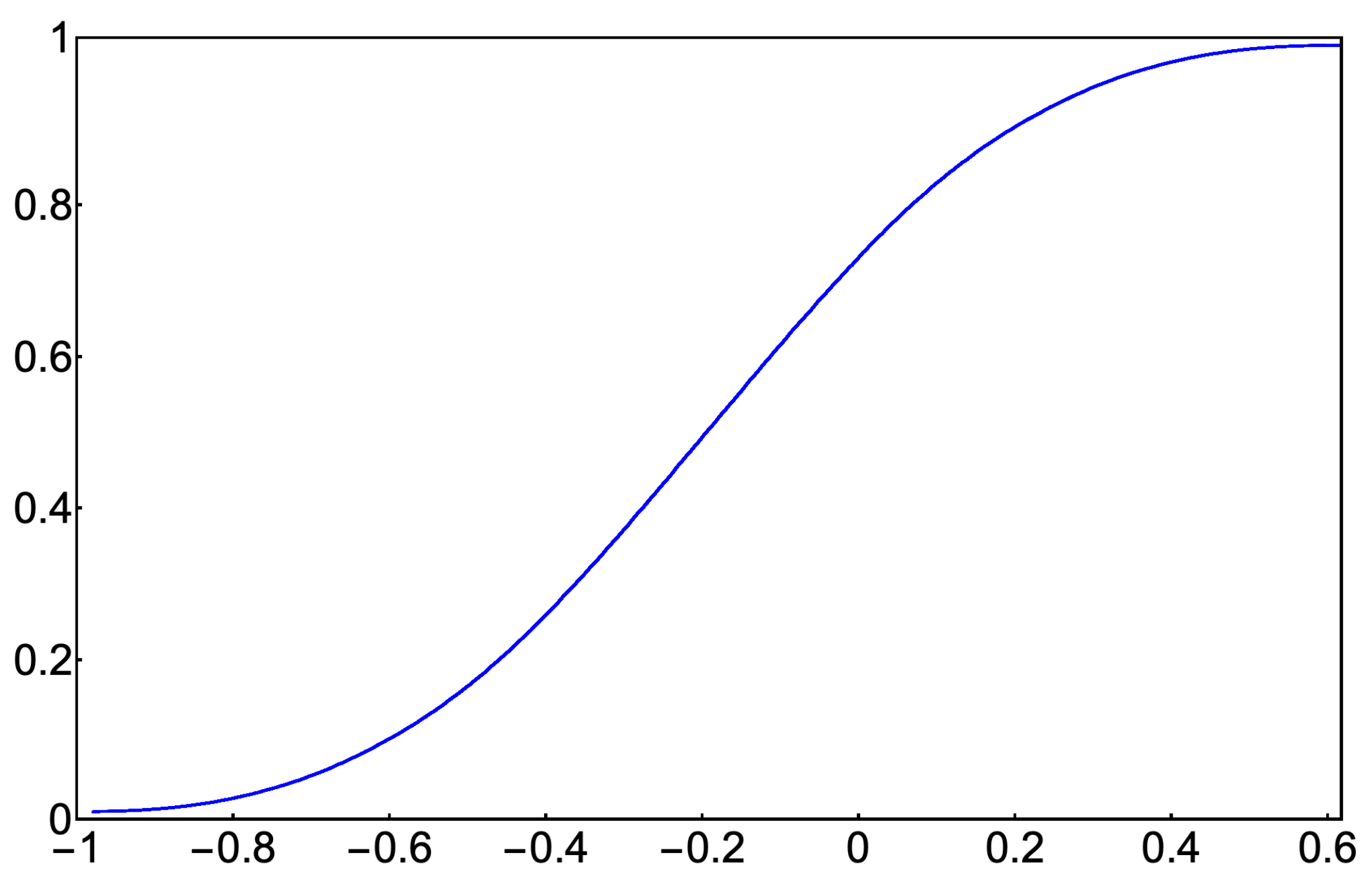}
    \hfill
    \includegraphics[scale=0.14]{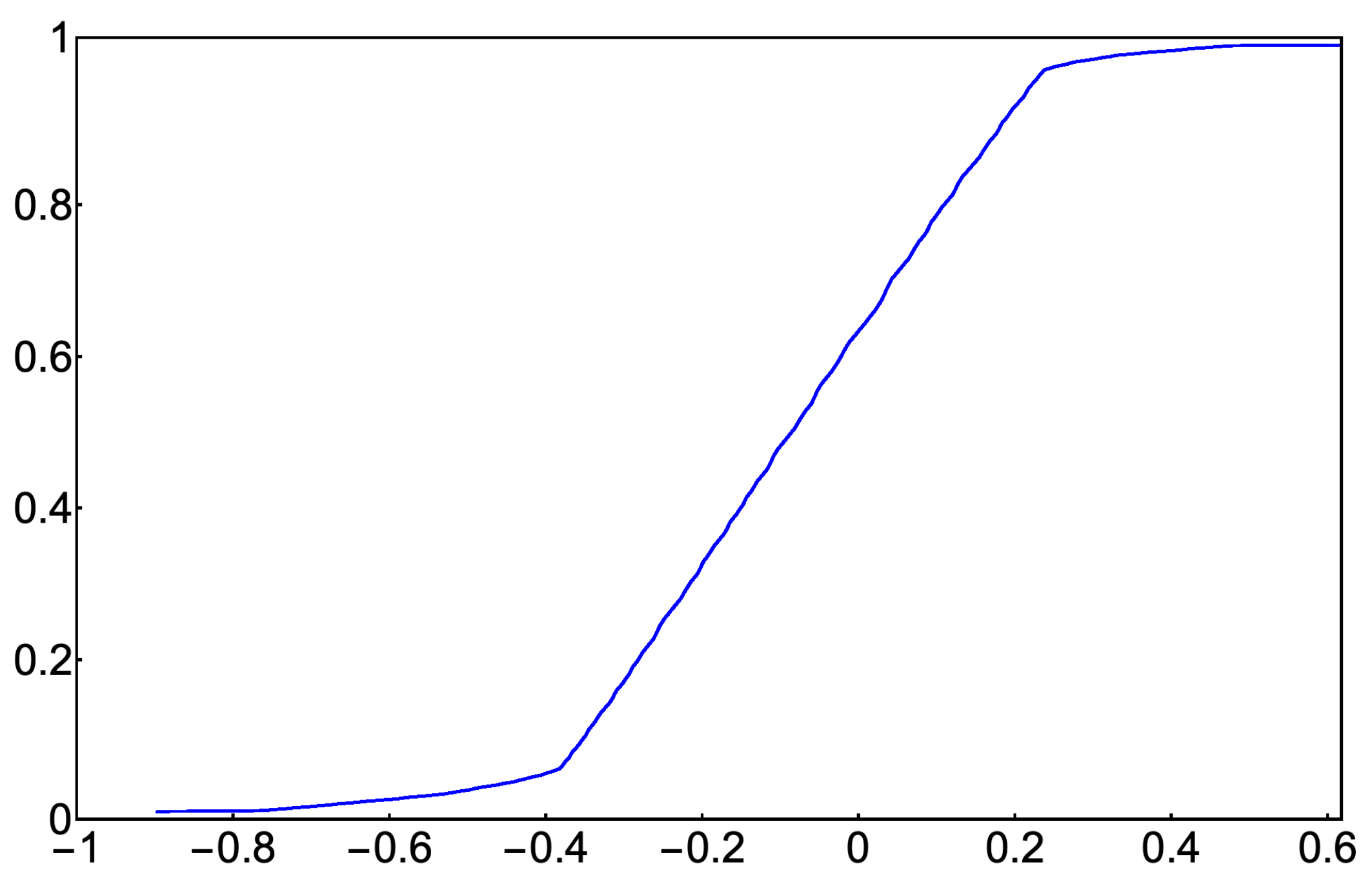}
    \caption{Computer simulations of the distribution functions for the measure $\mu_{a,p}$ for random Fibonacci, for $p=1/16$ (left), $p=1/2$ (middle) and $p=15/16$ (right).}
    \label{fig:distributions}
\end{figure}

Let $\mu_{1,p}$ and $\mu_{2,p}$ be the Rauzy measures on $\mc R_1$ and $\mc R_2$, respectively, induced by the random substitution $\vartheta_{\bm P}$.
By \Cref{PROP:meas-continuity}, the measures $\mu_{1,p}$ and $\mu_{2,p}$ are continuous with respect to $p$.
The mass distribution function of the measure $\mu_{1,p}$ is plotted in Figure \ref{fig:distributions} for several values of $p$.
We note that when $p=0$ or $p=1$, the measure $\mu_{1,p}$ is Lebesgue measure supported on the $a$ tile of the respective marginal.
However, when $0 < p < 1$, $\mu_{1,p}$ has support $\mc R_1 = [-1,\tau^{-1}]$.
\end{example}

A broad class of Rauzy fractals in one dimension can be obtained by considering ``reshuffled'' versions of powers of the Fibonacci substitution.
Namely, by changing the ordering of the letters in realisations of $\theta^m$, for some $m \in \N$.
There are six different deterministic substitutions that have the same substitution matrix as the square of the Fibonacci substitution and the Rauzy fractals associated with four of these substitutions are intervals.
However, the substitutions $\theta_1 \colon 1 \mapsto 112, \, 2 \mapsto 21$ and $\theta_2 \colon 1 \mapsto 211, \, 2 \mapsto 12$, give rise to an enantiomporphic (mirror) pair of Rauzy fractals with a Cantorval structure, the boundary of which have Hausdorff dimension equal to $\log(1+\sqrt{2})/2\log(\tau) \approx 0. 915785$---see \cite{bgm-survey,baake-gorodetski-mazac} for more details.
For each $m \in \N$, locally mixing two or more substitutions with the same substitution matrix as $\theta^m$ gives rise to a compatible, unimodular, irreducible Pisot random substitution.

\begin{example}\label{ex:additional_note_to_cor:new}
    Let $\vartheta_{\bm P}$ be any (non-degenerate) random substitution defined over the set-valued substitution $\vartheta \colon 1 \mapsto \{112\}, \, 2 \mapsto \{12,21\}$, obtained by locally mixing a substitution whose Rauzy fractal is an interval and one whose Rauzy fractal has boundary with positive Hausdorff dimension.
    The Rauzy fractal of the random substitution $\vartheta_{\bm P}$ is an interval, with overlapping subtiles $\mc R_1 = [-\tau^{-1},\tau^{-1}]$ and $\mc R_2 = [0,1]$.
    In fact, it can be verified that all random substitutions with the same substitution matrix as the square of Fibonacci give rise to Rauzy fractals that are intervals, with the exception of the enantiomorphic pair of deterministic substitutions whose intervals have fractal boundary.
    However, it is possible to construct non-deterministic random substitutions that give rise to a Rauzy fractal with fractal boundary.
    This is the case for any random substitution defined over the set-valued substitution $\vartheta \colon 1 \mapsto \{12211211\}, \, 2 \mapsto \{12112,21112\}$, obtained by locally mixing two reshuffled versions of the fourth power of the Fibonacci substitution.

    We highlight that it is possible to delete a realisation from a random substitution without changing its Rauzy fractal. 
    For example, if $\vartheta_{\bm P}'$ is a (non-degenerate) random substitution defined over the set-valued substitution $\vartheta' \colon 1 \mapsto \{112,121\}, \, 2 \mapsto \{12,21\}$, then the Rauzy fractal associated with $\vartheta_{\bm P}'$ is the same as the Rauzy fractal associated with the random substitution $\vartheta_{\bm P}$.

    Note, the length of $\mathcal{R}_\vartheta$ is $1 + \tau^{-1}=\tau$ and the volume of the fundamental domain of $\mathcal{J}$ is $1$. Hence, by \Cref{cor:new}, the recurrent directive sequences built from the marginals of $\vartheta$ give rise to S-adic subshifts with pure-point spectrum. 
    We remark, it was shown in \cite{BMST16} that (under appropriate conditions satisfied by this example) all two-letter Pisot S-adic subshifts have pure-point spectrum. 
    We only highlight that our methods can be used to prove similar statements.
\end{example}

Finally, we present some consequences of our results for the random tribonacci substitution, which we recall gives rise to a Rauzy fractal in two dimensions.

\begin{example}
    Given $p \in (0,1)$, let $\vartheta_{\bm P} = (\vartheta, \bm P)$ be the random tribonacci substitution
    \begin{align*}
        \vartheta_{\bm P} \colon 
        \begin{cases}
            1 \mapsto 
            \begin{cases}
                12 &\text{with probability $p$,}\\
                21 &\text{with probability $1-p$,}
            \end{cases}\\
            2 \mapsto 13 \quad \text{with probability $1$,}\\
            3 \mapsto 1\phantom{3} \quad \text{with probability $1$.}
        \end{cases}
    \end{align*}
    Let $h \colon \mathbb H \mapsto \mathbb H$ denote the action of the substitution matrix on the contracting plane and, for each $i \in \mc A$, let $h_{i} = h(x) + \pi (\psi (i))$.
    By \Cref{THM: r-Rauzy GIFS}, the subtiles $\mc R_1$, $\mc R_2$ and $\mc R_3$ of the Rauzy fractal of $\vartheta_{\bm P}$ are the unique, non-empty, compact sets satisfying the graph-directed iterated function system
    \begin{align*}
        \mc R_1 = h (\mc R_1) \cup h_2 (\mc R_1) \cup h (\mc R_2) \cup h (\mc R_3) \qquad
        \mc R_2 = h (\mc R_1) \cup h_1 (\mc R_1) \qquad
        \mc R_3 = h_1 (\mc R_2) \tp
    \end{align*}
    In an analogous manner to random Fibonacci, the above system can be rewritten as an ordinary iterated function system for $\mc R_1$:
    \begin{align*}
        \mc R_1 = h (\mc R_1) \cup h_2 (\mc R_1) \cup h \circ h (\mc R_1) \cup h \circ h_1 (\mc R_1) \cup h \circ h_1 \circ h (\mc R_1) \cup h \circ h_1 \circ h_1 (\mc R_1) .
    \end{align*}
    Now, let $\mu_{1,p}$, $\mu_{2,p}$ and $\mu_{3,p}$ denote the Rauzy measures associated with the letters $1$, $2$ and $3$, respectively.
    By \Cref{PROP:meas-continuity}, these measures are weakly continuous with respect to $p$.
    Their mass distributions are plotted for a selection of $p$ in Figure \ref{fig:rauzy-measure-trib}. 
    \begin{figure}[t]
    \parbox[m]{0.235\textwidth}{\includegraphics[trim={0 60pt 0 60pt},clip,width=0.235\textwidth]{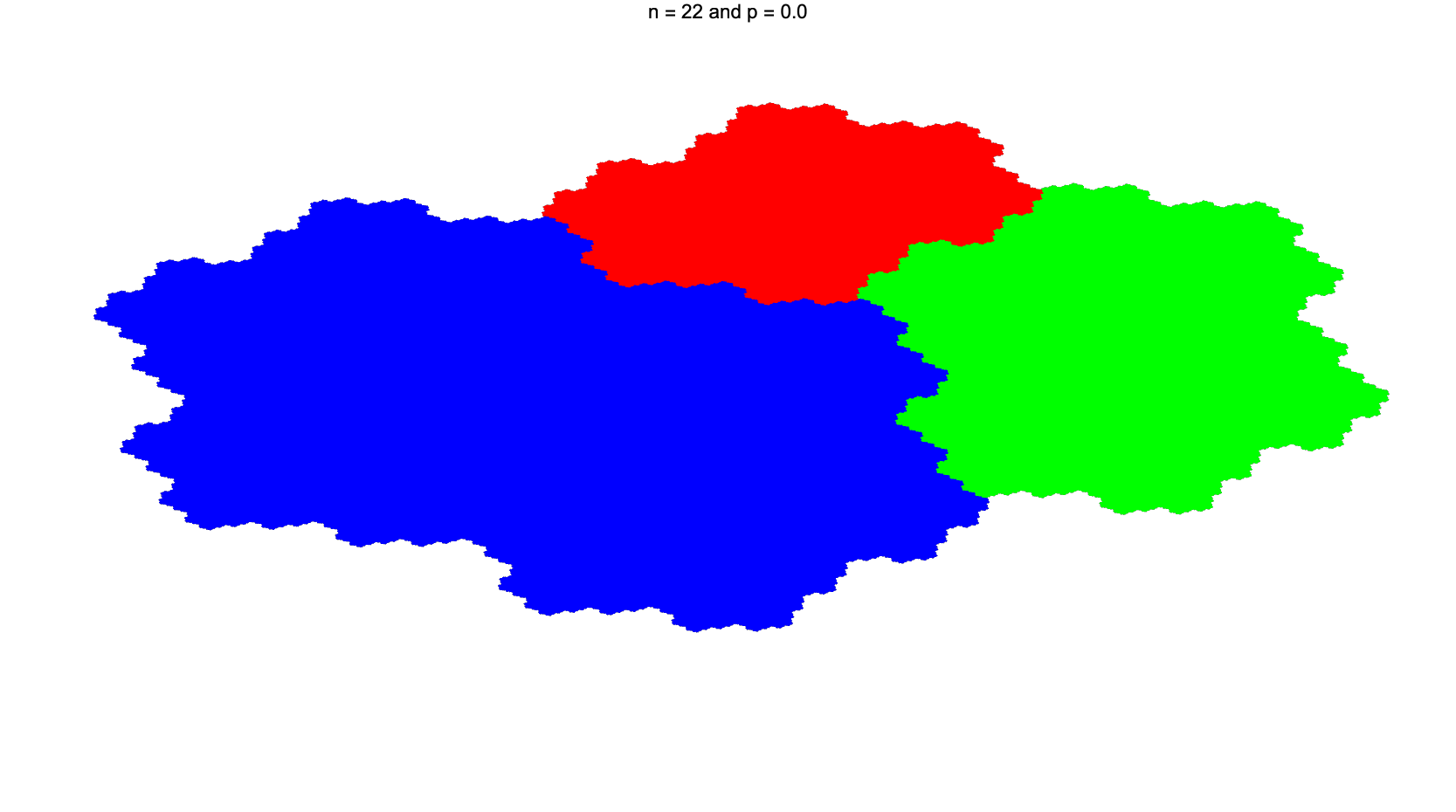
    }}
    \parbox[m]{0.375\textwidth}{
    \includegraphics[trim={0 60pt 0 60pt},clip,width=0.375\textwidth]{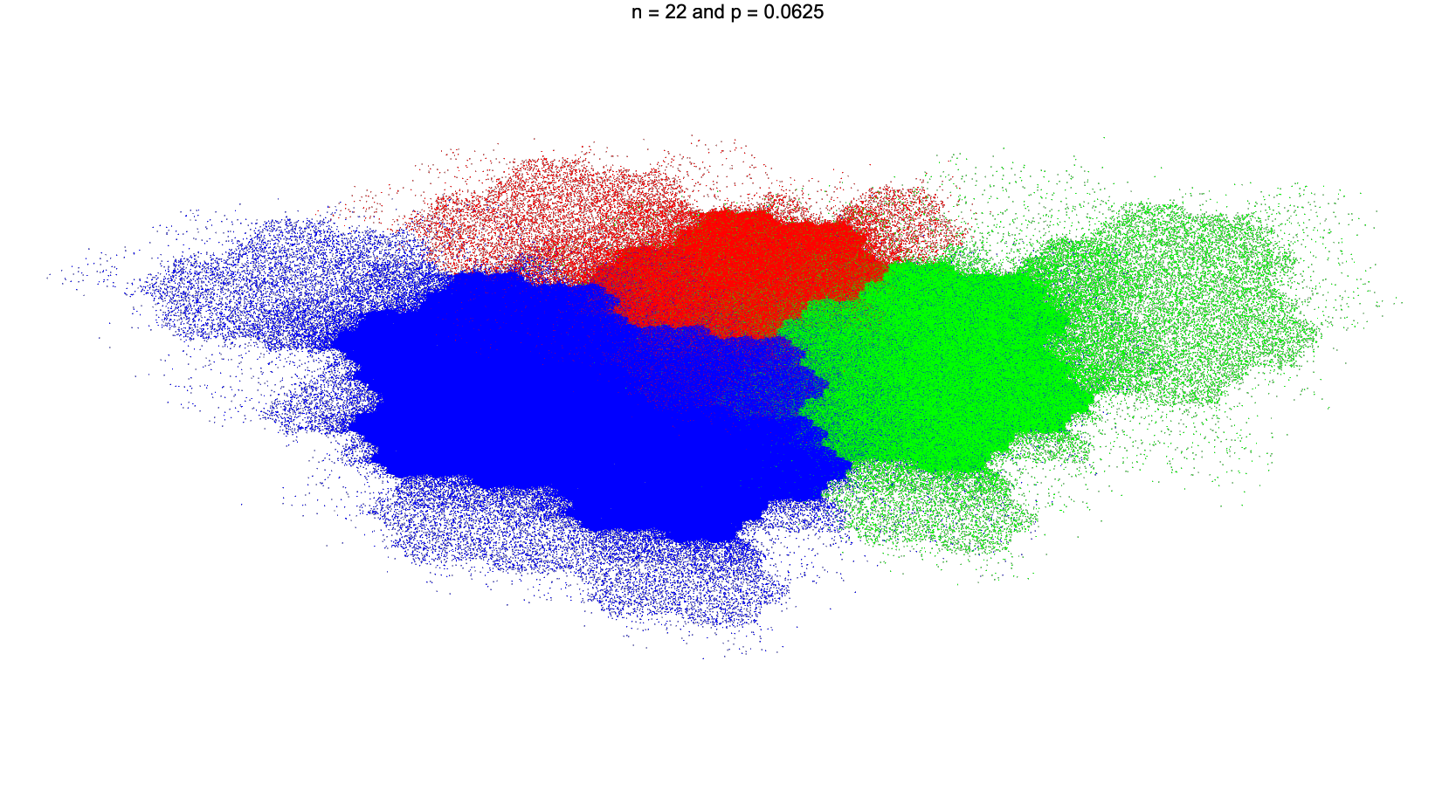}}
    \parbox[m]{0.375\textwidth}{
    \includegraphics[trim={0 60pt 0 60pt},clip,width=0.375\textwidth]{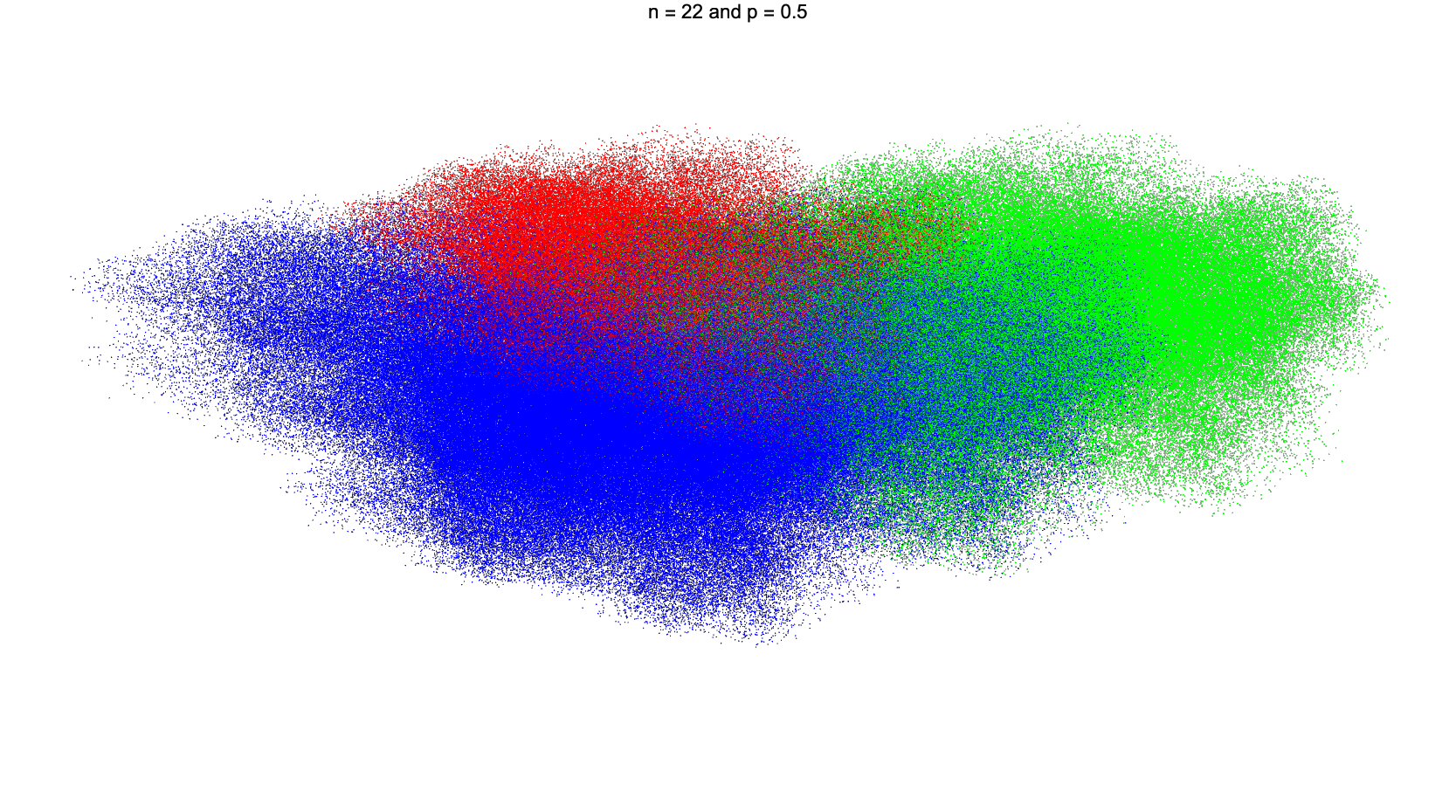}}\\
    \parbox[m]{0.375\textwidth}{
    \includegraphics[trim={0 60pt 0 60pt},clip,width=0.375\textwidth]{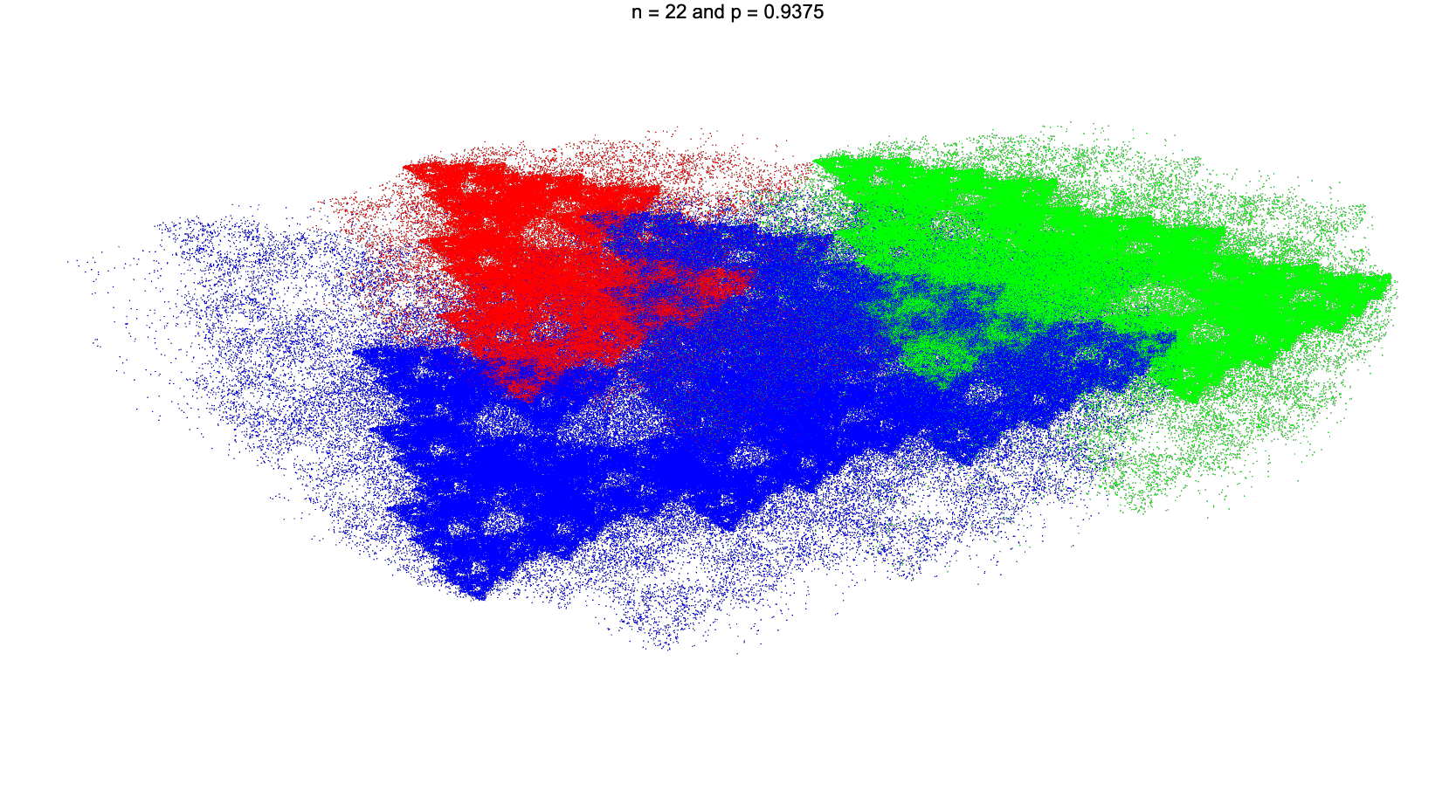}}
    \parbox[m]{0.3\textwidth}{
    \includegraphics[trim={0 60pt 0 60pt},clip,width=0.3\textwidth]{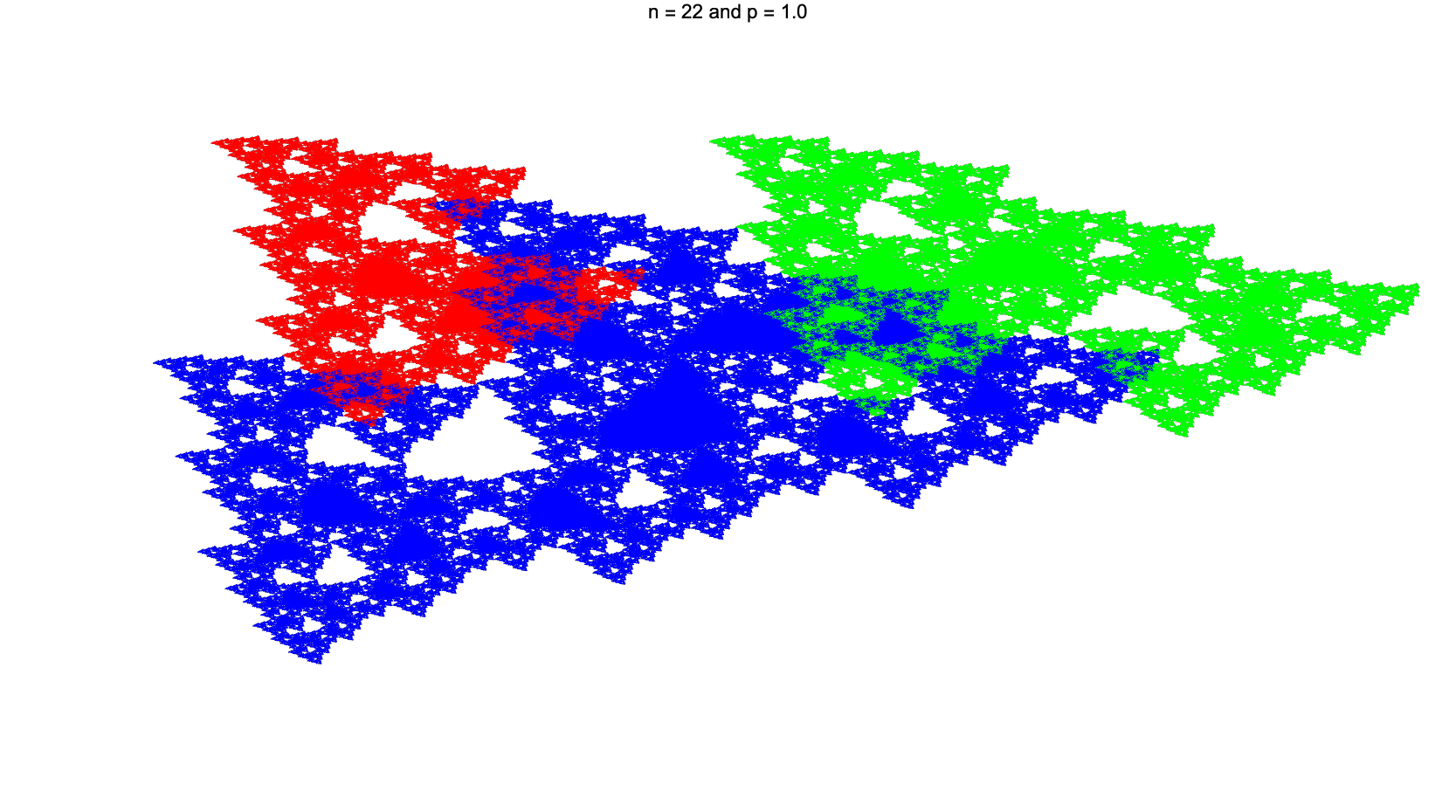}}
    \caption{
    Evolution of Rauzy measures for the random tribonacci substitution as $p$ goes from $1$ to $0$. 
    From top-left to bottom-right, the values of $p$ are $1$, $15/16$, $1/2$, $1/16$ and $0$ respectively.
    Points were generated experimentally (justified by \Cref{PROP:as-convergence-markov}) according to the generating probabilities in such a way that point-densities approximate the mass distributions of each measure.
    }
    \label{fig:rauzy-measure-trib}
\end{figure}
\end{example}

\section{Discussion and open questions}\label{S:open-probs}

Several of our results prompt natural followup questions that we feel warrant further investigation.
As usual, we let $\vartheta$ denote a compatible, irreducible Pisot random substitution throughout this section.

An important feature that distinguishes the Rauzy fractals of random substitutions from their deterministic counterparts is the occurrence of positive-measure overlaps in the self-similarity relation for the subtiles in \eqref{EQ:GIFS-sets}. Nevertheless, it might still be possible to dissect the subtiles further and to find a corresponding GIFS that produces the same Rauzy fractal without such overlaps, possibly tied to a deterministic substitution. 
For general (graph-directed) iterated function systems such a removal of overlaps is a subtle problem, and its success seems to depend on appropriate (weak) separation conditions being fulfilled by the original system. 
For a recent result in this direction that treats the case of equicontractive IFSs with the finite type property, see \cite{bandt-barnsley}. 
For more on the role of weak separation and finite type conditions, we refer the reader to \cite{hare-hare-rutar} and the references therein.
Of course, the additional structure provided by the framework of (random) Pisot substitution systems might make this problem more tractable, motivating the following questions.

\begin{question}
Does the GIFS $\mc G(\vartheta)$ corresponding to a compatible irreducible Pisot random substitution $\vartheta$ satisfy a weak separation condition or some version of a finite type condition? 
\end{question}

\begin{question}\label{QN:mother-sub}
Given $\vartheta$ on $\mc A$, is it possible to find a \emph{deterministic} Pisot substitution $\theta$ on some alphabet $\mc A'$ such that $\mc R_{\vartheta} = \mc R_{\theta}$? Further, can $\theta$ be chosen such that for each $a \in \mc A$ there exists a $\mc B \subset \mc A'$ with $\mc R_{\vartheta,a} = \cup_{b \in \mc B} \mc R_{\theta,b}$?
\end{question}

For some explicit examples, solutions to the above question can be found. For example, the Rauzy fractal of the random Fibonacci substitution $\vartheta$ can be shown to coincide with that of the (reducible) Pisot substitution $1\mapsto 232,\ 2\mapsto 1,\ 3\mapsto 2$, where each letter corresponds to an element of a particular partition of $\mc R_\vartheta$.
The same holds for random tribonacci and other related three-letter examples, however we highlight that in these cases the cardinality of the alphabet over which $\theta$ is defined is considerably larger than the alphabet on which $\vartheta$ is defined.
We thank Paul Mercat for providing computer-assisted calculations of the substitution $\theta$ in these cases.

A positive solution to Question \ref{QN:mother-sub} would provide a method for calculating the dimension of the boundary of $\mc{R}_\vartheta$, as there are well-known methods that apply to Rauzy fractals of deterministic substitutions --- see \cite{Siegel-Thuswaldner} and the references therein.
Likewise, topological properties of $\mc R_\vartheta$ such as connectedness and simple-connectedness can be determined. However, the question remains if these properties can be determined in the general case, and if more efficient methods exist that do not require calculating auxiliary deterministic substitutions on large alphabets.

\begin{question}
    What is the dimension of the boundary of $\mc R_{\vartheta}$?
\end{question}
\begin{question}
    What are the topological properties of $\mc R_{\vartheta}$?
\end{question}

Random substitutions provide a means of interpolating between two (or more) deterministic substitutions sharing the same abelianisation.
Thus, it is possible that the theory we have developed will provide a new approach to open problems concerning Pisot substitutions, such as the Pisot substitution conjecture.
In particular, the extra freedom afforded by being able to locally mix deterministic substitutions and interpolate between them could allow one to transfer known properties of one substitution (such as a substitution whose spectrum is known to be pure point) to some other target substitution whose properties remain unknown. 

\begin{question}
    Given a (deterministic) irreducible Pisot substitution $\theta$, is there a power $\theta^k$ and a substitution $\theta'$ with the same substitution matrix such that the following holds: the local mixture of $\theta^k$ and $\theta'$ yields a random substitution $\vartheta$ with $\Leb(\mc R_\vartheta)$ being smaller than two times the measure of the fundamental domain in $\mc J$?
\end{question}

The intuition behind this question is that if $\theta'$ is known to be pure point and $\theta^k$ is reasonably close combinatorially, there is hope that the Rauzy fractal does not grow too much when both substitutions are mixed, such as was the case in \Cref{ex:additional_note_to_cor:new}.
In light of \Cref{cor:new}, an affirmative answer to this question would in particular imply the Pisot substitution conjecture, and even a partial answer could show pure point spectrum for some appropriate class of substitutions.
A first step in this direction would be to gain a better understanding of how the Rauzy fractals of the substitutions in a set $\Theta$ are related to the Rauzy fractal of the enveloping random substitution $\vartheta$.

\begin{question}
Let $\vartheta$ be the enveloping random substitution of a compatible collection $\Theta$ of irreducible Pisot substitutions. Given $\theta \in \Theta$, under which conditions is $\mc R_\theta$ contained in the interior of $\mc R_\vartheta$? Is there a bound for $\Leb(\mc R_\vartheta)$ in terms of $\{\Leb(\mc R_\theta)\}_{\theta \in \Theta}$ and other data relating to the collection $\Theta$?
\end{question}

We have already seen that in the case of the random Fibonacci substitution $\vartheta$ discussed in \Cref{EX:rand-fib}, the Rauzy fractals of the marginals of $\vartheta$ are contained in the interior of $\mc R_\vartheta$. Conversely, for a variant of the random tribonacci substitution, given by $\vartheta \colon 1 \mapsto \{12\}, 2 \mapsto \{ 13,31 \}, 3 \mapsto \{1\}$, both marginals seem to produce Rauzy fractals that intersect the boundary of $\mc R_\vartheta$ (see Figure~\ref{fig:position-of-marginals}). 

\begin{figure}
    \centering
    \includegraphics[width=0.4\linewidth]{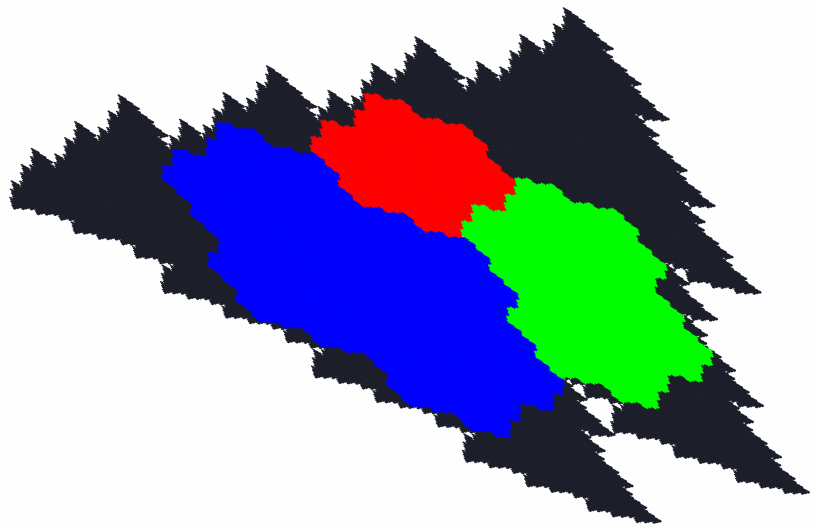}
    \qquad
    \includegraphics[width=0.4\linewidth]{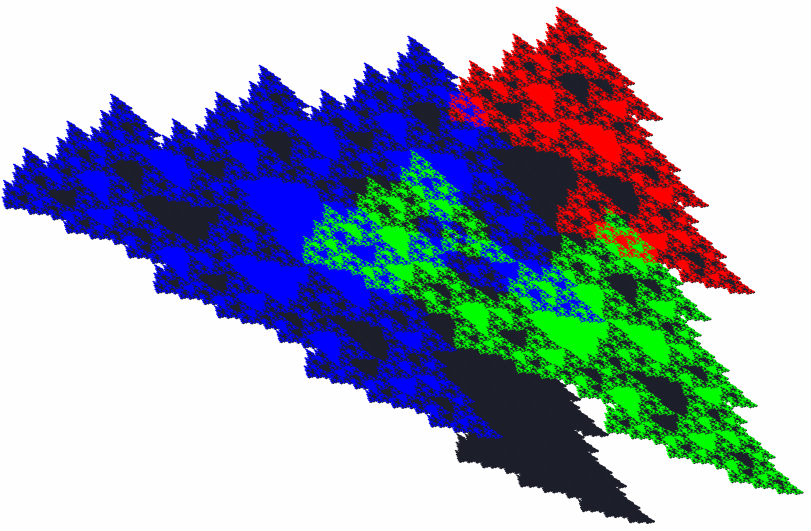}
    \vspace{0.5em}
    \caption{Rauzy fractal of the random substitution $\vartheta \colon 1 \mapsto \{12\}, 2 \mapsto \{13,31\}, 3 \mapsto \{1\}$ and positions of the Rauzy fractals of the marginals.}
    \label{fig:position-of-marginals}
\end{figure}

Even though $(X_\vartheta,S)$ cannot be expected to have purely pure point spectrum, it appears natural to ask for how much of the pure point component the Rauzy fractal contributes in the random setting. In \Cref{COR:generic-factor}, we established that, up to a projection to the torus $\mathbb{H}/\mc J$, the Rauzy fractal $\mc R_\vartheta$, equipped with a rotation $T$, provides a natural generic factor of $(X_\vartheta,S)$. The maximality of this factor has been established for the special case of the random Fibonacci substitution in \cite{baake-spindeler-strungaru}; compare also the discussion in \Cref{REM:MEGF}.

\begin{question}
Given a compatible, irreducible Pisot random substitution $\vartheta$, is $(\mathbb{H}/\mc J,T)$ always the MEGF of $(X_\vartheta,S)$? Under which conditions are there (non-constant) continuous eigenfunctions of $(X_{\vartheta},S)$?
\end{question}

A spectral measure that is particularly relevant in the context of physical applications is the diffraction measure, due to its property of being observable in scattering experiments. In fact it has been established under quite general conditions that pure spectrum in the sense of dynamical spectrum and diffraction spectrum are equivalent \cite{BLvE}. In the context of deterministic substitutions it is also known that diffraction measures often produce a measure of maximal spectral type, thus capturing all spectral information \cite{queffelec}. Whether this holds true for random substitutions as well remains to be determined.

The first calculation of the diffraction measure for a random substitution (random Fibonacci) was achieved by Godr\`{e}che and Luck in \cite{godreche-luck}. A key insight in this calculation was that the diffraction measure can be split appropriately into an expectation part and a variance part, which corresponds to a splitting into a pure point component and an absolutely continuous component. The observation that the variance part produces an absolutely continuous measure holds under fairly general conditions, and in particular for every compatible primitive random substitution \cite[Thm.~4.3.4]{gohlke_diss}. The problem of determining the spectral types that arise from the expectation part is in general more subtle, although criteria for it being pure point exist in the setting of constant length random substitutions \cite[Thm.~4.3.24]{gohlke_diss}, generalising the work of Dekking~\cite{dekking}. 

For a family of random substitutions that includes the random Fibonacci substitution, it was shown in \cite{baake-spindeler-strungaru} that the expectation part coincides with the diffraction of a weighted model set with continuous density, proving its pure point nature. Here, the weights are implicitly given by the Rauzy measure and can be used to yield an explicit expression for the pure point component.
This approach seems promising for more general random substitutions.

\begin{question}
Given a compatible, irreducible Pisot random substitution, is the expectation part in the Godr\`{e}che--Luck splitting always the diffraction of a weighted model set, with weights provided by the Rauzy measure?
\end{question}

An essential step for concluding that the expectation part is pure point in \cite{baake-spindeler-strungaru} was to prove that the weights on the Rauzy fractal are given by a continuous function. In our setting, this translates to the following problem.

\begin{question}
    How smooth is the density function of a Rauzy measure? Which general criteria ensure that this density function is continuous?
\end{question}

Establishing continuity of the density would thus prove pure pointedness of the expectation part and absence of a singular continuous component in the complete diffraction measure. Since a non-trivial absolutely continuous component always appears under some non-degeneracy condition, this can be seen as a generalisation of the Pisot substitution conjecture to the random setting. A different variant of generalising the Pisot substitution conjecture can be formulated without reference to diffraction.

\begin{question}
    Given a compatible, irreducible Pisot random substitution $\vartheta$, is every minimal subsystem of $(X_\vartheta,S)$ a pure point system?
\end{question}

This is motivated by the fact that all marginals of $\vartheta^n$ with $n \in \N$ produce minimal substitutive subsystems (and corresponding S-adic subsystems).
If the Pisot substitution conjecture and S-adic Pisot conjecture hold, then all of these subsystems are pure point.
It is therefore natural to ask whether the same holds for \emph{all} minimal subsystems of $(X_{\vartheta},S)$.

\section*{Acknowledgements}

The authors thank Pierre Arnoux for making them aware of connections between Rauzy measures and S-adic systems. 
They are also grateful to Michael Baake, Jan Maz\'a\v{c} and Alex Rutar for helpful discussions and to Dan Herring for his assistance in generating the images in \Cref{fig:position-of-marginals}.
In addition, the authors would like to thank the anonymous referee for numerous thought-provoking comments that significantly improved the presentation of this work.

PG acknowledges support from DFG grant 509427705. 
AM thanks EPSRC DTP and the University of Birmingham for financial support, and Collaborative Research Center 1283 of Bielefeld University for funding research visits to Bielefeld, where part of this work was undertaken. 
AM and TS would also like to thank the Birmingham--Leiden collaboration fund of the University of Birmingham and EPSRC grant EP/Y023358/1, which also supported this work.

\bibliographystyle{abbrv}
\bibliography{ref}

\end{document}